\documentclass[reqno,10pt]{amsart}
\usepackage{amsfonts, latexsym, amssymb, amsthm, amsmath, mathrsfs, esint, color, comment, fullpage}
\usepackage[left=2.8cm,right=2.8cm,top=2.8cm,bottom=2.8cm]{geometry}
\usepackage{amscd}
\usepackage{enumitem} 
\usepackage{mathtools}
\usepackage{graphicx}
\graphicspath{  }
\usepackage{relsize}
\usepackage{url} 
\usepackage[colorlinks=true, pdfstartview=FitV, linkcolor=blue, 
            citecolor=blue, urlcolor=blue]{hyperref}
\usepackage{bbm}
\usepackage{cleveref}
\numberwithin{equation}{section}

\newtheorem{theorem}{Theorem}[section]
\newtheorem{lemma}[theorem]{Lemma}
\newtheorem{proposition}[theorem]{Proposition}
\newtheorem{corollary}[theorem]{Corollary}
\newtheorem{assumption}{Assumption}
\newtheorem{definition}[theorem]{Definition}
\newtheorem{remark}[theorem]{Remark}
\theoremstyle{remark}

\newcommand{\EEE}{\color{black}}

\newcommand{\MMM}{\color{black}}
\newcommand{\BBB}{\color{black}}

\newcommand{\ep}{\varepsilon}
\newcommand{\mb}{\mathcal{M}_b}
\newcommand{\wk}{\rightharpoonup}

\newcommand{\BLACK}{\color{black}}
\newcommand{\BLUE}{\color{black}}

\renewcommand{\div}{{\rm div}}

\DeclareMathOperator{\rot}{curl}

\DeclareMathOperator{\dev}{dev}
\DeclareMathOperator{\diag}{diag}
\DeclareMathOperator{\dist}{dist}

\DeclareMathOperator{\skw}{skew}
\DeclareMathOperator{\supp}{supp}
\DeclareMathOperator{\sym}{sym}
\DeclareMathOperator{\tr}{tr}

\newcommand{\calA}{\mathcal{A}}
\newcommand{\calB}{\mathcal{B}}
\newcommand{\C}{\mathbb{C}}
\newcommand{\calC}{\mathcal{C}}
\newcommand{\calD}{\mathcal{D}}
\newcommand{\calE}{\mathcal{E}}
\newcommand{\calH}{\mathcal{H}}
\newcommand{\calI}{\mathcal{I}}
\newcommand{\calJ}{\mathcal{J}}
\newcommand{\calK}{\mathcal{K}}
\newcommand{\calL}{\mathcal{L}}
\newcommand{\M}{\mathbb{M}}

\newcommand{\Mb}{\mathcal{M}_b}
\newcommand{\N}{\mathbb{N}}

\newcommand{\calQ}{\mathcal{Q}}
\newcommand{\R}{\mathbb{R}}

\newcommand{\calS}{\mathcal{S}}
\newcommand{\calT}{\mathcal{T}}
\newcommand{\U}{U}

\newcommand{\calV}{\mathcal{V}}

\newcommand{\calX}{\mathcal{X}}
\newcommand{\calY}{\mathcal{Y}}
\newcommand{\Z}{\mathbb{Z}}
\newcommand{\zz}{x}
\newcommand{\calZ}{\mathcal{Z}}

\newcommand{\Dir}{D}

\newcommand{\torustridim}{\calY^{(3)}}

\newcommand{\DiffOpA}{\mathcal{A}}
\newcommand{\LinOpA}{A}
\newcommand{\DiffOp}{\mathcal{A}_{x_2}}
\newcommand{\CorrSpace}[3]{\calX^{{#1}}({#2})}
\newcommand{\calXgamma}[1]{\calX_{\gamma}({#1})}

\newcommand{\BDgamma}{BD_\gamma(I \times \calY)}
\newcommand{\uKL}{E\bar{u}(x') - x_3 D^2u_3(x')}

\newcommand{\eps}{\varepsilon}
\newcommand{\epsh}{{\varepsilon_h}}

\newcommand{\strong}{\to}
\newcommand{\weak}{\,\xrightharpoonup{}\,}
\newcommand{\weakstar}{\xrightharpoonup{*}}
\newcommand{\strongtwoscale}{\xrightarrow{\,2\,}}
\newcommand{\weaktwoscale}{\xrightharpoonup{2}}

\newcommand{\weakstartwoscale}{\xrightharpoonup{2-*}}

\newcommand{\ext}[1]{\widetilde{#1}}
\newcommand{\genprod}{\stackrel{\text{gen.}}{\otimes}}
\newcommand{\mres}{\lfloor}
\newcommand{\charfun}[1]{\mathbbm{1}_{#1}}

\newcommand{\proj}{proj}
\newcommand{\closure}[1]{\overline{#1}}

\newcommand{\tangential}{_\nu^\perp}
\newcommand{\normal}{_\nu}
\newcommand{\expfun}[1]{\exp\left({#1}\right)}

\usepackage[
maxbibnames=99,
    backend=biber,
    citestyle=numeric, doi=false,isbn=false,url=false,eprint=false, giveninits=true,
]{biblatex}
\title [Periodic homogenization of elastoplastic plates]
{Effective quasistatic evolution models for perfectly plastic plates with periodic microstructure}
\author[M. Bu\v{z}an\v{c}i\'{c}] {Marin Bu\v{z}an\v{c}i\'{c}} 
\address[M. Bu\v{z}an\v{c}i\'{c}]{Faculty of Chemical Engineering and Technology, University of Zagreb, Trg Marka Maruli\'{c}a 19,
10000 Zagreb, Croatia}
\email{buzancic@fkit.hr}
\author[E. Davoli] {Elisa Davoli} 
\address[E. Davoli]{Institute of Analysis and Scientific Computing, TU Wien,
Wiedner Hauptstrasse 8-10, A-1040 Vienna, Austria}
\email{elisa.davoli@tuwien.ac.at}
\author[I. Vel\v{c}i\'{c}] {Igor Vel\v{c}i\'{c}} 
\address[I. Vel\v{c}i\'{c}]{Faculty of Electrical Engineering and Computing, University of Zagreb, Unska 3, 10000 Zagreb, Croatia}
\email{igor.velcic@fer.hr}

\begin{document}
\maketitle
\vspace{-\baselineskip}
\begin{abstract}
An effective model is identified for thin perfectly plastic plates whose microstructure consists of the periodic assembling of two elastoplastic phases, as the periodicity parameter converges to zero. Assuming that the thickness of the plates and the periodicity of the microstructure are comparably small, a limiting description is obtained in the quasistatic regime via simultaneous homogenization and dimension reduction by means of evolutionary $\Gamma$-convergence, two-scale convergence, and periodic unfolding.
\end{abstract}
\section{Introduction}
With this paper, we begin the task of identifying reduced models for thin composite elastoplastic plates with periodic microstructure. We focus here on the case in which the thickness $h$ of the plates and their microstructure width $\ep_h$ are asymptotically comparable, namely, we assume the existence of the limit
$$\lim_{h\to 0}\frac{h}{\ep_h}=:\gamma\in (0,+\infty).$$ This corresponds, roughly speaking, to the situation in which homogenization and dimension reduction occur somewhat simultaneously and a strong interaction between vanishing thickness and periodicity comes into play.
Different scalings of $\gamma$ (i.e., $\gamma=0$ and $\gamma=+\infty$) will be   the subject of a forthcoming companion paper. 

Finding lower dimensional models for thin three-dimensional structures is a classical task in the Mathematics of Continuum Mechanics. A rigorous identification of a reduced model for perfectly plastic plates in the quasistatic regime has been undertaken in \cite{Davoli.Mora.2013}. An additional regularity result for the associated stress has been established in \cite{Davoli.Mora2015}. The case of dynamic perfect plasticity is the subject of \cite{Maggiani.Mora2016, Gidoni.Maggiani.Scala2018}, whereas the setting of shallow shells has been tackled in \cite{Maggiani.Mora2017}. A parallel analysis in the presence of hardening has been performed in \cite{liero2011evolutionary,Liero.Roche2012} 
We further mention the two works \cite{DavoliCOCV,DavoliM3AS} in the purview of finite plasticity.

The study of composite elastoplastic materials is a challenging endeavour. In the small strain regime, limit plasticity equations have been identified in \cite{Schweizer.Veneroni2015, Heida.Schweizer2016, Heida.Schweizer2018}
both in the periodic and in the aperiodic and stochastic settings.
The Fleck and Willis model
is the subject of \cite{Francfort.Giacomini.Musesti,Giacomini.Musesti}, whereas gradient plasticity has been studied in \cite{Hanke}. For completeness, we also mention \cite{Cristowiak.Kreisbeck, Cristowiak.Kreisbeck2, Davoli.Ferreira.Kreisbeck, Davoli.Kreisbeck} for an analysis of 
large-strain
stratified composites in crystal plasticity and \cite{Davoli.Gavioli.Pagliari} for a static result in the finitely plastic setting. The characterization of inhomogeneous perfectly plastic materials and a subsequent periodic homogenization have been undertaken in \cite{Francfort.Giacomini.2012, Francfort.Giacomini.2014}.

The novelty of the present contribution consists in the fact that we combine both dimension reduction and periodic homogenization in order to deduce a limiting description, as the two smallness scales (thickness and width of the microstructure) converge to zero, for perfectly plastic thin plates. 

To complete our literature overview, we briefly recall the main mathematical contributions on simultaneous homogenization and dimension reduction.
In  \cite{cailleriethin}, the author derives a limiting plate model starting from 3d linearized elasticity, while assuming the material to be isotropic and the microstructure to be periodic. In \cite{damlamianvogelius}, the case of linear elastic plates with possible aperiodic microstructure is tackled by relying on material (planar) symmetries of the elasticity tensor, and by introducing the notion of $H$-convergence adapted to dimension reduction. 
In \cite{bukal2017simultaneous} an effective plate model is identified in the general case (without further periodicity or material-symmetries assumptions) by means of $\Gamma$-convergence (the analysis presented there also covers some non-linear models).  
We also mention the book \cite{panasenkobook} where linear rod and plate models are obtained by simultaneous homogenization and dimension reduction, and appropriate estimates are also provided, as well as the recent work \cite{BCVZ} on high-contrast elastic plates.
Different non-linear elastic plate models  obtained by $\Gamma$-convergence are discussed in \cite{cherdantsev2015,neukammvelcic2013,hornung2014derivation,bufford2015multiscale,velcic2015}.

 To the Authors knowledge, this manuscript represents instead the first work on effective theories for plates undergoing inelastic deformations.
 
 We conclude this introduction by briefly presenting our results.
First, after establishing a general disintegration result for measures in the image of suitable first-order differential operators, cf. Proposition \ref{BV^A main property}, and relying on an auxiliary result related \MMM to De Rham cohomology, cf. Proposition \ref{auxiliary result - regime gamma} \BBB, in Theorem \ref{two-scale weak limit of scaled strains}, we identify two-scale limits of rescaled strains. We point out that the intermediate results in Proposition \ref{BV^A main property} are of independent interest and apply to a more general setting than that investigated in this contribution. We have chosen to pursue this avenue for those tools will be instrumental also for the analysis of further regimes of plastic thin-plates homogenization.
We emphasize that for identifying  two-scale limits of rescaled strains we could not rely on the results obtained in the context of elasticity (see, e.g. \cite{bukal2017simultaneous}), since these results relied 
on Korn inequalities which are not available in the plastic setting, hence a new approach needed to be developed.

For a given boundary datum $w$, the limiting model that we identify is finite on triples $(u,E,P) \in \calA^{hom}_{\gamma}(w)$, where the latter denotes the set of limits of plastic triples given by displacements, elastic, and plastic strains in the sense of two-scale convergence for measures, cf. Definition \ref{def:2-scale-meas}. We refer to Definition \ref{definition A^hom_gamma} and to Subsection \ref{subs:dis} for the precise definition and main disintegration properties of the class $\calA^{hom}_{\gamma}(w)$.
On such triples, the effective elastic energy and dissipation potential are homogenized densities depending only on the limiting two-scale elastic and plastic strain, respectively.
 \EEE Our analysis stems from adapting the approach of \cite{Francfort.Giacomini.2014} to the setting of dimension reduction problems for composite plates. This is, however, a non-trivial task: a first hurdle consists in the already mentioned  compactness result for rescaled strains, see Section \ref{scaled sym gradients section}. Further difficulties originate from the fact that the limit problem is of fourth order, see Section \ref{statics}. Further, analogously to \cite{Davoli.Mora.2013}, the limiting description is truly three-dimensional. We refer to \cite[Section 5]{Davoli.Mora2015}
 for a discussion of this issue and an example.
 \BBB
Our effective model is completely characterized in Subsection \ref{Lower semicontinuity of energy functionals}. After introducing a suitable notion of stress-strain duality, in Theorem \ref{two-scale Hill's principle - regime gamma} we prove a two-scale limiting Hill's principle.  
The lower semicontinuity of the effective energy and dissipation functionals is proven in Theorem \ref{lsc-gamma finite} Key tools are an adaptation of unfolding techniques for dimension reduction (see Proposition \ref{associated unfolding measure}), as well as a technical rank-one decomposition characterization (see Lemma \ref{rank-1 lemma}).
Finally, with Theorem \ref{main result} we prove the main result of this contribution, showing via evolutionary \BLUE$\Gamma$\BLACK-convergence, cf. \cite{mielke.roubicek.stefanelli} the convergence of three-dimensional inhomogeneous quasistatic evolutions to energetic solutions for our two-scale reduced model.

 The paper is organized as follows.
 Section \ref{prel} contains some preliminary results on two-scale convergence, disintegration of Radon measures, $BD$ and $BH$ functions, as well as \MMM some auxiliary claims about stress tensors. \BBB In Section \ref{setting} we specify the setting of the problem and the main assumptions. We additionally recall the existence results for quasistatic evolution for general multi-phase materials. The characterization of limiting triples in the sense of two-scale convergence for Radon measures is the focus of Section \ref{compactness}. The effective stress-strain duality is analyzed in Section \ref{statics}, whereas the convergence of quasistatic evolutions is proven in Section \ref{dynamics}.

\section{Preliminaries}
\label{prel}
\EEE In this section we specify our notation and collect a few preliminary results.\BBB
\subsection{Notation}
We will write any point $x \in \R^3$ as 
a pair $(x',x_3)$, with $x' \in \R^2$ and $x_3 \in \R$,
and we will use the notation $\nabla_{x'}$ to denote the gradient with respect to $x'$.
We denote by $y \in \calY$ the points on a flat 2-dimensional torus.
\MMM We denote by $I$ the open interval $I := \left(-\frac{1}{2}, \frac{1}{2}\right)$. \BBB
In what follows we will also adopt the following notation for scaled gradients and symmetrized scaled gradients:
\begin{gather} \label{defsymmscgrad}
    \nabla_h v := \Big[\; \nabla_{x'} v \;\Big|\Big.\; \tfrac{1}{h}\,\partial_{x_3} v \;\Big], \quad E_h v := \sym \nabla_h v,\\
    \nonumber \widetilde{\nabla}_{\gamma} v := \Big[\; \nabla_y v \;\Big|\Big.\; \tfrac{1}{\gamma}\partial_{x_3} v \;\Big], \quad \widetilde{E}_{\gamma} v := \sym \widetilde{\nabla}_{\gamma} v,
\end{gather}
\MMM where $h,\gamma>0$ and $v$ \BLUE is a \MMM function on the appropriate domain. \BBB
The scaled divergence operators $\div_h$ and $\widetilde{\div}_{\gamma}$ are \MMM defined in the following way: 
$$ \div_h v:=\partial_{x_1}v_1+\partial_{x_2}v_2+\frac{1}{h} \partial_{x_3} v_3, \quad \widetilde{\div}_{\gamma}v:=\partial_{y_1}v_1+\partial_{y_2}v_2+\frac{1}{\gamma} \partial_{x_3} v_3. $$
\EEE Analogously, \MMM we define the 
operators $\rot$ and $\widetilde{\rot}_{\gamma}$, for functions taking values in $\R^3$. 
\EEE Note that \MMM the operators $\widetilde{\nabla}_{\gamma}$, $\widetilde{\div}_{\gamma}$, $\widetilde{\rot}_{\gamma}$ act on functions that have as (part of) their domain $I \times \calY$ (\EEE with a slight abuse of notation  we write this domain with \MMM $I$  on the first place, despite the fact that the associated differential operators are defined as above). 
\BBB

If $a, b \in \R^N$, we write $a \cdot b$ for the Euclidean scalar product, and we denote by $|a| := \sqrt{a \cdot a}$ the Euclidean norm. 
We write $\M^{N \times N}$ for the set of real $N \times N$ matrices. 
If $A, B \in \M^{N \times N}$, we use the Frobenius scalar product $A : B := \sum_{i,j}A_{ij}\,B_{ij}$ and the associated norm $|A| := \sqrt{A:A}$.
We denote by $\M^{N \times N}_{\sym}$ the space of real symmetric $N \times N$ matrices, and by $\M^{N \times N}_{\dev}$ the set of real deviatoric matrices, respectively, i.e. the subset of $\M^{N \times N}_{\sym}$ given by matrices having null trace. 
For every matrix $A \in \M^{N \times N}$ we denote its trace by ${\rm tr}{A}$, and its deviatoric part by $A_{\dev}$ will be given by
\[
    A_{\dev} = A - \frac{1}{N}{\rm tr}{A}.
\]
The {\em symmetrized tensor product} $a \odot b$ of two vector $a, b \in \R^N$ is the symmetric matrix with entries $(a \odot b)_{ij} := \frac{a_i b_j + a_j b_i}{2}$. 
Note that ${\rm tr}{\left(a \odot b\right)} = a \cdot b$, and that $|a \odot b|^2 = \frac{1}{2}|a|^2|b|^2 + \frac{1}{2}(a \cdot b)^2$, so that
\begin{equation*}
    \frac{1}{\sqrt{2}}|a||b| \leq |a \odot b| \leq |a||b|.
\end{equation*}
Given a vector $v \in \R^3$, we will use the notation $v^{\prime}$ to denote the vector
\begin{equation*}
    v^{\prime} := \begin{pmatrix} v_1 \\ v_2 \end{pmatrix}.
\end{equation*}
Analogously, given a matrix $A \in \M^{3 \times 3}$, we will denote by $A^{\prime\prime}$ the minor
\begin{equation*}
    A^{\prime\prime} := \begin{pmatrix} A_{11} & A_{12} \\ A_{21} & A_{22} \end{pmatrix}.
\end{equation*}

The Lebesgue measure in $\R^N$ and the $(N-1)$-dimensional Hausdorff measure are denoted by $\calL^N$ and $\calH^{N-1}$, respectively. \MMM For $U \subset \R^N$, $\closure{U}$ denotes its closure. \BBB
Given an open subset $U \subset \R^N$ and a finite dimensional Euclidean space $E$, we use standard notations for Lebesgue spaces $L^p(\U;E)$ and Sobolev spaces $H^1(\U;E)$ or $W^{1,p}(\U;E)$. The characteristic function of $U$ will be given by $\charfun{U}$.

We will write $C^k(\U;E)$ for the space of $k$-times continuously differentiable functions $\varphi : \U \to E$ and $C^{\infty}(\U;E) := \bigcap_{k=0}^{\infty} C^k(\U;E)$ for the space of infinitely differentiable function.
We will distinguish between the spaces $C_c^k(\U;E)$ ($C^k$ functions with compact support contained in $\U$) and $C_0^k(\U;E)$ ($C^k$ functions ``vanishing on $\partial{\U}$"). 
We will write $C(\calY;E)$ to denote the space of all continuous functions \MMM which \BBB are $[0,1]^2$-periodic, and set $C^k(\calY;E) := C^k(\R^2;E) \cap C(\calY;E)$. 
We will identify $C^k(\calY;E)$ with the space of all $C^k$ functions on the 2-dimensional torus. 

We will frequently make use of the \emph{standard mollfier} $\rho \in C^{\infty}(\R^N)$, defined by
\begin{equation*}
    \rho(x)
    :=
    \begin{cases}
    C\,\expfun{\frac{1}{|x|^2-1}} & \ \text{ if } |x| < 1,\\
    0 & \ \text{ otherwise},
    \end{cases}
\end{equation*}
where the constant $C > 0$ is selected so that $\int_{\R^N} \rho(x) \,dx = 1$, and the associated family $\{\rho_\epsilon\}_{\epsilon>0} \subset C^{\infty}(\R^N)$ with
\begin{equation*}
    \rho_\epsilon(x) := \frac{1}{\epsilon^N} \rho\left(\frac{x}{\epsilon}\right).
\end{equation*}

Throughout the text, the letter $C$ stands for generic constants which may vary from line to line.

\subsection{Measures} \label{Measures}
We first recall some basic notions from measure theory that we will use throughout the paper (see, e.g. \cite{fonseca2007modern}).

Given a Borel set $\U \subset \R^N$ and a finite dimensional Hilbert space $X$, we denote by $\Mb(\U;X)$ the space of bounded Borel measures on $\U$ taking values in $X$, and endowed with the norm $\|\mu\|_{\Mb(\U;X)} := |\mu|(\U)$, where $|\mu| \in \Mb(\U\BLUE;\R\BLACK)$ is the total variation of the measure $\mu$. 
For every $\mu \in \Mb(\U;X)$ we consider the Lebesgue decomposition $\mu = \mu^a + \mu^s$, where $\mu^a$ is absolutely continuous with respect to the Lebesgue measure $\calL^N$ and $\mu^s$ is singular with respect to $\calL^N$. 
If $\mu^s = 0$, we always identify $\mu$ with its density with respect to $\calL^N$, which is a function in $L^1(\U;X)$. \EEE With a slight abuse of notation, we will write $\Mb(U;\R)=\Mb(U)$ and $\Mb(U;\R^+)=\Mb^+(U)$. \BBB

If the relative topology of $\U$ is locally compact, by Riesz representation theorem the space $\Mb(\U;X)$ can be identified with the dual of $C_0(\U;X)$, which is the space of all continuous functions $\varphi : \U \to X$ such that the set $\{|\varphi|\geq\delta\}$ is compact for every $\delta > 0$. 
The weak* topology on $\Mb(\U;X)$ is defined using this duality.

The \emph{restriction} of $\mu \in \Mb(\U;X)$ to a subset $E \in \U$ is the measure $\mu\mres{E} \in \Mb(E;X)$ defined by
\begin{equation*}
    \mu\mres{E}(B) := \mu(E \cap B), \quad \text{ for every Borel set $B \subset \U$}.
\end{equation*}

Given two real-valued measures $\mu_1,\, \mu_2 \in \Mb(\U)$ we write $\mu_1 \geq \mu_2$ if $\mu_1(B) \geq \mu_2(B)$ for every Borel set $B \subset \U$. 

\subsubsection{Convex functions of measures} \label{Convex functions of measures}
Let $U$ be an open set of $\R^N$. 
For every $\mu \in \Mb(U;X)$ let $\frac{d\mu}{d|\mu|}$ be the Radon-Nikodym derivative of $\mu$ with respect to its variation $|\mu|$. 
Let $H : X \to [0,+\infty)$ be a convex and positively one-homogeneous function such that
\begin{equation} \label{coercivity of H}
    r |\xi| \leq H(\xi) \leq R |\xi| \quad \text{for every}\, \xi \in X,
\end{equation}
where $r$ and $R$ are two constants, with $0 < r \leq R$. 

Using the theory of convex functions of measures, developed in \cite{goffman1964sublinear} and \cite{demengel1984convex}, we introduce the nonnegative Radon measure $H(\mu) \in \Mb^+(U)$ defined by
\[
    H(\mu)(A) := \int_{A} H\left(\frac{d\mu}{d|\mu|}\right) \,d|\mu|,
\]
for every Borel set $A \subset U$. 
We also consider the functional $\calH: \Mb(U;X) \to [0,+\infty)$ defined by
\[
    \calH(\mu) := H(\mu)(\U) = \int_{\U} H\left(\frac{d\mu}{d|\mu|}\right) \,d|\mu|.
\]
One can prove that $\calH$ is lower semicontinuous on $\Mb(U;X)$ with respect to weak* convergence (see, e.g., \cite[Theorem 2.38]{ambrosio2000functions}).

Let $a,\, b \in [0,T]$ with $a \leq b$. 
The \emph{total variation} of a function $\mu : [0,T] \to \Mb(U;X)$ on $[a,b]$ is defined by
\begin{equation*}
    \calV(\mu; a, b) := \sup\left\{ \sum_{i = 1}^{n-1} \left\|\mu(t_{i+1}) - \mu(t_i)\right\|_{\Mb(U;X)} : a = t_1 < t_2 < \ldots < t_n = b,\ n \in \N \right\}.
\end{equation*}
Analogously, we define the \emph{$\calH$-variation} of a function $\mu : [0,T] \to \Mb(U;X)$ on $[a,b]$ as
\begin{equation*}
    \calD_{\calH}(\mu; a, b) := \sup\left\{ \sum_{i = 1}^{n-1} \calH\left(\mu(t_{i+1}) - \mu(t_i)\right) : a = t_1 < t_2 < \ldots < t_n = b,\ n \in \N \right\}.
\end{equation*}
From \eqref{coercivity of H} it follows that
\begin{equation} \label{equivalence of variations}
    r \calV(\mu; a, b) \leq \calD_{\calH}(\mu; a, b) \leq R \calV(\mu; a, b).
\end{equation}

\subsubsection{Disintegration of a measure}
Let $S$ and $T$ be measurable spaces and let $\mu$ be a measure on $S$. 
Given a measurable function $f : S \to T$, we denote by $f_{\#}\mu$ the \emph{push-forward} of $\mu$ under the map $f$, defined by
\[
    f_{\#}\mu(B) := \mu\left(f^{-1}(B)\right), \quad \text{ for every measurable set $B \subseteq T$}.
\]
In particular, for any measurable function $g : T \to \overline{\R}$ we have
\[
    \int_{S} g \circ f \,d\mu = \int_{T} g \,d(f_{\#}\mu).
\]
Note that in the previous formula $S = f^{-1}(T)$. 

Let \MMM $S_1 \subset \R^{N_1}$, $S_2 \subset \R^{N_2}$, \MMM for some $N_1,N_2 \in \N$,  \BBB be open sets, and let $\eta \in \Mb^+(S_1)$. 
We say that a function $x_1 \in S_1 \to \mu_{x_1} \in \Mb(S_2;\MMM \R^M \BBB)$ is $\eta$-measurable if $x_1 \in S_1 \to \mu_{x_1}(B)$ is $\eta$-measurable for every Borel set $B \subseteq S_2$.

Given a $\eta$-measurable function $x_1 \to \mu_{x_1}$ such that $\int_{S_1}|\mu_{x_1}|\,d\eta<+\infty$, then the \emph{generalized product} $\eta \genprod \mu_{x_1}$ satisfies $\MMM \eta \genprod \mu_{x_1}\BBB \in \Mb(S_1 \times S_2;\MMM \R^M \BBB)$ and is such that 
\[
	\langle \eta \genprod \mu_{x_1}, \varphi \rangle := \int_{S_1} \left( \int_{S_2} \varphi(x_1,x_2) \,d\mu_{x_1}(x_2) \right) \,d\eta(x_1),
\]
for every bounded Borel function $\varphi : S_1 \times S_2 \to \R$.

Moreover, the following disintegration result holds (c.f. \cite[Theorem 2.28 and Corollary 2.29]{ambrosio2000functions}): 
\begin{theorem} \label{the basic disintegration theorem} 
Let $\mu \in \Mb(S_1 \times S_2;\MMM \R^M \BBB)$ and let $\proj : S_1 \times S_2 \to S_1$ be the projection on the first factor. 
Denote by $\eta$ the push-forward measure $\eta := \proj_{\#}|\mu| \in \Mb^+(S_1)$. 
Then there exists a unique family of bounded Radon measures $\{\mu_{x_1}\}_{x_1 \in S_1} \subset \Mb(S_2;\MMM \R^M\BBB)$ such that $x_1 \to \mu_{x_1}$ is $\eta$-measurable, and
\begin{equation*}
    \mu = \eta \genprod \mu_{x_1}.
\end{equation*}
For every $\varphi \in L^1(S_1 \times S_2, d|\mu|)$ we have
\begin{align*}
    \varphi(&x_1,\cdot) \in L^1(S_2,d|\mu_{x_1}|) \quad \text{ for } \eta\text{-a.e. } x_1 \in S_1,\\
    &x_1 \to \int_{S_2} \varphi(x_1,x_2) \,d\mu_{x_1}(x_2) \ \in L^1(S_1, d\eta),\\
    \int_{S_1 \times S_2} \varphi(&x_1,x_2) \,d\mu(x_1,x_2) = \int_{S_1} \left( \int_{S_2} \varphi(x_1,x_2) \,d\mu_{x_1}(x_2) \right) \,d\eta(x_1).
\end{align*}
Furthermore, 
\begin{equation*} \label{disintegration of the variation}
    |\mu| = \eta \genprod |\mu_{x_1}|.
\end{equation*}
\end{theorem}

Arguing as in \cite[Remark 5.5]{Francfort.Giacomini.2014}, we have the following:

\begin{proposition} \label{important disintegration remark}
With the same notation as in \Cref{the basic disintegration theorem}, for $\eta$-a.e. $x_1 \in S_1$
\[
    \frac{d\mu}{d|\mu|}(x_1,\cdot) = \frac{d\mu_{x_1}}{d|\mu_{x_1}|} \quad \text{$|\mu_{x_1}|$-a.e. on } S_2.
\]
\end{proposition}

\begin{proof}
Since $\frac{d\mu}{d|\mu|} \in L^1(S_1 \times S_2, d|\mu|)$, from \Cref{the basic disintegration theorem} we have $\frac{d\mu}{d|\mu|}(x_1,\cdot) \in L^1(S_2,d|\mu_{x_1}|)$ for $\eta$-a.e. $x_1 \in S_1$.
Thus,
\begin{align*}
    \eta \genprod \frac{d\mu_{x_1}}{d|\mu_{x_1}|}\,|\mu_{x_1}| = \eta \genprod \mu_{x_1} = \mu = \frac{d\mu}{d|\mu|}\,|\mu| = \eta \genprod \frac{d\mu}{d|\mu|}(x_1,\cdot)\,|\mu_{x_1}|,
\end{align*}
from which we have the claim.
\end{proof}

\subsection{$BD$ and $BH$ functions}

\subsubsection{Functions with bounded deformation}
Let $U$ be an open set of $\R^N$. The space $BD(U)$ of functions with {\em bounded deformation} is
the space of all functions $u \in L^1(U;\R^N)$ whose symmetric gradient $Eu:=\sym\, Du$ (in the sense of distributions) satisfies  $Eu \in \mb(U;\mathbb M^{N \times N}_{sym})$. We point out that $BD(U)$ is a 
Banach space endowed with the norm
$$
\|u\|_{L^1(U;\R^N)} +\|Eu\|_{\mb(U;\mathbb M^{N \times N}_{sym})}.
$$
We say that a sequence $\{u^k\}_k$ converges to $u$ weakly* in $BD(U)$ if $u^k\wk u$ weakly in
$L^1(U;\R^N)$ and $Eu^k\wk Eu$ weakly* in $\mb(U;\mathbb M^{N \times N}_{sym})$. 
\MMM As a consequence of compactness, then necessarily $\{u^k\}_k$ converges to $u$ strongly in $L^1$. \BBB
Every bounded
sequence in $BD(U)$ has a weakly* converging subsequence. If $U$ is bounded and has a Lipschitz boundary, $BD(U)$ can be embedded into $L^{N/(N-1)}(U;\R^N)$ (\MMM the embedding is compact in $L^p$, for $1 \leq p<N/(N-1)$\BBB) and every function $u \in BD(U)$ has a trace,
still denoted by $u$, which belongs to $L^1(\partial U;\R^N)$. If $\Gamma$ is a nonempty open subset of $\partial U$, there exists a constant $C>0$, depending on $U$ and $\Gamma$, such that
\begin{equation} \label{kornbd}
\|u\|_{L^1(U;\R^N)}\leq C \|u\|_{L^1(\Gamma)}+C\|Eu\|_{\mb(U;\mathbb M^{N \times N}_{sym})}.
\end{equation}
(see \cite[Chapter~II, Proposition~2.4 and Remark~2.5]{Temam.1985}). 
For the general properties of the space $BD(U)$ we refer to \cite{Temam.1985}.

\subsubsection{Functions with bounded Hessian}
The space $BH(U)$ of functions with {\em bounded Hessian} is
the space of all functions $u \in W^{1,1}(U)$ whose Hessian $D^2u$ (in the sense of distributions) belongs to $\mb(U;\mathbb M^{N \times N}_{\rm sym})$. It is a Banach space endowed with the norm
$$
\|u\|_{L^1(U)} +\|\nabla u\|_{L^1(U;\R^{N})}+\|D^2u\|_{\mb(U;\mathbb M^{N \times N}_{sym})}.
$$
If $U$ has the cone property, then $BH(U)$ coincides with the space of functions in $L^1(U)$ whose
Hessian belongs to $\mb(U;\mathbb M^{N \times N}_{sym})$.
If $U$ is bounded and has a Lipschitz boundary, $BH(U)$ can be embedded into $W^{1,N/(N-1)}(U)$. If $U$ is bounded and has a $C^2$ boundary, then for every function $u \in BH(U)$ one can define the traces of $u$ and of $\nabla u$, still denoted by $u$ and $\nabla u$; they satisfy $u \in W^{1,1}(\partial U)$, $\nabla u \in L^1(\partial U;\R^N)$, and $\frac{\partial u}{\partial\tau}=\nabla u\cdot\tau$ in $L^1(\partial U)$, where $\tau$ is
any tangent vector to $\partial U$. If, in addition, $N=2$, then $BH(U)$ embeds into $C(\overline U)$, which is the space of all continuous functions on $\overline U$. The general properties of the space $BH(U)$ can be found in \cite{Demengel.1984}.

\subsection{Auxiliary claims about stress tensors}

\subsubsection{Traces of stresses}
We suppose here that $\U$ is an open bounded set of class $C^2$ in $\R^N$.
If $\sigma \in L^2(\U;\M^{N \times N}_{\sym})$ and $\div\sigma \in L^2(\U;\R^N)$, then we can define a distribution $[ \sigma \nu ]$ on $\partial{\U}$ by
\begin{equation} \label{traces of the stress}
    [ \sigma \nu ](\psi) := \int_{\U} \psi \cdot \div\sigma \,dx + \int_{\U} \sigma : E\psi \,dx,
\end{equation}
for every $\psi \in H^1(\U;\R^N)$. 
It \MMM follows \BBB that $[ \sigma \nu ] \in H^{-1/2}(\partial{\U};\R^N)$ (see, e.g., \cite[Chapter 1, Theorem 1.2]{temam2001navier}). 
If, in addition, $\sigma \in L^{\infty}(\U;\M^{N \times N}_{\sym})$ and $\div\sigma \in L^N(\U;\R^N)$, then \eqref{traces of the stress} holds for $\psi \in W^{1,1}(\U;\R^N)$. 
By Gagliardo’s extension theorem, in this case we have $[ \sigma \nu ] \in L^{\infty}(\partial{\U};\R^N)$, and 
\begin{equation*}
    [ \sigma_k \nu ] \weakstar [ \sigma \nu ] \quad \text{weakly* in $L^{\infty}(\partial{\U};\R^N)$},
\end{equation*}
whenever $\sigma_k \weakstar \sigma$ weakly* in $L^{\infty}(\U;\M^{N \times N}_{\sym})$ and $\div\sigma_k \weak \div\sigma$ weakly in $L^N(\U;\R^N)$.

We will consider the normal and tangential parts of $[ \sigma \nu ]$, defined by
\begin{equation*}
    [ \sigma \nu ]\normal := ([ \sigma \nu ] \cdot \nu) \nu, \quad
    [ \sigma \nu ]\tangential := [ \sigma \nu ]-([ \sigma \nu ] \cdot \nu) \nu.
\end{equation*}
Since $\nu \in C^1(\partial{\U};\R^N)$, we have that $[ \sigma \nu ]\normal,\, [ \sigma \nu ]\tangential \in H^{-1/2}(\partial{\U};\R^N)$. If, in addition, $\sigma_{\dev} \in L^{\infty}(\U;\M^{N \times N}_{\dev})$, then it was proved in \cite[Lemma 2.4]{Kohn.Temam.1983} that $[ \sigma \nu ]\tangential \in L^{\infty}(\partial{\U};\R^N)$ and
\begin{equation*}
    \|[ \sigma \nu ]\tangential\|_{L^{\infty}(\partial{\U};\R^N)} \leq \frac{1}{\sqrt2} \|\sigma_{\dev}\|_{L^{\infty}(\U;\M^{N \times N}_{\dev})}.
\end{equation*}

More generally, if $\U$ has Lipschitz boundary and is such that there exists a compact set $S \subset \partial{U}$ with $\calH^{N-1}(S) = 0$ such that $\partial{U} \setminus S$ is a $C^2$-hypersurface, then arguing as in \cite[Section 1.2]{Francfort.Giacomini.2012} we can uniquely determine $[ \sigma \nu ]\tangential$ as an element of $L^{\infty}(\partial{U};\R^N)$ through any approximating sequence $\{\sigma_n\} \subset C^{\infty}(\closure{\U};\M^{N \times N}_{\sym})$ such that 
\begin{align*}
    &\sigma_n \strong \sigma \quad \text{strongly in } L^2(\U;\M^{N \times N}_{\sym}),\\
    &\div\sigma_n \strong \div\sigma \quad \text{strongly in } L^2(\U;\R^N),\\
    &\| (\sigma_n)_{\dev} \|_{L^{\infty}(\U;\M^{N \times N}_{\dev})} \leq \| \sigma_{\dev} \|_{L^{\infty}(\U;\M^{N \times N}_{\dev})}.
\end{align*}

\subsubsection{$L^p$ regularity}
\MMM We \EEE recall \MMM the following proposition from \BBB \cite{Francfort.Giacomini.2012} (see also \cite{Kohn.Temam.1983})\BLACK.

\begin{proposition} \label{Kohn-Temam embedding lemma}
\MMM Let $\U \subset \R^N$ be an open, bounded set with Lipschitz boundary. \BBB The set
\begin{equation*}
    \calS(\U) := \left\{ \sigma \in L^2(\U;\M^{N \times N}_{\sym}) : \div\,\sigma \in L^N(\U;\R^N),\ \sigma_{\dev} \in L^\infty(\U;\M^{N \times N}_{\dev}) \right\},
\end{equation*}
is a subset of $L^p(\U;\M^{N \times N}_{\sym})$ for every $1 \leq p < \infty$, and
\begin{equation*}
    \| \sigma \|_{L^p(\U;\M^{N \times N}_{\sym})} \leq C_p \left( \| \sigma \|_{L^2(\U;\M^{N \times N}_{\sym})} + \| \div\,\sigma \|_{L^N(\U;\R^N)} + \| \sigma_{\dev} \|_{L^{\infty}(\U;\M^{N \times N}_{dev})} \right).
\end{equation*}
\end{proposition}

\section{Setting of the problem}
\label{setting}
\EEE We describe here our modeling assumptions and recall a few associated instrumental results. \BBB
\MMM Unless otherwise stated,  $\omega \subset \R^2$ is a bounded, connected, and open set with $C^2$   boundary. \BBB
Given a small positive number $h > 0$, we assume that the set 
\begin{equation*}
    \Omega^h := \omega \times (h I),
\end{equation*}
is the reference configuration of a linearly elastic and perfectly plastic plate.

We consider a non-zero Dirichlet boundary condition on the whole lateral surface, i.e. the Dirichlet boundary of $\Omega^h$ is given by $\Gamma_\Dir^h := \partial\omega \times (h I)$.

We work under the assumption that the body is only submitted to a hard device on $\Gamma_\Dir^h$ and that there are no applied loads, i.e. the evolution is only driven by time-dependent boundary conditions. 
More general boundary conditions, together with volume and surfaces forces have been considered, e.g., in \cite{DalMaso.DeSimone.Mora.2006, Francfort.Giacomini.2012, Davoli.Mora.2013} but will, \MMM for simplicity of exposition, \BBB be neglected in this analysis.

\subsection{Phase decomposition}
We recall here some basic notation and assumptions from \cite{Francfort.Giacomini.2014}. 

\MMM Recall that  $\calY = \R^2/\Z^2$ is \BBB the $2$-dimensional torus, let $Y := [0, 1)^2$ be its associated periodicity cell, and denote by $\calI : \calY \to Y$ their canonical identification.
We denote by $\calC$ the set 
\begin{equation*}
    \calC := \calI^{-1}(\partial Y).
\end{equation*}
For any $\calZ \subset \calY$, we define
\begin{equation} \label{periodic set notation}
    \calZ_\eps := \left\{ x \in \R^2 : \frac{x}{\eps} \in \Z^2+\calI(\calZ) \right\},
\end{equation}
and to every function $F : \calY \to X$ we associate the $\eps$-periodic function $F_\eps : \R^2 \to X$, given by
\begin{equation*}
    F_\eps(x) := F\left(y_\eps\right), \;\text{ for }\; \frac{x}{\eps}-\left\lfloor \frac{x}{\eps} \right\rfloor = \calI(y_\eps) \in Y.
\end{equation*}
With a slight abuse of notation we will also write $F_\eps(x) = F\left(\frac{x}{\eps}\right)$.

The torus $\calY$ is assumed to be made up of finitely many phases $\calY_i$ together with their interfaces. 
We assume that those phases are pairwise disjoint open sets with Lipschitz boundary.
Then we have $\calY = \bigcup_{i} \closure{\calY}_i$ and we denote the interfaces by 
\begin{equation*}
    \Gamma := \bigcup_{i,j} \partial\calY_i \cap \partial\calY_j.
\end{equation*}
Furthermore, the interfaces are assumed to have a negligible intersection with the set $\calC$, i.e. for every $i$
\begin{equation} \label{transversality condition}
    \calH^{1}(\partial\calY_i \cap \calC) = 0.
\end{equation}

We will write
\begin{equation*}
    \Gamma := \bigcup_{i \neq j} \Gamma_{ij},
\end{equation*}
where $\Gamma_{ij}$ stands for the interface between $\calY_i$ and $\calY_j$.

We assume that $\omega$ is composed of the finitely many phases $(\calY_i)_\eps$, and that $\Omega^h \cup \Gamma_\Dir^h$ is a geometrically admissible multi-phase domain in the sense of \cite[Subsection 1.2]{Francfort.Giacomini.2012}. 
Additionally, we assume that $\Omega^h$ is a specimen of an elasto-perfectly plastic material having periodic elasticity tensor and dissipation potential.

We are interested in the situation when the period $\eps$ is a function of the thickness $h$, i.e. $\eps = \epsh$, and we assume that the limit 
\begin{equation*}
   \gamma := \lim_{h \to 0} \frac{h}{\epsh}.
\end{equation*}
exists in $(0, +\infty)$. 
We additionally require that $\Gamma$ satisfies the following:
there exists a compact set $S \subset \Gamma$ with $\calH^1(S) = 0$ such that
$\Gamma \setminus S$ is a $C^2$-hypersurface.

We say that a multi-phase torus $\calY$ is \emph{geometrically admissible} if it satisfies the above assumptions.

\begin{remark}
We point out that we assume greater regularity than that in \cite{Francfort.Giacomini.2014}, where the interface $\Gamma \setminus S$ was allowed to be a $C^1$-hypersurface. 
Under such weaker assumptions, in fact, the tangential part of the trace of an admissible stress $[ \sigma \nu ]\tangential$ at a point $x$ on $\Gamma \setminus S$ would not be defined independently of the considered approximating sequence. By requiring a higher regularity of $\Gamma\setminus S$, we will avoid dealing with this situation. 
\end{remark}

\medskip
\noindent{\bf The set of admissible stresses.}

We assume there exist convex compact sets $K_i \in \M^{3 \times 3}_{\dev}$ associated to each phase $\calY_i$.
We work under the assumption that there exist two constants $r_K$ and $R_K$, with $0 < r_K \leq R_K$, such that for every $i$
\begin{equation*}
    \{ \xi \in \M^{3 \times 3}_{\sym} : |\xi| \leq r_K \} \subseteq K_i \subseteq\{ \xi \in \M^{3 \times 3}_{\sym}: |\xi| \leq R_K \}.
\end{equation*}
Finally, we define
\begin{equation*}
    K(y) := K_i, \quad \text{ for } y \in \calY_i.
\end{equation*}

\medskip
\noindent{\bf The elasticity tensor.}

For every $i$, let $(\C_{\rm dev})_i$ and $k_i$ be a symmetric positive definite tensor on $\M^{3\times 3}_{\rm dev}$ and a positive constant, respectively, such that there exist two constants $r_c$ and $R_c$, with $0 < r_c \leq R_c$, satisfying
\begin{align} \label{tensorassumption}
    &r_c |\xi|^2 \leq (\C_{\dev})_i \xi : \xi \leq R_c |\xi|^2 \quad \text{ for every }\xi \in \M^{3 \times 3}_{\dev},\\
    &r_c \leq k_i \leq R_c.
\end{align}

Let $\C$ be the {\em elasticity tensor}, considered as a map from $\calY$ taking values in the set of symmetric positive definite linear operators, $\C : \calY \times \M^{3 \times 3}_{\sym} \to \M^{3 \times 3}_{\sym}$, defined as
\begin{equation*}
    \C(y) \xi := \C_{\dev}(y)\,\xi_{\dev} + \left(k(y)\,{\rm tr}{\xi}\right)\,I_{3 \times 3} \quad \text{ for every } y \in \calY \text{ and } \xi \in \M^{3 \times 3},
\end{equation*}
where $\C_{\dev}(y)=(\C_{\dev})_i$ and $k(y) = k_i$ for every $y \in \calY_i$.

Let $Q:\calY \times \M^{3 \times 3}_{\sym}\to[0,+\infty)$ be the quadratic form associated with $\C$, and given by
\begin{equation*}
    Q(y, \xi) := \frac{1}{2} \C(y) \xi : \xi \quad \text{ for every } y \in \calY \text{ and } \xi \in\M^{3 \times 3}_{\sym}.
\end{equation*}
It follows that $Q$ satisfies
\begin{equation} \label{coercivity of Q}
    r_c |\xi|^2 \leq Q(y, \xi) \leq R_c |\xi|^2 \quad \text{for every }y\in \calY\text{ and }\xi \in \M^{3 \times 3}_{\sym}.
\end{equation}

\medskip
\noindent{\bf The dissipation potential.}

For each $i$, let $H_i : \M^{3 \times 3}_{\dev} \to [0,+\infty)$ be the support function of the set $K_i$, i.e
\begin{equation*}
    H_i(\xi) = \sup\limits_{\tau \in K_i} \tau : \xi.
\end{equation*}
It follows that $H_i$ is convex, positively 1-homogeneous, and satisfies
\begin{equation} \label{coercivity of H_i}
    r_k |\xi| \leq H_i(\xi) \leq R_k |\xi| \quad \text{for every}\, \xi \in \M^{3 \times 3}_{\dev}.
\end{equation}

Then we define the dissipation potential $H : \calY \times \M^{3 \times 3}_{\dev} \to [0,+\infty]$ as follows:
\begin{enumerate}[label=(\roman*)]
    \item For every $y \in \calY_i$, we take
    \begin{equation*}
        H(y, \xi) := H_i(\xi).
    \end{equation*}
    \item 
    For a point $y \in \Gamma \setminus S$ on the interface between $\calY_i$ and $\calY_j$, such that the associated normal $\nu(y)$ points from $\calY_j$ to $\calY_i$, we set
    \begin{equation*}
        H(y, \xi) 
        := 
        \begin{cases}
        H_{ij}(a, \nu(y)) & \ \text{ if } \xi = a \odot \nu(y) \in \M^{3 \times 3}_{\dev},\\
        +\infty & \ \text{ otherwise on }\M^{3 \times 3}_{\dev},
        \end{cases}
    \end{equation*}
    where for $a \in \R^3$ and $\nu \perp a \in \mathbb{S}^2$, 
    \begin{align*}
        H_{ij}(a, \nu) := \inf\Big\{& H_i( a_i \odot \nu ) + H_j( -a_j \odot \nu ) :\\
        &a = a_i - a_j,\ a_i \perp \nu,\ a_j \perp \nu \Big\}.
    \end{align*}
    \item
    For $y \in S$, we define H arbitrarily (e.g. $H(y, \xi) := r_k\,|\xi|$).
\end{enumerate}

\begin{remark} \label{Reshetnyak remark 2}
We point out that $H$ is a Borel function on $\calY \times \M^{3 \times 3}_{\dev}$. 
Furthermore, for each $y \in \calY$, the function $\xi \mapsto H(y, \xi)$ is positively 1-homogeneous and convex.
However, the function $(y, \xi) \mapsto H(y, \xi)$ is not necessarily \MMM lower semicontinous. This creates additional difficulties in proving lower semicontinuity of dissipation functional  given in \Cref{lsc-gamma finite}, see also \cite[Theorem 5.7]{Francfort.Giacomini.2014}. \BBB  
\end{remark}

\medskip
\noindent{\bf Admissible triples and energy.}

On $\Gamma_\Dir^h$ we prescribe a boundary datum being the trace of a map $w^h \in H^1(\Omega^h;\R^3)$ of the following form:
\begin{equation} \label{MGform0}
    w^h(z) := \left( \bar{w}_1(z') - \frac{z_3}{h}\partial_1 \bar{w}_3(z'),\, \bar{w}_2(z') - \frac{z_3}{h}\partial_2 \bar{w}_3(z'),\, \frac{1}{h}\bar{w}_3(z') \right) \,\text{ for a.e. }z=(z',z_3) \in \Omega^h,
\end{equation}
where $\bar{w}_\alpha \in H^1(\omega)$, $\alpha=1,2$, and $\bar{w}_3 \in H^2(\omega)$.
The {\em set of admissible displacements and strains} for the boundary datum $w^h$ is denoted by 
$\calA(\Omega^h, w^h)$ and is defined as the class of all triples 
$(v,f,q) \in BD(\Omega^h) \times L^2(\Omega^h;\M^{3 \times 3}_{\sym}) \times \Mb(\Omega^h;\M^{3 \times 3}_{\dev})$ satisfying
\begin{eqnarray*}
& Ev=f+q \quad\text{in }\Omega^h,
\\
& q=(w^h-v)\odot\nu_{\partial\Omega^h}\calH^2 \quad\text{on }\Gamma_\Dir^h.
\end{eqnarray*} 
The function $v$ represents the {\em displacement} of the plate, while $f$ and $q$ are called the {\em elastic} and {\em plastic strain}, respectively.

For every admissible triple $(v,f,q) \in \calA(\Omega^h, w^h)$ we define the associated {\em energy} as
\begin{equation*}
    \calE_{h}(v,f,q) := 
    \int_{\Omega^h} Q\left(\frac{z'}{\epsh}, f(z)\right) \,dz + 
    \int_{\Omega^h \cup \Gamma_\Dir^h} H\left(\frac{z'}{\epsh}, \frac{dq}{d|q|} \right) \,d|q|.
\end{equation*}
The first term represents the elastic energy, while the second term accounts for plastic dissipation.

\subsection{The rescaled problem} \label{rescaled}
As usual in dimension reduction problems, it is convenient to perform a change of variables in such a way to rewrite the system on a fixed domain independent of $h$. 
To this purpose, 
we consider the open interval $I = \left(-\frac{1}{2}, \frac{1}{2}\right)$ and set
\begin{equation*}
    \Omega \,:=\, \omega \times I, \qquad 
    \Gamma_\Dir \,:=\, \partial\omega \times I. 
\end{equation*}
We consider the change of variables $\psi_h : \closure{\Omega} \to \closure{\Omega^h}$, defined as
\begin{equation} \label{eq:def-psih}
    \psi_h(x',x_3) := (x', hx_3) \quad \text{for every}\, (x',x_3) \in \closure{\Omega},
\end{equation}
and the linear operator $\Lambda_h : \M_{\sym}^{3 \times 3} \to \M_{\sym}^{3 \times 3}$ given by 
\begin{equation} \label{definition Lambda_h}
    \Lambda_h \xi:=\begin{pmatrix}
    \xi_{11} & \xi_{12} & \frac{1}{h}\xi_{13}
    \vspace{0.1 cm}\\
    \xi_{21} & \xi_{22} & \frac{1}{h}\xi_{23}
    \vspace{0.1 cm}\\
    \frac{1}{h}\xi_{31} & \frac{1}{h}\xi_{32} & \frac{1}{h^2}\xi_{33}
    \end{pmatrix}
    \quad\text{for every }\xi\in\M^{3 \times 3}_{\sym}.
\end{equation}
To any triple $(v,f,q) \in \calA(\Omega^h, w^h)$ we associate a triple $(u,e,p) \in BD(\Omega) \times L^2(\Omega;\M^{3 \times 3}_{\sym}) \times \allowbreak\Mb(\Omega \cup \Gamma_\Dir;\M^{3 \times 3}_{\sym})$ defined as follows:
\begin{equation*}
    u := (v_1, v_2, h v_3) \circ \psi_h, \qquad
    e := \Lambda_{h}^{-1}f\circ\psi_h, \qquad 
    p := \tfrac{1}{h}\Lambda_h^{-1}\psi_h^{\#}(q).
\end{equation*}
Here the measure $\psi_h^{\#}(q) \in \Mb(\Omega;\M^{3 \times 3})$ is the pull-back measure of $q$, satisfying 
\begin{equation*}
    \int_{\Omega \cup \Gamma_\Dir} \varphi : d \psi_h^{\#}(q) = \int_{\Omega^h \cup \Gamma_\Dir^h} (\varphi\circ\psi_h^{-1}) : dq \quad \text{ for every } \varphi \in C_0(\Omega \cup \Gamma_\Dir;\M^{3 \times 3}).
\end{equation*}
According to this change of variable we have
\begin{equation*}
    \calE_h(v,f,q) = h \calQ_h(\Lambda_h e) + h \calH_h(\Lambda_h p),
\end{equation*}
where
\begin{equation} \label{definition Q_h} 
    \calQ_h(\Lambda_h e) = \int_{\Omega} Q\left(\frac{x'}{\epsh}, \Lambda_h e\right) \,dx
\end{equation}
and
\begin{equation} \label{definition H_h} 
    \calH_h(\Lambda_h p) = \int_{\Omega \cup \Gamma_\Dir} H\left(\frac{x'}{\epsh}, \frac{d\Lambda_h p}{d|\Lambda_h p|}\right) \,d|\Lambda_h p|.
\end{equation}

We also introduce the scaled Dirichlet boundary datum $w \in H^1(\Omega;\R^3)$, given by
\begin{equation*}
    w(x):=(\bar{w}_1(x')-{x_3}\partial_1 w_3(x'), \bar{w}_2(x')-x_3\partial_2 w_3(x'), w_3(x'))\quad\text{for a.e.\ }x\in\Omega.
\end{equation*}
By the definition of the class $\calA(\Omega^h, w^h)$ it follows that the scaled triple $(u,e,p)$ satisfies the equalities
\begin{eqnarray}
& Eu=e+p \quad\text{in }\Omega,
 \label{straindec*}
\\
& p=(w-u)\odot\nu_{\partial\Omega}\calH^2 \quad\text{on }\Gamma_\Dir,
 \label{boundcondp*}
\\
& p_{11}+p_{22}+\frac{1}{h^2}p_{33}=0\quad\text{in }\Omega \cup \Gamma_\Dir.
 \label{trace*}
\end{eqnarray} 
We are thus led to introduce the class $\calA_h(w)$ of all triples $(u,e,p) \in BD(\Omega) \times L^2(\Omega;\M^{3 \times 3}_{\sym}) \times \allowbreak\Mb(\Omega \cup \Gamma_\Dir;\M^{3 \times 3}_{\sym})$ satisfying \eqref{straindec*}--\eqref{trace*}, and to define the functional
\begin{equation} \label{jep}
    \calJ_{h}(u,e,p) := \calQ_h(\Lambda_h e)+\calH_h(\Lambda_h p)
\end{equation}
for every $(u,e,p) \in \mathcal A_h(w)$. 
In the following we will study the asymptotic behaviour \MMM of  the quasistatic evolution \BBB associated with $\calJ_{h}$, as $h \to 0$ and $\eps_h \to 0$.

\MMM Notice that if  $\bar{w}_\alpha \in H^1(\ext{\omega})$, $\alpha=1,2$, and $\bar{w}_3 \in H^2(\ext{\omega})$, where $\omega \subset \ext{\omega}$, then we can trivially extend the triple $(u,e,p)$ to $\ext{\Omega}:= \ext{\omega} \times I$ by 
$$u=w, \qquad e=Ew, \qquad p=0 \qquad \text{ on } \ext{\Omega} \setminus \closure{\Omega}. $$
In the following we will always denote this extension also by $(u,e,p)$, whenever such an extension procedure is needed. 
\BBB 
\medskip

\noindent{\bf Kirchhoff-Love admissible triples and limit energy.} 

We consider the set of {\em Kirchhoff-Love displacements}, defined as
\begin{equation*}
    KL(\Omega):= \big\{u \in BD(\Omega) : \ (Eu)_{i3}=0 \quad\text{for } i=1,2,3\big\}.
\end{equation*}
We note that $u \in KL(\Omega)$ if and only if $u_3\in BH(\omega)$ and there exists $\bar{u}\in BD(\omega)$ such that 
\begin{equation} \label{ualfa}
    u_{\alpha}=\bar{u}_{\alpha}-x_3\partial_{x_\alpha}u_3, \quad \alpha=1,2.
\end{equation}
In particular, if $u \in KL(\Omega)$, then 
\begin{equation} \label{symmetric gradient of KL functions}
    Eu = \begin{pmatrix} \begin{matrix} E\bar{u} - x_3 D^2u_3 \end{matrix} & \begin{matrix} 0 \\ 0 \end{matrix} \\ \begin{matrix} 0 & 0 \end{matrix} & 0 \end{pmatrix}.
\end{equation}
If, in addition, $u \in W^{1,p}(\Omega;\R^3)$ for some $1 \leq p \leq \infty$, then $\bar{u}\in W^{1,p}(\omega;\R^2)$ and $u_3\in W^{2,p}(\omega)$. 
We call $\bar{u}, u_3$ the {\em Kirchhoff-Love components} of $u$. 

For every $w\in H^1(\Omega;\R^3) \cap KL(\Omega)$ we define the class $\calA_{KL}(w)$ of {\em Kirchhoff-Love admissible triples} for the boundary datum~$w$ as the set of all triples $(u,e,p) \in KL(\Omega) \times L^2(\Omega;\M^{3 \times 3}_{\sym}) \times \Mb(\Omega \cup \Gamma_\Dir;\M^{3 \times 3}_{\sym})$ satisfying
\begin{eqnarray}
& Eu=e+p \quad\text{in }\Omega, \qquad  
p=(w-u)\odot\nu_{\partial\Omega}\calH^2 \quad\text{on }\Gamma_\Dir, \label{AKL1}\\ 
& e_{i3}=0 \quad\text{in }\Omega, \quad p_{i3}=0 \quad\text{in }\Omega \cup \Gamma_\Dir,  \quad i=1,2,3.
 \label{AKL2}
\end{eqnarray}
Note that the space
\begin{equation*}
    \big\{\xi\in \M^{3 \times 3}_{\sym}: \ \xi_{i3}=0 \text{ for }i=1,2,3\big\}
\end{equation*}
is canonically isomorphic to $\M^{2 \times 2}_{\sym}$. Therefore, in the following, given a triple $(u,e,p) \in \calA_{KL}(w)$ we will usually identify $e$ with a function in $L^2(\Omega;\M^{2 \times 2}_{\sym})$ and $p$ with a measure in $\Mb(\Omega \cup \Gamma_\Dir;\M^{2 \times 2}_{\sym})$. Note also that the class $\calA_{KL}(w)$ is always nonempty as it contains the triple $(w,Ew,0)$. 

To provide a useful characterisation of admissible triplets in $\calA_{KL}(w)$, let us first recall the definition of zeroth and first order moments of functions. 

\begin{definition} \label{moments of functions}
For $f \in L^2(\Omega;\M^{2 \times 2}_{\sym})$ we denote by $\bar{f}$, $\hat{f} \in L^2(\omega;\M^{2 \times 2}_{\sym})$ and $f^\perp \in L^2(\Omega;\M^{2 \times 2}_{\sym})$ the following orthogonal components (with respect to the scalar product of $L^2(\Omega;\M^{2 \times 2}_{\sym})$) of $f$:
\begin{equation*}
    \bar{f}(x') := \int_{-\frac{1}{2}}^{\frac{1}{2}} f(x',x_3)\, dx_3, \qquad
    \hat{f}(x') := 12 \int_{-\frac{1}{2}}^{\frac{1}{2}}x_3 f(x',x_3)\,dx_3
\end{equation*}
for a.e.\ $x' \in \omega$, and
\begin{equation*}
    f^\perp(x) := f(x) - \bar{f}(x') - x_3 \hat{f}(x')    
\end{equation*}
for a.e.\ $x \in \Omega$.
We name $\bar{f}$ the \emph{zero-th order moment} of $f$ and $\hat{f}$ the \emph{first order moment} of $f$.
\end{definition}

The coefficient in the definition of $\hat{f}$ is chosen from the computation $\int_{I} x_3^2 \,dx_3 = \frac{1}{12}$. 
It ensures that if $f$ is of the form $f(x) = x_3 g(x')$, for some $g \in L^2(\omega;\M^{2 \times 2}_{\sym})$, then $\hat{f} = g$.

Analogously, we have the following definition of zeroth and first order moments of measures. 

\begin{definition} \label{moments of measures}
For $\mu \in M_b(\Omega \cup \Gamma_\Dir;\M^{2 \times 2}_{\sym})$ we define $\bar{\mu}$, $\hat{\mu} \in M_b(\omega \cup \gamma_\Dir;\M^{2 \times 2}_{\sym})$ and $\mu^\perp \in M_b(\Omega \cup \Gamma_\Dir;\M^{2 \times 2}_{\sym})$ as follows:
\begin{equation*}
    \int_{\omega \cup \gamma_\Dir} \varphi: d\bar{\mu}:= \int_{\Omega \cup \Gamma_\Dir} \varphi: d\mu , \qquad
    \int_{\omega \cup \gamma_\Dir} \varphi: d\hat{\mu} := 12\int_{\Omega \cup \Gamma_\Dir} x_3 \varphi: d\mu 
\end{equation*}
for every $\varphi \in C_0(\omega \cup \gamma_\Dir;\M^{2 \times 2}_{\sym})$, and
\begin{equation*}
    \mu^\perp := \mu - \bar{\mu} \otimes \calL^{1}_{x_3} - \hat{\mu} \otimes x_3 \calL^{1}_{x_3},
\end{equation*}
where $\otimes$ is the usual product of measures, and $\calL^{1}_{x_3}$ is the Lebesgue measure restricted to the third component of $\R^3$. 
We name $\bar{\mu}$ the \emph{zero-th order moment} of $\mu$ and $\hat{\mu}$ the \emph{first order moment} of $\mu$.
\end{definition}

\begin{remark}
More generally, for any function $f$ which is integrable over $I$, we will use the short-hand notation
\begin{equation*}
    \bar{f} 
    := \int_{I} f(\cdot, x_3) \,dx_3, \qquad
    \hat{f} 
    := 12 \int_{I} x_3\,f(\cdot, x_3) \,dx_3.
\end{equation*}
\end{remark}

We are now ready to recall the following characterisation of $\calA_{KL}(w)$, given in \cite[Proposition 4.3]{Davoli.Mora.2013}.

\begin{proposition} \label{A_KL characherization}
Let $w \in H^1(\Omega;\R^3) \cap KL(\Omega)$ and let $(u,e,p) \in  {KL}(\Omega) \times L^2(\Omega;\M^{3 \times 3}_{\sym}) \times \Mb(\Omega \cup \Gamma_\Dir;\M^{3 \times 3}_{\dev})$. 
Then $(u,e,p) \in \calA_{KL}(w)$ if and only if the following three conditions are satisfied:
\begin{enumerate}[label=(\roman*)]
    \item $E\bar{u} = \bar{e}+\bar{p}$ in $\omega$ and $\bar{p} = (\bar{w}-\bar{u}) \odot \nu_{\partial\omega} \calH^1$ on $\gamma_\Dir$;
    \item $D^2u_3 = - (\hat{e}+\hat{p})$ in $\omega$, $u_3 = w_3$ on $\gamma_\Dir$, and $\hat{p} = (\nabla u_3-\nabla w_3) \odot \nu_{\partial\omega} \calH^1$ on $\gamma_\Dir$;
    \item $p^\perp = - e^\perp$ in $\Omega$ and $p^\perp = 0$ on $\Gamma_\Dir$.
\end{enumerate}
\end{proposition}

\subsection{Definition of quasistatic evolutions}
Recalling \Cref{Measures}, the $\calH_h$-variation of a map $p^h : [0,T] \to  \Mb(\Omega \cup \Gamma_\Dir;\M^{3 \times 3}_{\dev})$ on $[a,b]$ is defined as
\begin{equation*}
    \calD_{\calH_h}(P; a, b) := \sup\left\{ \sum_{i = 1}^{n-1} \calH\left(P(t_{i+1}) - P(t_i)\right) : a = t_1 < t_2 < \ldots < t_n = b,\ n \in \N \right\}.
\end{equation*}

For every $t \in [0, T]$ we prescribe a boundary datum $w(t) \in H^1(\Omega;\R^3) \cap KL(\Omega)$ and we assume the map $t\mapsto w(t)$ to be absolutely continuous from $[0, T]$ into $H^1(\Omega;\R^3)$.

\begin{definition} \label{h-quasistatic evolution} 
Let $h > 0$. 
An \emph{$h$-quasistatic evolution} for the boundary datum $w(t)$ is a function $t \mapsto (u^h(t), e^h(t), p^h(t))$ from $[0,T]$ into $BD(\Omega) \times L^2(\Omega;\M^{3 \times 3}_{\sym}) \times \Mb(\Omega \cup \Gamma_\Dir;\M^{3 \times 3}_{\dev})$ that satisfies the following conditions:
\begin{enumerate}[label=(qs\arabic*)$_{h}$]
    \item \label{h-qs S} for every $t \in [0,T]$ we have $(u^h(t), e^h(t), p^h(t)) \in \calA_h(w(t))$ and
    \begin{equation*}
        \calQ_h(\Lambda_h e^h(t)) \leq \calQ_h(\Lambda_h \eta) + \calH_h(\Lambda_h \pi-\Lambda_h p^h(t)),
    \end{equation*}
    for every $(\upsilon,\eta,\pi) \in \calA_h(w(t))$.
    \item \label{h-qs E} the function $t \mapsto p^h(t)$ from $[0, T]$ into $\Mb(\Omega \cup \Gamma_\Dir;\M^{3 \times 3}_{\dev})$ has bounded variation and for every $t \in [0, T]$
    \begin{equation*}
        \calQ_h(\Lambda_h e^h(t)) + \calD_{\calH_h}(\Lambda_h p^h; 0, t) = \calQ_h(\Lambda_h e^h(0)) 
        + \int_0^t \int_{\Omega} \C\left(\tfrac{x'}{\epsh}\right) \Lambda_h e^h(s) : E\dot{w}(s) \,dx ds.
    \end{equation*}
\end{enumerate}
\end{definition}

The following existence result of a quasistatic evolution for a general multi-phase material can be found in \cite[Theorem 2.7]{Francfort.Giacomini.2012}.

\begin{theorem} \MMM Assume \eqref{tensorassumption} and \eqref{coercivity of Q}. \BBB
Let $h > 0$ and let $(u^h_0, e^h_0, p^h_0) \in \calA_h(w(0))$ satisfy the global stability condition \ref{h-qs S}. 
Then, there exists a two-scale quasistatic evolution $t \mapsto (u^h(t), e^h(t), p^h(t))$ for the boundary datum $w(t)$ such that $u^h(0) = u_0$,\, $e^h(0) = e^h_0$, and $p^h(0) = p^h_0$.
\end{theorem}

\MMM Our goal is to study the asymptotics of the quasistatic evolution when $h$ goes to zero. The main result is given by Theorem \ref{main result}.\BBB 

\subsection{Two-scale convergence adapted to dimension reduction}

We briefly recall some results and definitions from \cite{Francfort.Giacomini.2014}. 

\begin{definition}
\label{def:2-scale-meas}
Let $\Omega \subset \R^3$ be an open set.
Let $\{\mu^h\}_{h>0}$ be a family in $\Mb(\Omega)$ and consider $\mu \in \Mb(\Omega \times \calY)$. 
We say that
\begin{equation*}
    \mu^h \weakstartwoscale \mu \quad \text{two-scale weakly* in }\Mb(\Omega \times \calY),
\end{equation*}
if for every $\chi \in C_0(\Omega \times \calY)$
\begin{equation*}
    \lim_{h \to 0} \int_{\Omega} \chi\left(x,\frac{x'}{\epsh}\right) \,d\mu^h(x) = \int_{\Omega \times \calY} \chi(x,y) \,d\mu(x,y).
\end{equation*}
The convergence above is called \emph{two-scale weak* convergence}.
\end{definition}
\MMM
\begin{remark} \label{transfertwoscale}
Notice that the family $\{\mu^h\}_{h>0}$ determines the family of measures $\{\nu^h\}_{h>0} \subset \Mb(\Omega \times \calY)$ obtained by setting
$$\int_{\Omega \times \calY} \chi(x,y)\,d\nu^h=\int_{\Omega} \chi \left(x,\frac{x'}{h}\right) \,d\mu^h (x)$$
for every $\chi \in C_0^0(\Omega \times \calY)$. Thus $\mu$ is simply the weak* limit in $\Mb(\Omega \times \calY)$ of a suitable subsequence
of $\{\nu^h\}_{h>0}$. 
\end{remark} 
\BBB
We collect some basic properties \MMM of two-scale convergence \BBB below:

\begin{proposition}
\begin{enumerate}[label=(\roman*)]
    \item
    Any sequence that is bounded in $\Mb(\Omega)$ admits a two-scale weakly* convergent subsequence.
    \item 
    Let $\calD \subset \calY$ and assume that 
    $\supp(\mu^h) \subset \Omega \cap (\calD_\epsh \times I)$.
    If $\mu^h \weakstartwoscale \mu$ two-scale weakly* in $\Mb(\Omega \times \calY)$, then $\supp(\mu) \subset \Omega \times \closure{\calD}$.
\end{enumerate}
\end{proposition}

\section{Compactness results}
\label{compactness}

In this section, we provide a characterization of two-scale limits of symmetrized scaled gradients. 
We will consider sequences of deformations $\{v^h\}$ such that $v^h \in BD(\Omega^h)$ for every $h > 0$, their $L^1$-norms are uniformly bounded \MMM (up to rescaling)\BBB, and their symmetrized gradients $E v^h$ form a sequence of uniformly bounded Radon measures \MMM(again, up to rescaling). \BBB 
\MMM As already explained in Section \ref{rescaled}, \BBB we associate to the sequence $\{v^h\}$ above a rescaled sequence of maps $\{u^h\} \subset BD(\Omega)$, defined as
\begin{equation*}
    u^h := (v^h_1, v^h_2, h v^h_3) \circ \psi_h,
\end{equation*}
where $\psi_h$ is defined in \eqref{eq:def-psih}.
The symmetric gradients of the maps $\{v^h\}$ and $\{u^h\}$ are related as follows
\begin{equation} \label{eq:scaled-gradient}
    \MMM \frac{1}{h} E v^h = (\psi_h)_{\#} (\Lambda_h Eu^h). \BBB
\end{equation}
\MMM The boundedness of $\frac{1}{h}\|Ev^h\|_{\mathcal{M}_b(\Omega^h;\mathbb M^{3 \times 3}_{sym}) }  $ is equivalent to the boundedness of $\|\Lambda_h Eu^h\|_{\mathcal{M}_b(\Omega;\mathbb M^{3 \times 3}_{sym}) } $.
\MMM We will express our compactness result with respect to the sequence $\{u^h\}_{h>0}$. 
\BBB 

We first recall a compactness result for sequences of non-oscillating fields (see \cite{Davoli.Mora.2013}).

\begin{proposition} \label{two-scale weak limit of scaled strains - 2x2 submatrix}
\MMM Let $\{u^h\}_{h>0} \subset BD(\Omega)$ be a sequence such that there exists a constant $C>0$ for which $$\|u^h\|_{L^1(\Omega;\R^3)}+\|\Lambda_h Eu^h\|_{\mathcal{M}_b(\Omega;\mathbb M^{3 \times 3}_{sym}) } \leq C. $$
\BBB 
Then, there exist functions $\bar{u} = (\bar{u}_1, \bar{u}_2) \in BD(\omega)$ and $u_3 \in BH(\omega)$ such that, up to subsequences, there holds 
\begin{align*}
    u^h_{\alpha} &\strong \bar{u}_{\alpha}-x_3 \partial_{x_\alpha}u_3, \quad \text{strongly in }L^1(\Omega), \quad \alpha \in \{1,2\},\\
    u^h_3 &\strong u_3, \quad \text{strongly in }L^1(\Omega),\\
     Eu^h &\weakstar  \begin{pmatrix} E \bar{u} - x_3 D^2u_3 & 0 \\ 0 & 0 \end{pmatrix}  \quad \text{weakly* in }\Mb(\Omega;\M^{3 \times 3}_{\sym}). \BLACK
\end{align*}
\end{proposition}

Now we turn to identifying the two-scale limits of the sequence $\Lambda_h E u^h$. 

\subsection{Corrector properties and duality results}
In order to define and analyze the space of measures which arise as two-scale limits of scaled symmetrized gradients of $BD$ functions, we will consider the following general framework \MMM (see also \cite{breit2020trace}). \BBB

Let $V$ and $W$ be finite-dimensional Euclidean spaces of dimensions $N$ and $M$, respectively.
We will consider $k$\textsuperscript{th} order linear homogeneous partial differential operators with constant coefficients $\DiffOpA : C_c^{\infty}(\R^n;V) \to C_c^{\infty}(\R^n;W)$. 
More precisely, the operator $\DiffOpA$ acts on functions $u : \R^n \to V$ as
\begin{equation*} \label{DiffOp definition}
    \DiffOpA u  \,:=\, \sum_{|\alpha| = k} \LinOpA_\alpha \partial^\alpha u.
\end{equation*}
where the coefficients $\LinOpA_\alpha \in W \otimes V^* \cong \mathrm{Lin}(V;W)$ are constant tensors, $\alpha = (\alpha_1, \dots, \alpha_n) \in \N_0^n$ is a multi-index and $\partial^\alpha := \partial_1^{\alpha_1} \cdots \partial_n^{\alpha_n}$ denotes the distributional partial derivative of order $|\alpha| = \alpha_1 + \cdots + \alpha_n$.

We define the space
\begin{equation*}
    BV^{\DiffOpA}(U) = \Big\{ u \in L^1(U;V) :\DiffOpA u \in \Mb(U;W) \Big\}
\end{equation*}
of \emph{functions with bounded $\DiffOpA$-variations} on an open subset $U$ of $\R^n$.
This is a Banach space endowed with the norm
\begin{equation*}
    \|u\|_{BV^{\DiffOpA}(U)} := \|u\|_{L^1(U;V)} + |\DiffOpA u|(U). 
\end{equation*}
Here, the distributional $\DiffOpA$-gradient is defined and extended to distributions via the duality
\begin{equation*}
    \int_{U} \varphi \cdot d\DiffOpA u := \int_{U} \DiffOpA^* \varphi \cdot u \,dx, \quad \varphi \in C_c^{\infty}(U;W^*),
\end{equation*}
where $\DiffOpA^* : C_c^{\infty}(\R^n;W^*) \to C_c^{\infty}(\R^n;V^*)$ is the formal $L^2$-adjoint operator of $\DiffOpA$
\begin{equation*} \label{DiffOp adjont definition}
    \DiffOpA^* \,:=\, (-1)^k \sum_{|\alpha| = k} \LinOpA_\alpha^* \partial^\alpha.
\end{equation*}
The \emph{total $\DiffOpA$-variation} of $u \in L^1_{loc}(U;V)$ is defined as
\begin{align*} \label{DiffOp variation}
    |\DiffOpA u|(U) := \sup\left\{\int_{U} \DiffOpA^*\varphi \cdot u \,dx : \varphi \in C_c^k(U;W^*), \; |\varphi| \leq 1 \right\}.
\end{align*}
Let $\{u_n\} \subset BV^{\DiffOpA}(U)$ and $u \in BV^{\DiffOpA}(U)$. We say that $\{u_n\}$ converges weakly* to $u$ in $BV^{\DiffOpA}$ if $u_n \strong u \;\text{ in } L^1(U;V)$ and $\DiffOpA u_n \weakstar \DiffOpA u \;\text{ in } \Mb(U;W)$.

In order to characterize the two-scale weak* limit of scaled symmetrized gradients, we will generally consider two domains  \MMM $\Omega_1 \subset \R^{N_1}$, $\Omega_2 \subset \R^{N_2}$, for some $N_1, N_2 \in \N$ \BBB and assume that the operator $\DiffOp$ is defined through partial derivatives only with respect to the entries of the $n_2$-tuple $x_2$.
In the spirit of \cite[Section 4.2]{Francfort.Giacomini.2014}, we will define the space
\begin{align*}
    \CorrSpace{\DiffOp}{\Omega_1}{\Omega_2} := \Big\{\mu \in \Mb(\Omega_1 \times \Omega_2&;V) : \DiffOp\mu \in \Mb(\Omega_1 \times \Omega_2;W),\\
    \mu(F \times \Omega_2&) = 0 \textrm{ for every Borel set } S \subseteq \Omega_1 \Big\}.
\end{align*}

We will assume that $BV^{\DiffOp}(\Omega_2)$ satisfies the following weak* compactness property:
\begin{assumption} \label{BV^A assumption 1}
If $\{u_n\} \subset BV^{\DiffOp}(\Omega_2)$ is uniformly bounded in the $BV^{\DiffOp}$-norm, then there exists a subsequence $\{u_m\} \subseteq \{u_n\}$ and a function $u \in BV^{\DiffOp}(\Omega_2)$ such that $\{u_m\}$ converges weakly* to $u$ in $BV^{\DiffOp}(\Omega_2)$, i.e.
\begin{equation*} \label{BV^A Poincare-Korn weak* compactness}
    u_m \strong u \;\text{ in } L^1(\Omega_2;V) \;\text{ and }\; \DiffOp u_m \weakstar \DiffOp u \;\text{ in } \Mb(\Omega_2;W).
\end{equation*}

Furthermore, there exists a countable collection $\{\U^k\}$ of open subsets of $\R^{n_2}$ that increases to $\Omega_2$ (i.e. $\closure{\U^k} \subset \U^{k+1}$ for every $k\in \N$, and $\Omega_2 = \bigcup_{k} \U^k$) such that $BV^{\DiffOp}(\U^k)$ satisfies the weak* compactness property above for every $k\in \N$.
\end{assumption}

The following theorem is our main disintegration result for measures in $\CorrSpace{\DiffOp}{\Omega_1}{\Omega_2}$, which will be instrumental to define a notion of duality for admissible two-scale configurations.
The proof is an adaptation of the arguments in \cite[Proposition 4.7]{Francfort.Giacomini.2014}.

\begin{proposition} \label{BV^A main property}
\MMM Let Assumption \ref{BV^A assumption 1} be satisfied. \BBB
Let $\mu \in \CorrSpace{\DiffOp}{\Omega_1}{\Omega_2}$. 
Then there exist $\eta \in \Mb^+(\Omega_1)$ and a Borel map $(x_1,x_2) \in \Omega_1 \times \Omega_2 \mapsto \mu_{x_1}(x_2) \in V$ such that, for $\eta$-a.e. $x_1 \in \Omega_1$,
\begin{equation} \label{BV^A main property 1}
    \mu_{x_1} \in BV^{\DiffOp}(\Omega_2), \qquad \int_{\Omega_2} \mu_{x_1}(x_2) \,dx_2 = 0, \qquad |\DiffOp\mu_{x_1}|(\Omega_2) \neq 0,
\end{equation}
and
\begin{equation} \label{BV^A main property 2}
    \mu = \mu_{x_1}(x_2) \,\eta \otimes \calL^{n_2}_{x_2}.
\end{equation}
Moreover, the map $x_1 \mapsto \DiffOp\mu_{x_1} \in \Mb(\Omega_2;W)$ is $\eta$-measurable and
\begin{equation*}
    \DiffOp\mu = \eta \genprod \DiffOp\mu_{x_1}.
\end{equation*}
\end{proposition}

\begin{proof}
By assumption, we have $\mu \in \Mb(\Omega_1 \times \Omega_2;V)$ and $\lambda := \DiffOp\mu \in \Mb(\Omega_1 \times \Omega_2;W)$. 
Setting 
\begin{equation*}
    \eta := \proj_{\#}|\mu|+\proj_{\#}|\lambda| \in \Mb^+(\Omega_1),
\end{equation*}
where $\proj_{\#}$ is the push-forward by the projection of $\Omega_1 \times \Omega_2$ on $\Omega_1$, we obtain \MMM as a consequence of Theorem \ref{the basic disintegration theorem}: \BBB
\begin{equation} \label{BV^A main property disintegrations}
    \mu = \eta \genprod \mu_{x_1} \;\text{ and }\; \lambda = \eta \genprod \lambda_{x_1},
\end{equation}
with $\mu_{x_1} \in \Mb(\Omega_2;V)$ and $\lambda_{x_1} \in \Mb(\Omega_2;W)$. 
Further, if we set $S := \{x_1 \in \Omega_1 : |\lambda_{x_1}|(\Omega_2) \neq 0\}$, then $\lambda = \eta\mres{S} \genprod \lambda_{x_1}$.

For every $\varphi^{(1)} \in C_c^{\infty}(\Omega_1)$ and $\varphi^{(2)} \in C_c^{\infty}(\Omega_2;W^*)$ we have
\begin{align*}
    \int_{\Omega_1} \varphi^{(1)}(x_1) \left\langle \mu_{x_1}, \DiffOp^*\varphi^{(2)} \right\rangle \cdot d\eta(x_1)
    &= \int_{\Omega_1} \left( \int_{\Omega_2} \varphi^{(1)}(x_1) \DiffOp^*\varphi^{(2)}(x_2) \cdot d\mu_{x_1}(x_2) \right) \cdot d\eta(x_1)\\
    &= \left\langle \eta \genprod \mu_{x_1}, \varphi^{(1)} \DiffOp^*\varphi^{(2)} \right\rangle
    = \left\langle \mu, \DiffOp^*\left(\varphi^{(1)} \varphi^{(2)}\right) \right\rangle\\
    &= \left\langle \DiffOp\mu, \varphi^{(1)} \varphi^{(2)} \right\rangle
    = \left\langle \eta\mres{S} \genprod \lambda_{x_1}, \varphi^{(1)} \varphi^{(2)} \right\rangle\\
    &= \int_{\Omega_1} \left( \int_{\Omega_2} \varphi^{(1)}(x_1) \varphi^{(2)}(x_2) \cdot d\lambda_{x_1}(x_2) \right) \charfun{S}(x_1) \cdot d\eta(x_1)\\
    &= \int_{\Omega_1} \varphi^{(1)}(x_1) \left\langle \charfun{S}(x_1)\lambda_{x_1}, \varphi^{(2)} \right\rangle \cdot d\eta(x_1).
\end{align*}
From this we infer that for $\eta$-a.e. $x_1 \in \Omega_1$ and for every $\varphi \in C_c^{\infty}(\Omega_2;W^*)$
\begin{equation} \label{corrector disintegration relation}
    \left\langle \mu_{x_1}, \DiffOp^*\varphi \right\rangle = \left\langle \charfun{S}(x_1)\lambda_{x_1}, \varphi \right\rangle.
\end{equation}

We can consider $\mu_{x_1}$ and $\lambda_{x_1}$ as measures on $\R^{n_2}$ if we extend the measure $\mu$ by zero on the complement of $\Omega_1$. 
Then, using the standard mollifiers $\{\rho_\epsilon\}_{\epsilon>0}$ on $\R^{n_2}$, we define the functions $\mu_{x_1}^\epsilon := \mu_{x_1} \ast \rho_\epsilon$ and $\lambda_{x_1}^\epsilon := \lambda_{x_1} \ast \rho_\epsilon$,
which are smooth and uniformly bounded in $L^1(\Omega_2;V)$ and $L^1(\Omega_2;W)$, respectively.
For every $\varphi \in C_c^k(\Omega_2;W^*)$,
$\supp(\varphi) \subset \U^k$ for $k$ large enough.
Furthermore, the support of $\varphi \ast \rho_\epsilon$ is contained in $\Omega_2$ provided $\epsilon$ is sufficiently small (\MMM smallness depending only on $k$\BBB), and thus from \eqref{corrector disintegration relation}
we have 
\begin{align*}
    \langle \mu_{x_1}^\epsilon, \DiffOp^*\varphi \rangle 
    &= \int_{\R^{n_2}} \left( \mu_{x_1} \ast \rho_\epsilon \right) \cdot \DiffOp^*\varphi \,dx_2
    = \int_{\R^{n_2}} \left( \DiffOp^*\varphi \ast \rho_\epsilon \right) \cdot d\mu_{x_1}\\
    &= \int_{\R^{n_2}} \DiffOp^*\left( \varphi \ast \rho_\epsilon \right) \cdot d\mu_{x_1}
    = \langle \mu_{x_1}, \DiffOp^*\left( \varphi \ast \rho_\epsilon \right) \rangle\\
    &= \langle \charfun{S}(x_1)\lambda_{x_1}, \varphi \ast \rho_\epsilon \rangle
    = \int_{\R^{n_2}} \left( \varphi \ast \rho_\epsilon \right) \cdot \charfun{S}(x_1) \,d\lambda_{x_1}\\
    &= \int_{\R^{n_2}} \charfun{S}(x_1) \left( \lambda_{x_1} \ast \rho_\epsilon \right) \cdot \varphi \,dx_2\\
    &= \langle \charfun{S}(x_1)\lambda_{x_1}^\epsilon, \varphi \rangle.
\end{align*}
Hence, for $\eta$-a.e. $x_1 \in \Omega_1$ the sequence $\{\mu_{x_1}^\epsilon\}$ is eventually bounded in $BV^{\DiffOp}(\U^k)$. By \Cref{BV^A assumption 1}, this implies strong convergence in $L^1(\U^k;V)$ up to a subsequence. 
As $\epsilon \to 0$, we have both $\varphi \ast \rho_\epsilon \strong \varphi$ and $\DiffOp^*\varphi \ast \rho_\epsilon \strong \DiffOp^*\varphi$ uniformly, so by the Lebesgue's dominated convergence theorem we obtain, for $\eta$-a.e. $x_1 \in \Omega_1$,
\begin{equation*}
    \langle \mu_{x_1}^\epsilon, \DiffOp^*\varphi \rangle \to \langle \mu_{x_1}, \DiffOp^*\varphi \rangle \;\text{ and }\; \langle \charfun{S}(x_1)\lambda_{x_1}^\epsilon, \varphi \rangle \to \langle \charfun{S}(x_1)\lambda_{x_1}, \varphi \rangle.
\end{equation*}
From the convergence above, we conclude for $\eta$-a.e. $x_1 \in \Omega_1$ that $\mu_{x_1}^\epsilon \strong \mu_{x_1}$ strongly in $L^1(\U^k;V)$. 
Since $\mu_{x_1}$ has bounded total variation, we have that $\mu_{x_1} \in L^1(\Omega_2;V)$ for $\eta$-a.e. $x_1 \in \Omega_1$.
This, together with \eqref{corrector disintegration relation}, implies
\begin{equation*}
    \mu_{x_1} \in BV^{\DiffOp}(\Omega_2) \;\text{ and }\; \DiffOp\mu_{x_1} = \charfun{S}(x_1)\lambda_{x_1}.
\end{equation*}
From \eqref{BV^A main property disintegrations} we now have that $\mu$ is absolutely continuous with respect to $\eta \otimes \calL^{n_2}_{x_2}$. 
Consequently, for $\eta$-a.e. $x_1 \in \Omega_1$ there exists a Borel measurable function which is equal to $\mu_{x_1}$ for $\calL^{n_2}_{x_2}$-a.e. $x_2 \in \Omega_2$, so that \eqref{BV^A main property 2} immediately follows.

Finally, since $\mu(F \times \Omega_2) = 0$ for every Borel set $F \subseteq \Omega_1$, we have
\begin{equation*}
    \int_{\Omega_1} f(x_1)\left( \int_{\Omega_2} \mu_{x_1}(x_2) \,dx_2 \right) \,d\eta(x_1) = \int_{\Omega_1 \times \Omega_2} f(x_1) \,d\mu(x_1,x_2) = 0
\end{equation*}
for every $f \in C_c(\Omega_1)$, from which we obtain the second claim in \eqref{BV^A main property 1}.
This concludes the proof.
\end{proof}

Lastly, we give a necessary and sufficient condition with which we can characterize the $\DiffOp$-gradient of a measure, under the following two assumptions.

\begin{assumption} \label{BV^A assumption 2}
For every $\chi \in C_0(\Omega_1 \times \Omega_2;W)$ with $\DiffOp^*\chi = 0$ (in the sense of distributions), there exists a sequence of smooth functions $\{\chi_n\} \subset C_c^{\infty}(\Omega_1 \times \Omega_2;W)$ such that $\DiffOp^*\chi_n = 0$ for every $n$, and $\chi_n \strong \chi$ in $L^{\infty}(\Omega_1 \times \Omega_2;W)$.
\end{assumption}

\begin{assumption} \label{BV^A assumption 3}
The following Poincar\'{e}-Korn type inequality holds in $BV^{\DiffOp}(\Omega_2)$:
\begin{equation*} \label{BV^A Poincare-Korn inequality}
    \left\|u - \int_{\Omega_2} u \,dx_2\right\|_{L^1(\Omega_2;V)} \leq C |\DiffOp u|(\Omega_2), \quad \forall u \in BV^{\DiffOp}(\Omega_2).
\end{equation*}
\end{assumption}

\begin{proposition} \label{BV^A duality lemma}
\MMM Let Assumption \ref{BV^A assumption 1}, \ref{BV^A assumption 2} and \ref{BV^A assumption 3} be satisfied. \BBB
Let $\lambda \in \Mb(\Omega_1 \times \Omega_2;W)$. 
Then, the following items are equivalent:
\begin{enumerate}[label=(\roman*)]
	\item \label{BV^A duality lemma (i)}
	For every $\chi \in C_0(\Omega_1 \times \Omega_2;W)$ with $\DiffOp^*\chi = 0$ (in the sense of distributions) we have
	\begin{equation*}
	    \int_{\Omega_1 \times \Omega_2} \chi(x_1,x_2) \cdot d\lambda(x_1,x_2) = 0.
	\end{equation*}
	\item \label{BV^A duality lemma (ii)}
	There exists $\mu \in \CorrSpace{\DiffOp}{\Omega_1}{\Omega_2}$ such that $\lambda = \DiffOp\mu$. 
\end{enumerate}
\end{proposition}

\begin{proof}
\BLACK
Let $\chi \in C_0(\Omega_1 \times \Omega_2;W)$ with $\DiffOp^*\chi = 0$ (in the sense of distributions) and let $\{\chi_n\}$ be an approximating sequence of $\chi$ as in \Cref{BV^A assumption 2}. 
Assume that \ref{BV^A duality lemma (ii)} holds. Then, we have
\MMM \begin{align*}
    \int_{\Omega_1 \times \Omega_2} \chi(x_1,x_2) \cdot d\lambda(x_1,x_2) 
    &= \int_{\Omega_1 \times \Omega_2} \chi(x_1,x_2) \cdot d\DiffOp\mu(x_1,x_2)\\
    &= \lim_n \int_{\Omega_1 \times \Omega_2} \chi_n(x_1,x_2) \cdot d\DiffOp\mu(x_1,x_2)\MMM \\
    &=\lim_n \int_{\Omega_1 \times \Omega_2} \DiffOp^* \chi_n(x_1,x_2)\, d\mu(x_1,x_2)=0. \BBB
\end{align*}
So we have \ref{BV^A duality lemma (i)} \BBB.
\BLACK

Let us prove that the space
\begin{equation*}
    \calE^{\DiffOp} = \left\{ \DiffOp\mu : \mu \in \CorrSpace{\DiffOp}{\Omega_1}{\Omega_2} \right\}
\end{equation*}
is weakly* closed in $\Mb(\Omega_1 \times \Omega_2;W)$. 
By the Krein-\v{S}mulian theorem it is enough to show that the intersection of $\calE^{\DiffOp}$ with every closed ball in $\Mb(\Omega_1 \times \Omega_2;W)$ is weakly* closed. 
This implies, since the weak* topology is metrizable on any closed ball of $\Mb(\Omega_1 \times \Omega_2;W)$, that it is enough to prove that $\calE^{\DiffOp}$ is sequentially weakly* closed.

Let $\{\lambda_n\}_{n\in\N} \subset \calE^{\DiffOp}$ and $\lambda \in \Mb(\Omega_1 \times \Omega_2;W)$ be such that
\begin{equation*}
    \lambda_n \weakstar \lambda \;\text{ in } \Mb(\Omega_1 \times \Omega_2;W).
\end{equation*}
By the definition of the space $\calE^{\DiffOp}$, there exist measures $\mu_n \in \Mb(\Omega_1 \times \Omega_2;V)$ such that $\lambda_n = \DiffOp \mu_n$. 
By \Cref{BV^A main property}, for every $n \in \N$ we have that there exist $\eta_n \in \Mb^+(\Omega_1)$ and $\mu^n_{x_1} \in BV^{\DiffOp}(\Omega_2)$ such that, for $\eta_n$-a.e. $x_1 \in \Omega_1$,
\begin{equation*} 
    \mu_n = \mu^n_{x_1}(x_2) \,\eta_n \otimes \calL^{n_2}_{x_2}, \qquad \DiffOp\mu_n = \eta_n \genprod \DiffOp\mu^n_{x_1}.
\end{equation*}
Additionally, $\mu^n_{x_1}$ satisfies $\int_{\Omega_2} \mu^n_{x_1}(x_2) \,dx_2 = 0$ for every $n\in \N$. 
Then, 
by \Cref{BV^A assumption 3}, 
there is a constant $C$ independent of $n$ such that
\begin{align*}
    |\mu_n|(\Omega_1 \times \Omega_2)
    &= \int_{\Omega_1 \times \Omega_2} |\mu_n(x_1,x_2)| \,dx_1 dx_2
    = \int_{\Omega_1} \left( \int_{\Omega_2} |\mu^n_{x_1}(x_2)| \,dx_2 \right) \,d\eta_n(x_1)\\ 
    &\leq C \int_{\Omega_1} |\DiffOp\mu^n_{x_1}|(\Omega_2) \,d\eta_n(x_1)
    = C \int_{\Omega_1} \left( \int_{\Omega_2} \,d|\DiffOp\mu^n_{x_1}|(x_2) \right) \,d\eta_n(x_1)\\
    &= C \int_{\Omega_1 \times \Omega_2} \,d\left( \eta_n \genprod |\DiffOp\mu^n_{x_1}| \right)
    = C |\DiffOp\mu_n|(\Omega_1 \times \Omega_2) 
    \leq C.
\end{align*}
Hence there exists a subsequence of $\{ \mu_n \}$, not relabeled, and an element $\mu \in \Mb(\Omega_1 \times \Omega_2;V)$ such that
\begin{equation*}
    \mu_n \weakstar \mu \;\text{ in } \Mb(\Omega_1 \times \Omega_2;V).
\end{equation*}
Then, for every $\varphi \in C_c^{\infty}(\Omega_1 \times \Omega_2;W^*)$ we have
\begin{align*}
    \langle \lambda, \varphi \rangle 
    &= \lim_{n} \langle \lambda_n, \varphi \rangle
    = \lim_{n} \langle \DiffOp \mu_n, \varphi \rangle\\
    &= \lim_{n} \langle \mu_n, \DiffOp^*\varphi \rangle
    = \langle \mu, \DiffOp^*\varphi \rangle.
\end{align*}
From the convergence above we deduce that $\lambda = \DiffOp \mu \in \calE^{\DiffOp}$.
This implies that $\calE^{\DiffOp}$ is weakly* closed in $\Mb(\Omega_1 \times \Omega_2;W) = \left( C_0(\Omega_1 \times \Omega_2;W^*) \right)'$.

Assume now that \ref{BV^A duality lemma (i)} holds. 
If $\lambda \notin \calE^{\DiffOp}$, by Hahn-Banach's theorem, there exists $\chi \in C_0(\Omega_1 \times \Omega_2;W^*)$ such that
\begin{equation} \label{BV^A lemma - (b) 1}
    \int_{\Omega_1 \times \Omega_2} \chi \cdot d\lambda = 1,
\end{equation}
and, for every $u \in BV^{\DiffOp}(\Omega_1 \times \Omega_2)$,
\begin{equation} \label{BV^A lemma - (b) 0}
    \int_{\Omega_1 \times \Omega_2} \chi \cdot d\DiffOp u = 0.
\end{equation}
In particular, choosing $u$ to be a smooth function, \eqref{BV^A lemma - (b) 0} implies that $\DiffOp^*\chi = 0$ (in the sense of distributions). 
As a consequence, \eqref{BV^A lemma - (b) 1} contradicts \ref{BV^A duality lemma (i)}. 
Thus, $\lambda \in \calE^{\DiffOp}$.
\end{proof}

\subsubsection{Compactness result for scaled maps with finite energy}
If we consider $\DiffOp = \widetilde{E}_{\gamma}$,\, $\DiffOp^* = \widetilde{\div}_{\gamma}$,\, $\Omega_1 = \omega$ with points $x_1 = x'$, and $\Omega_2 = I \times \calY$ with points $x_2 = (x_3,y)$, then we denote the associated spaces from the previous section by:
\begin{equation*}
    \BDgamma := \Big\{ u \in L^1(I \times \calY;\R^3) : \widetilde{E}_{\gamma}u \in \Mb(I \times \calY;\M^{3 \times 3}_{\sym}) \Big\},
\end{equation*}
\begin{align*}
    \calXgamma{\omega} := \Big\{\mu \in \Mb(\Omega \times \calY&;\R^3) : \widetilde{E}_{\gamma}\mu \in \Mb(\Omega \times \calY;\M^{3 \times 3}_{\sym}),\\
    \mu(F \times I \times \calY&) = 0 \textrm{ for every Borel set } F \subseteq \omega \Big\}.
\end{align*}
\MMM Despite the fact that $\calY$ is a flat torus, Proposition \ref{BV^A main property} and Proposition \ref{BV^A duality lemma} are satisfied if we establish the validity of Assumption \ref{BV^A assumption 1},  \ref{BV^A assumption 2} and  \ref{BV^A assumption 3}, which will be done below. \BBB
\begin{remark} \label{scaling from BD to BDgamma}
To each $u \in \BDgamma$, we can associate a function $v := \left(\frac{1}{\gamma}u_1, \frac{1}{\gamma}u_2, u_3\right)$.
Then
\MMM \begin{equation*}
    E v =
    \begin{pmatrix} \begin{matrix} \frac{1}{\gamma}\,E_{y}u' \end{matrix} & \frac{1}{2}\hspace{-0.2em}\left(D_{y}u_3+\frac{1}{\gamma}\,\partial_{x_3}u'\right) \\ \frac{1}{2}\hspace{-0.2em}\left(D_{y}u_3+\frac{1}{\gamma}\,\partial_{x_3}u'\right)^T & \partial_{x_3}u_3 \end{pmatrix},
\end{equation*}
\BBB
from which we can see that $v \in BD(I \times \calY)$. Here $E_yu'$ denotes the symmetrized gradient in $y$ of the field $u'$, which is a $2 \times 2$ matrix. 
Alternatively, we can define the change of variables $\psi : (\gamma I) \times \calY \to I \times \calY$ given by $\psi(x_3,y) := \left(\tfrac{1}{\gamma}x_3,y\right)$ and consider the function $w := u \circ \psi$.
Then $w \in BD((\gamma I) \times \calY)$ and we have
\begin{equation*}
    \widetilde{E}_{\gamma}u = \MMM \frac{1}{\gamma} \BBB \psi_{\#}(\BLUE \widetilde{E}_{1}w\BLACK).
\end{equation*}
Using any one of these scalings, we obtain that $\BDgamma$ satisfies the weak* compactness property \Cref{BV^A assumption 1}.
\end{remark}

\MMM The following lemma establishes the validity of Assumption \ref{BV^A assumption 2}.
\begin{lemma} \label{density argument for duality lemma - regime gamma} For any $\chi \in C_0(\Omega \times \calY;\M^{3 \times 3}_{\sym})$ with $\widetilde{\div}_{\gamma}\chi(x,y) = 0$ (in the sense of distributions), we can construct an approximating sequence which satisfies \Cref{BV^A assumption 2}. 
\end{lemma}
\begin{proof}
 We take $\chi \in C_0(\Omega \times \calY;\M^{3 \times 3}_{\sym})$, extend it by zero outside $\Omega$ and define 
\begin{equation*}
    \tilde{\chi}^{\epsilon} (x,y) :=  \Lambda_{1+\epsilon} \chi\left(\varphi^{\epsilon}(x') x', (1+\epsilon) x_3, y\right),
\end{equation*}
where $\Lambda_{1+\epsilon}$ is the linear operator described in \eqref{definition Lambda_h}, and $\varphi^{\epsilon}: \omega \to [0,1]$ is a continuous function that is zero in a neighbourhood of $\partial \omega$ and equal to $1$ for $x' \in \omega$ such that $\dist (x',\partial \omega)  \geq \epsilon$. 
Notice that $\tilde{\chi}^{\epsilon} \in C_c(\Omega \times \calY;\M^{3 \times 3}_{\sym})$, $\tilde{\chi}^{\epsilon} \strong \chi$ as $\epsilon \to 0$ in $L^{\infty}$ and $\widetilde{\div}_{\gamma}\tilde{\chi}^{\epsilon} = 0$ (in the sense of distributions). 
The $C^\infty$-regularity of the approximating sequence follows by convolving $\{\tilde{\chi}^{\epsilon}\}$ with a standard sequence of mollifiers.
\end{proof} 
\BBB
\MMM
The following claim establishes the validity of Assumption \ref{BV^A assumption 3}.

\begin{theorem}\label{theoremPoincareKorn} 
There exists a constant $C > 0$ such that
\begin{equation*}
    \left\|u-\int_{I \times \calY} u\right\|_{L^1(I \times \calY;\R^3)} \leq C |\widetilde{E}_{\gamma} u|(I \times \calY)
\end{equation*}
for each function $u \in BD_{\gamma} (I \times \calY)$. The constant $C$ can be chosen independently of $\gamma$ in a fixed interval $[\gamma_1,\gamma_2]$, for $0<\gamma_1<\gamma_2<\infty$. 
\end{theorem}


\begin{proof}
In view of \Cref{scaling from BD to BDgamma},  it is enough to show the claim for the case $\gamma = 1$. \BBB
We argue by contradiction. If the thesis does not hold, then there exists a sequence $\{u_n\}_n \subset BD(I \times \calY)$ such that 
\begin{equation*}
    \int_{I \times \calY} |u_n| \,dx_3 dy > n |\MMM \widetilde{E}_{1} \BBB u_n|(I \times \calY), \;\text{ with }\; \int_{I \times \calY} u_n \,dx_3 dy = 0.
\end{equation*}
We can normalize the sequence such that
\begin{equation*}
    \int_{I \times \calY} |u_n| \,dx_3 dy = 1, \;\text{ and }\; |\MMM \widetilde{E}_{1} \BBB u_n|(I \times \calY) < \frac{1}{n}.
\end{equation*}
In particular the sequence $\{u_n\}$ is bounded in $BD(I \times \calY)$.

By Assumption \ref{BV^A assumption 1}, there exists a subsequence $\{u_m\} \subseteq \{u_n\}$ and a function $u \in BD(I \times \calY)$ such that $\{u_m\}$ converges weakly* to $u$ in $BD(I \times \calY)$, i.e.
\begin{equation*} \label{Korn weak* compactness}
    u_m \strong u \;\text{ in } L^1(I \times \calY;\R^3), \;\text{ and }\; \MMM \widetilde{E}_{1}\BBB u_m \weakstar \MMM \widetilde{E}_{1}\BBB u \;\text{ in } \Mb(I \times \calY;\M^{3 \times 3}_{\sym}).
\end{equation*}
It's clear that the limit satisfies
\begin{equation} \label{Poincare contradiction}
    \int_{I \times \calY} |u| \,dx_3 dy = 1, \;\text{ with }\; \int_{I \times \calY} u \,dx_3 dy = 0.    
\end{equation}
Also, by the weak* lower semicontinuity of the total variation of measures, we have  
\begin{equation} \label{Korn zero total variation}
    |\MMM \widetilde{E}_{1}\BBB u|(I \times \calY) = 0,
\end{equation}
which implies $\MMM \widetilde{E}_{1}\BBB u = 0$. As a result, the limit $u$ is a rigid deformation, i.e. is of the form
\begin{equation*}
    u(x_3,y) = A \begin{pmatrix} y_1 \\ y_2 \\ x_3 \end{pmatrix} + b, \;\text{ where }\; A \in \M_{\skw}^{3 \times 3}, b \in \R^3.
\end{equation*}
Further, \eqref{Korn zero total variation} implies that $u$ has no jumps along $C^1$ hypersurfaces contained in $I \times \calY$. 
Hence, due to the structure of skew-symmetric matrices, $u$ must be a constant vector. 
However, this contradicts with \eqref{Poincare contradiction}.
\end{proof}
\BBB
\MMM 
\begin{remark} \label{poincarekornrem} 
If one doesn't assume periodicity, then the following version of the  Poincar\'{e}-Korn inequality can be proved, using the arguments in the proof of \Cref{theoremPoincareKorn}: 
There exists a constant $C > 0$ such that
\begin{equation*}
     \left\|u- A \begin{pmatrix} x_1 \\ x_2 \\ \gamma x_3 \end{pmatrix} - b\right\|_{L^1((0,1)^2 \times I;\R^3)}  \leq C |{E}_{\gamma} u|((0,1)^2 \times I)
\end{equation*}
for each function $u \in BD_{\gamma} ((0,1)^2 \times I)$ and suitably chosen  $ A \in \M_{\skw}^{3 \times 3}$, $b \in \R^3$, depending on $u$.  Again, the constant $C$ can be chosen independently of $\gamma$ in a fixed interval $[\gamma_1,\gamma_2]$, for $0<\gamma_1<\gamma_2<\infty$.
\end{remark} 
\BBB
The following two propositions are now a consequence of \Cref{BV^A main property} and \Cref{BV^A duality lemma}, respectively.

\begin{proposition} \label{corrector main property - regime gamma}
Let $\mu \in \calXgamma{\omega}$. 
Then there exist $\eta \in \Mb^+(\omega)$ and a Borel map $(x',x_3,y) \in \Omega \times \calY \mapsto \mu_{x'}(x_3,y) \in \R^3$ such that, for $\eta$-a.e. $x' \in \omega$,
\begin{equation} \label{corrector main property 1 - regime gamma}
    \mu_{x'} \in \BDgamma, \qquad \int_{I \times \calY} \mu_{x'}(x_3,y) \,dx_3 dy = 0, \qquad |\widetilde{E}_{\gamma}\mu_{x'}|(I \times \calY) \neq 0,
\end{equation}
and
\begin{equation} \label{corrector main property 2 - regime gamma}
    \mu = \mu_{x'}(x_3,y) \,\eta \otimes \calL^{1}_{x_3} \otimes \calL^{2}_{y}.
\end{equation}
Moreover, the map $x' \mapsto \widetilde{E}_{\gamma}\mu_{x'} \in \Mb(I \times \calY;\M^{3 \times 3}_{\sym})$ is $\eta$-measurable and
\begin{equation*}
    \widetilde{E}_{\gamma}\mu = \eta \genprod \widetilde{E}_{\gamma}\mu_{x'}.
\end{equation*}
\end{proposition}

\begin{proposition} \label{duality lemma - regime gamma}
Let $\lambda \in \Mb(\Omega \times \calY;\M^{3 \times 3}_{\sym})$. The following items are equivalent:
\begin{enumerate}[label=(\roman*)]
	\item For every $\chi \in C_0(\Omega \times \calY;\M^{3 \times 3}_{\sym})$ with $\widetilde{\div}_{\gamma}\chi(x,y) = 0$ (in the sense of distributions) we have
	\begin{equation*}
	    \int_{\Omega \times \calY} \chi(x,y) : d\lambda(x,y) = 0.
	\end{equation*}
    \item There exists $\mu \in \calXgamma{\omega}$ such that $\lambda = \widetilde{E}_{\gamma}\mu$. 
\end{enumerate}		
\end{proposition}

Additionally, we state the following property, which will be used in the proof of \Cref{rank-1 lemma}. 
The proof is analogous to \cite[Proposition 4.7. item (b)]{Francfort.Giacomini.2014}.

\begin{proposition} \label{corrector on C^1-hypersurface - regime gamma}
Let $\mu \in \calXgamma{\omega}$. 
For any $C^1$-hypersurface $\calD \subseteq \calY$, if $\nu$ denotes a continuous unit normal vector field to $\calD$, then
\begin{equation*}
    \widetilde{E}_{\gamma}\mu\mres{\Omega \times \calD} = a(x,y) \odot \nu(y) \,\eta \otimes (\calH^{2}_{x_3,y}\mres{I \times \calD}),
\end{equation*}
where $a : \Omega \times \calD \mapsto \R^3$ is a Borel function.
\end{proposition}

\subsection{Auxiliary results}
We will need the following result, which is connected with the compactly supported De Rham cohomology. \EEE Recall the definitions of $\widetilde{\rot}_\gamma$, $\widetilde{\nabla}_\gamma$, and $\widetilde{\div}_\gamma$. In the next proposition, we will consider the case $\gamma=1$. \BBB

\begin{proposition} \label{auxiliary result - regime gamma}
\begin{enumerate}[label=(\alph*)]
    \item \label{De Rham result 1} Let $\torustridim$ be a flat torus in $\mathbb{R}^3$ and let $\chi \in C^{\infty}(\torustridim;\mathbb{R}^3)$ be such that $\div \chi = 0$ and $\int_{\torustridim} \chi = 0$. 
    Then there exists $F \in C^{\infty}(\torustridim;\mathbb{R}^3)$ such that $\rot F = \chi$.  
    \item \label{De Rham result 2} Let $\calY$ be a flat torus in $\mathbb{R}^2$ and let $\chi \in C_c^{\infty}(I \times \calY;\mathbb{R}^3)$ be such that $\widetilde{div}_{1}\chi = 0$ and $\int_{I \times \calY} \chi = 0$. 
    Then there exists $F \in C_c^{\infty}(I \times \calY;\mathbb{R}^3)$ such that $$\widetilde{\rot}_{1} F  = \chi.$$  
\end{enumerate}
\end{proposition}

\begin{proof}
The first claim is standard and can be easily proved by, e.g, Fourier transforms. 
For the second claim, \MMM observing that $\chi$ is also periodic on $\torustridim$, \BBB by the first part of the statement we obtain $\tilde{F} \in C^{\infty}(\torustridim;\R^3)$ such that $\rot \tilde{F} = \chi$ on $\torustridim$. 
Since $\chi$ has compact support in \MMM $I \times \calY$ \BBB, there exists $0 < \delta < \frac{1}{2}$ such that $\MMM\widetilde{\rot}_{1}\tilde{F} \BBB= 0$ on \MMM $\tilde I_{\delta} \times \calY $ \BBB, where $\tilde I_{\delta} = \{ (\frac{1}{2}-\delta, \frac{1}{2}) \cup (-\frac{1}{2}, -\frac{1}{2}+\delta) \}$. 
Let now $\tilde{\varphi} \in C^{\infty}(S_{\delta})$, where $S_{\delta} = \tilde I_{\delta} \times (0,1)^2 $, be such that $\tilde{F} = \MMM \widetilde{\nabla}_{1}\BBB \tilde{\varphi}$ on $S_{\delta}$. 
For $\alpha\in\{1,2\}$, let 
\begin{equation*}
   \MMM \sum_{k \in \Z} a^{\alpha}_k(x_3,y_2) \BBB e^{2\pi i k y_1}
\end{equation*}
be the exponential Fourier series of $\tilde{F}_\alpha = \partial_{y_\alpha}\tilde{\varphi}$ with respect to the variable $y_1$. 
Note that the coefficients \MMM $\{a^{\alpha}_k(x_3,y_2)\}_{k \in \Z}$ \BBB are smooth functions and periodic with respect to the variable $y_2$ and $x_3$. 
Additionally, the Fourier series of smooth functions converges uniformly, and the result of differentiating or integrating the series term by term will converge to the derivative or integral of the original series.
Hence, we infer that
\MMM \begin{equation} \label{Fourier series 1}
    \tilde{\varphi}(x_3,y) = a^{1}_0(x_3,y_2) y_1 + \sum_{k \in \Z \setminus {\{0\}}} \frac{a^{1}_k(x_3,y_2)}{2\pi i k} e^{2\pi i k y_1} + b^1(x_3,y_2) \ \text{ on } S_{\delta},
\end{equation} \BBB
for a suitable smooth function \MMM $b^1(x_3,y_2)$ \BBB. 
Then, differentiating with respect to $y_1$ and $y_2$, we have that
\MMM\begin{equation*}
    \partial_{y_1 y_2}\tilde{\varphi}(x_3,y) = \partial_{y_2}a^{1}_0(x_3,y_2) + \sum_{k \in \Z \setminus {\{0\}}} \partial_{y_2}a^{1}_k(x_3,y_2) e^{2\pi i k y_1} \ \text{ on } S_{\delta}.
\end{equation*} \BBB
However, since
\MMM
\begin{equation*}
    \partial_{y_1 y_2}\tilde{\varphi}(x_3,y) = \partial_{y_1}\tilde{F}_2(x_3,y) = \sum_{k \in \Z \setminus {\{0\}}} 2\pi i k a^{2}_k(x_3,y_2) e^{2\pi i k y_1} \ \text{ on } S_{\delta},
\end{equation*}
\BBB
by the uniqueness of the Fourier expansion we have that \MMM $\partial_{y_2}a^{1}_0(x_3,y_2) = 0$ \BBB, i.e. 
\begin{equation} \label{Fourier series 2}
    \MMM a^{1}_0(x_3, y_2) \BBB = c_1(x_3),
\end{equation}
for some $c_1 \in C^{\infty}(\tilde I_\delta)$.
Further, differentiating \eqref{Fourier series 1} with respect to $y_2$, we have that
\MMM
\begin{equation*}
    \partial_{y_2}\tilde{\varphi}(x_3,y) = \sum_{k \in \Z \setminus {\{0\}}} \frac{\partial_{y_2}a^{1}_k(x_3,y_2)}{2\pi i k} e^{2\pi i k y_1} + \partial_{y_2}b^1(x_3,y_2) \ \text{ on } S_{\delta}.
\end{equation*}
\BBB
Since $\partial_{y_2}\tilde{\varphi} = \tilde{F}_2$ is periodic, we  conclude that $\partial_{y_2}b^1$ is also periodic with respect to the variable $y_2$ and we can consider its Fourier series. 
Let $c_2 \in C^{\infty}(\tilde I_\delta)$ be the corresponding zero-th term. 
Then
the antiderivative of $\partial_{y_2}b^1 - c_2 $ with respect to $y_2$ is a periodic function.
Combining this fact with \eqref{Fourier series 1} and \eqref{Fourier series 2}, we deduce that there exists a smooth function $\hat{\varphi} \in C^{\infty}(\tilde I_{\delta};C^{\infty}(\calY))$ such that $\tilde{\varphi}$ can be rewritten as
\MMM
\begin{equation*}
    \tilde{\varphi}(x_3,y) = \hat{\varphi}(x_3,y) + c_1(x_3)\,y_1 + c_2(x_3)\,y_2 \ \text{ on }  \tilde I_{\delta} \times \calY.
\end{equation*}
\BBB
From this, differentiating with respect to $x_3$, we have that
\MMM \begin{equation*}
    \tilde{F}_3(x_3,y) = \partial_{x_3}\hat{\varphi}(x_3,y) + c_1'(x_3)\,y_1 + c_2'(x_3)\,y_2 \ \text{ on } \tilde I_{\delta} \times \calY.
\end{equation*} \BBB
As a consequence of the periodicity of $\tilde{F}_3$ and $\partial_{x_3}\hat{\varphi}$ in the variables $y_1$ and $y_2$, we conclude that $c_1' = 0$ and $c_2' = 0$. 
Since \MMM$ \tilde I_{\delta} \times \calY $ \BBB  is a union of two disjoint open sets, we have that $c_1, c_2$ are constant on each connected component. 
Using the fact that, for $\alpha\in\{1,2\}$, 
\MMM \begin{equation} \label{Fourier series 3} 
    \partial_{y_\alpha}\tilde{\varphi}(x_3,y) = \partial_{y_\alpha}\hat{\varphi}(x_3,y) + c_{\alpha}(x_3) \ \text{ on }  \tilde I_{\delta} \times \calY,
\end{equation} \BBB
the periodicity of $\tilde{F}_\alpha = \partial_{y_\alpha}\tilde{\varphi}$ implies that $c_1, c_2$ are in fact constant. 
This can be seen by integrating the equation \eqref{Fourier series 3} over the plane $x_3 = -\frac{1}{2}$ and $x_3 = \frac{1}{2}$. 
Thus we conclude that
\MMM \begin{equation} \label{Fourier series 4}
    \tilde{F}(x_3,y) = \widetilde{\nabla}_1\hat{\varphi}(x_3,y) + \left(\begin{array}{c} c_1 \\ c_2 \\ 0 \end{array}\right) \ \text{ on } \tilde I_{\delta}\times \calY.
\end{equation} \BBB
Consider now the exponential Fourier series of $\tilde{F}_3$ with respect to the $x_3$ variable, such that
\MMM \begin{equation*}
    \tilde{F}_3(x_3,y) = \sum_{k \in \Z} a^{3}_k(y) e^{2\pi i k x_3} \ \text{ on }  \tilde I_{\delta} \times \calY.
\end{equation*}\BBB
Integrating the third component in \eqref{Fourier series 4} with respect to $x_3$, we have that there exists a smooth function \MMM $b^3(x_3,y)$ \BBB, which has values $b^3_{+}(y)$ and $b^3_{-}(y)$ on each of the two parts of \MMM $\tilde I_{\delta} \times \calY$ \BBB, such that
\MMM \begin{equation*}
    \hat{\varphi}(x_3,y) = a^{3}_0(y) x_3 + \sum_{k \in \Z \setminus {\{0\}}} \frac{a^{3}_k(y)}{2\pi i k} e^{2\pi i k x_3} + b^3(x_3,y) \ \text{ on } \tilde I_{\delta} \times \calY.
\end{equation*} \BBB
From this and \eqref{Fourier series 3} we have, for $\alpha\in\{1,2\}$,
\MMM \begin{equation*}
    \tilde{F}_\alpha(x_3,y) - c_\alpha = \partial_{y_\alpha}a^{3}_0(y) x_3 + \sum_{k \in \Z \setminus {\{0\}}} \frac{\partial_{y_\alpha}a^{3}_k(y)}{2\pi i k} e^{2\pi i k x_3} + \partial_{y_\alpha}b^3(x_3,y) \ \text{ on } \tilde I_{\delta} \times \calY.
\end{equation*} \BBB
Considering the continuity and periodicity in $x_3$ of the above terms, evaluating in $x_3 = -\frac{1}{2}$ and $x_3 = \frac{1}{2}$ gives $\partial_{y_\alpha}a^{3}_0(y) = \partial_{y_\alpha}b^3_{-}(y) - \partial_{y_\alpha}b^3_{+}(y)$.
From this we have that there exists a constant $c_3$ and a map $\varphi \in C^{\infty}(\calY \times \tilde I_{\delta})$ such that $\varphi$ and all its derivatives are periodic in the $x_3$ variable, and for which 
\MMM \begin{equation*} \label{Fourier series 5}
    \hat{\varphi}(x_3,y) = \varphi(x_3,y) + c_3 x_3 \ \text{ on }  \tilde I_{\delta} \times \calY.
\end{equation*} \BBB
From this and \eqref{Fourier series 4} we conclude that
\MMM \begin{equation*}
    \tilde{F}(x_3,y) = \widetilde{\nabla}_1\varphi(x_3,y) + \left(\begin{array}{c} c_1 \\ c_2 \\ c_3 \end{array}\right) \ \text{ on } \tilde I_{\delta} \times \calY.
\end{equation*}
\BBB
Finally, we consider a smooth function $k : I \to \R$ that is zero on the set $\left[ -\frac{1}{2}+\delta, \frac{1}{2}-\delta \right]$ and one in a neighbourhood of $x_3 = -\frac{1}{2}$, $x_3 = \frac{1}{2}$. 
By taking 
\begin{equation*}
    F := \tilde{F} - \MMM \widetilde{\nabla}_1\BBB (k\,\varphi) - \left(\begin{array}{c} c_1 \\ c_2 \\ c_3 \end{array}\right) \ \text{ on } \MMM I \times \calY .\BBB
\end{equation*}
we have the claim. 
\end{proof}
\BBB
\begin{remark} \label{extensiongamma}
By considering functions scaled by $\gamma$ in the third component and by $\frac{1}{\gamma}$ in the direction $x_3$, one can apply the proof item (b) in \Cref{auxiliary result - regime gamma} so that the statement is valid for maps in the space $C_c^{\infty}(\MMM (\gamma I) \times \calY \BBB;\mathbb{R}^3)$.

Consequently, for $\chi \in C_c^{\infty}(\MMM I \times \calY \BBB;\mathbb{R}^3)$ such that $\widetilde{\div}_{\gamma}\chi = 0$ and $\int_{\MMM I \times \calY\BBB} \chi = 0$ there exists $F \in C_c^{\infty}(\MMM I \times \calY\BBB;\mathbb{R}^3)$ such that $\widetilde{\rot}_{\gamma} F = \chi$, which can be easily seen by rescaling in the direction $x_3$.
\end{remark}

\begin{remark} \label{rmk third column zero}
If $\chi \in C_c^{\infty}(\Omega \times \calY;\M^{3 \times 3}_{\sym})$ is such that $\widetilde{\div}_{\gamma}\chi = 0$, then for a.e. $x' \in \omega$
\begin{equation*}
    \int_{I \times \calY} \chi_{3i}(x,y) \,dx_3 dy = 0, \quad i=1,2,3.
\end{equation*}
Indeed, by putting
\begin{equation*}
    \varphi(x) = \left(\begin{array}{c} 2\gamma\,x_3\,c_1(x') \\ 2\gamma\,x_3\,c_2(x') \\ \gamma\,x_3\,c_3(x') \end{array}\right),
\end{equation*}
for $c \in C_c^{\infty}(\omega;\R^3)$, we infer that 
\begin{equation*}
   \widetilde{E}_{\gamma}\varphi = \left(\begin{array}{ccc} 0 & 0 & c_1 \\ 0 & 0 & c_2 \\ c_1 & c_2 & c_3 \end{array} \right),
\end{equation*}
and the conclusion results from testing $\chi$ with $\widetilde{E}_{\gamma}\varphi$ on $I \times \calY$, and by the arbitrariness of the maps $c_i$,\, $i = 1, 2, 3$.
\end{remark}

\subsection{Two-scale limits of scaled symmetrized gradients}\label{scaled sym gradients section}
We are now ready to prove the first main result of this section.

\begin{theorem} \label{two-scale weak limit of scaled strains}
\MMM Let $\{u^h\}_{h>0} \subset BD(\Omega)$ be a sequence such that there exists a constant $C>0$ for which $$\|u^h\|_{L^1(\Omega;\R^3)}+\|\Lambda_h Eu^h\|_{\mathcal{M}_b(\Omega; \mathbb M^{3 \times 3}_{sym}) } \leq C. $$ \BBB
Then there exist $\bar{u} = (\bar{u}_1, \bar{u}_2) \in BD(\omega)$, $u_3 \in BH(\omega)$ and $\mu \in \calXgamma{\omega}$, and a subsequence of $\{u^h\}_{h>0}$, not relabeled, which satisfy:
\begin{eqnarray*}
    \Lambda_h Eu^h 
    \weakstartwoscale 
    \begin{pmatrix} E \bar{u} - x_3 D^2u_3 & 0 \\ 0 & 0 \end{pmatrix} \otimes \calL^{2}_{y}
    + 
    \widetilde{E}_{\gamma}\mu 
    \quad \text{two-scale weakly* in $\Mb(\Omega \times \calY;\M^{3 \times 3}_{\sym})$}.
\end{eqnarray*}
\end{theorem}

\begin{proof}
Owing to \cite[Chapter II, Remark 3.3]{Temam.1985}, we can assume without loss of generality that the maps $u^h$ are smooth functions for every $h > 0$.
Further, the uniform boundedness of the sequence $\{Ev^h\}$ implies that
\begin{align}
    \label{recall C h} &\int_{\Omega} |\partial_{x_\alpha}u^h_3+\partial_{x_3}u^h_\alpha| \,dx \leq C h, \quad \text{ for } \alpha = 1, 2,\\
    \label{recall C h^2} &\int_{\Omega} |\partial_{x_3}u^h_3| \,dx \leq C h^2.
\end{align}

In the following, we will consider $\lambda \in \Mb(\Omega \times \calY;\M^{3 \times 3}_{\sym})$ such that 
\begin{eqnarray*}
    \Lambda_h Eu^h \weakstartwoscale \lambda \quad \text{two-scale weakly* in $\Mb(\Omega \times \calY;\M^{3 \times 3}_{\sym})$}.
\end{eqnarray*}

By using \Cref{two-scale weak limit of scaled strains - 2x2 submatrix} we have that there exist $(\bar{u}_1, \bar{u}_2) \in BD(\omega)$, $u_3 \in BH(\omega)$ such that 
\begin{equation*}
    (Eu^h)^{\prime\prime}(x) \,\weakstar\, E\bar{u}(x') - x_3 D^2u_3(x') \quad \text{weakly* in }\Mb(\Omega;\M^{2 \times 2}_{\sym}).
\end{equation*}
Let $\chi \in C_c^{\infty}(\Omega \times \calY;\M^{3 \times 3}_{\sym})$ be such that $\widetilde{\div}_{\gamma}\chi = 0$. 
We have 
\begin{align}
    \nonumber & \int_{\Omega \times \calY} \chi(x,y) : d\lambda(x,y)\\ 
    \nonumber &= \lim_{h \to 0} \int_{\Omega} \chi\hspace{-0.25em}\left(x,\tfrac{x'}{\epsh}\right) : d\left(\Lambda_h Eu^h(x)\right)
    \nonumber = - \lim_{h \to 0} \int_{\Omega} u^h(x) \cdot \div\left( \Lambda_h \chi\hspace{-0.25em}\left(x,\tfrac{x'}{\epsh}\right)\right) \,dx\\
    \nonumber &= 
    - \lim_{h \to 0} \sum_{\alpha=1,2} \int_{\Omega}u^h_{\alpha}(x)\,(\partial_{x_1}\chi_{\alpha 1}+\partial_{x_2}\chi_{\alpha 2})\hspace{-0.25em}\left(x,\tfrac{x'}{\epsh}\right) \,dx
    - \lim_{h \to 0}\frac{1}{h} \int_{\Omega} u^h_3(x)\,(\partial_{x_1}\chi_{31}+\partial_{x_2}\chi_{32})\hspace{-0.25em}\left(x,\tfrac{x'}{\epsh}\right) \,dx\\
    \nonumber &\,\quad - \lim_{h \to 0} \sum_{\alpha=1,2} \frac{1}{\epsh} \int_{\Omega}u^h_{\alpha}(x)\,(\partial_{y_1}\chi_{\alpha 1}+\partial_{y_2}\chi_{\alpha 2})\hspace{-0.25em}\left(x,\tfrac{x'}{\epsh}\right) \,dx
    - \lim_{h \to 0}\frac{1}{h \epsh} \int_{\Omega} u^h_3(x)\,(\partial_{y_1}\chi_{31}+\partial_{y_2}\chi_{32})\hspace{-0.25em}\left(x,\tfrac{x'}{\epsh}\right) \,dx\\
    \nonumber &\,\quad - \lim_{h \to 0} \sum_{\alpha=1,2} \frac{1}{h} \int_{\Omega} u^h_{\alpha}(x)\,\partial_{x_3}\chi_{\alpha 3}\hspace{-0.25em}\left(x,\tfrac{x'}{\epsh}\right) \,dx 
    - \lim_{h \to 0} \frac{1}{h^2} \int_{\Omega} u^h_3(x)\,\partial_{x_3}\chi_{33}\hspace{-0.25em}\left(x,\tfrac{x'}{\epsh}\right) \,dx\\
    \begin{split} \label{eq0}
        &= - \lim_{h\to0} \sum_{\alpha=1,2} \int_{\Omega}u^h_{\alpha} \cdot (\partial_{x_1} \chi_{\alpha 1}+\partial_{x_2} \chi_{\alpha 2})\hspace{-0.25em}\left(x,\tfrac{x'}{\epsh}\right) \,dx
        - \lim_{h\to0}\frac{1}{h} \int_{\Omega} u^h_3 \cdot  (\partial_{x_1} \chi_{31}+\partial_{x_2} \chi_{32})\hspace{-0.25em}\left(x,\tfrac{x'}{\epsh}\right) \,dx\\
        &\,\quad + \lim_{h\to0} \left(\frac{h}{\epsh \gamma}-1\right) \left( \sum_{\alpha=1,2} \frac{1}{h} \int_{\Omega} u^h_{\alpha} \cdot \partial_{x_3} \chi_{\alpha 3}\hspace{-0.25em}\left(x,\tfrac{x'}{\epsh}\right) \,dx + \frac{1}{h^2} \int_{\Omega} u^h_3 \cdot \partial_{x_3} \chi_{33}\hspace{-0.25em}\left(x,\tfrac{x'}{\epsh}\right) \,dx \right),
    \end{split}
\end{align}
where in the last equality we used that $\frac{1}{\epsh}\partial_{y_1}\chi_{i1}+\frac{1}{\epsh}\partial_{y_2}\chi_{i2}+\frac{1}{h}\partial_{x_3}\chi_{i3} = \left(\frac{1}{h}-\frac{1}{\epsh \gamma}\right) \partial_{x_3}\chi_{i3}$.

From Proposition \ref{two-scale weak limit of scaled strains - 2x2 submatrix} we know that we have the following convergences:
\begin{align*}
    u^h_{\alpha} &\strong \bar{u}_{\alpha}-x_3 \partial_{x_\alpha}u_3, \quad \text{strongly in }L^1(\Omega), \quad \alpha=1,2,\\
    u^h_3 &\strong u_3, \quad \text{strongly in }L^1(\Omega).  
\end{align*} 
Notice that
\begin{align} \label{eqelisa}
    \nonumber & \lim_{h \to 0} \sum_{\alpha=1,2} \int_{\Omega}u^h_{\alpha}(x)\,(\partial_{x_1}\chi_{\alpha 1}+\partial_{x_2}\chi_{\alpha 2})\hspace{-0.25em}\left(x,\tfrac{x'}{\epsh}\right) \,dx\\ 
    \nonumber &= \sum_{\alpha=1,2} \int_{\Omega} (\bar{u}_{\alpha}-x_3 \partial_{x_\alpha}u_3)\,\left(\partial_{x_1}\int_{\calY } \chi_{\alpha 1}(x,y) \,dy+\partial_{x_2}\int_{\calY } \chi_{\alpha 2}(x,y) \,dy\right) \,dx\\ 
    &= - \int_{\Omega \times \calY} \chi(x,y) : d\left(\begin{pmatrix} \uKL & 0\\ 0 & 0 \end{pmatrix} \otimes \calL^{2}_{y}\right).
\end{align}
Next, in view of Remark \ref{rmk third column zero}, we can use item \ref{De Rham result 2} in \Cref{auxiliary result - regime gamma}, i.e. \Cref{extensiongamma} to conclude that there exists $F \in C_c^{\infty}(\Omega \times \calY;\R^3)$ such that $\widetilde{\rot}_{\gamma} F= (\chi_{3i})_{i=1,2,3}$. Thus we have
\begin{eqnarray}
    \label{eq1}	\chi_{31} &=& \partial_{y_2}F_3 - \frac{1}{\gamma}\partial_{x_3}F_2,\\
    \label{eq2} \chi_{32} &=& \frac{1}{\gamma}\partial_{x_3}F_1 - \partial_{y_1}F_3.
\end{eqnarray}
Next we compute
\begin{eqnarray}
  \nonumber \lim_{h \to 0}\frac{1}{\epsh} \int_{\Omega} u^h_3(x)\,\partial_{x_1 y_2}F_3\hspace{-0.25em}\left(x,\tfrac{x'}{\epsh}\right) \,dx &=&  \lim_{h \to 0} \int_{\Omega} u^h_3(x)\,\partial_{x_2}\left( \partial_{x_1}F_3\hspace{-0.25em}\left(x,\tfrac{x'}{\epsh}\right)\right)dx\\ \label{eq3}& & - \lim_{h \to 0} \int_{\Omega} u^h_3(x)\,\partial_{x_1 x_2}F_3\hspace{-0.25em}\left(x,\tfrac{x'}{\epsh}\right)dx.
\end{eqnarray}                         
Notice that 
\begin{equation} \label{eq4} 
 \lim_{h \to 0} \int_{\Omega} u^h_3(x)\,\partial_{x_1 x_2}F_3\hspace{-0.25em}\left(x,\tfrac{x'}{\epsh}\right)=\int_{\Omega \times \calY} u_3\,\partial_{x_1 x_2}F_3 (x,y) dx dy =\int_{\Omega} \partial_{x_1 x_2}u_3\,\int_{\calY } F_3 (x,y) dy dx. 
 \end{equation} 
Recalling \eqref{recall C h}, we find
\begin{eqnarray}
\nonumber \lim_{h \to 0} \int_{\Omega} u^h_3(x)\,\partial_{x_2}\left( \partial_{x_1}F_3\hspace{-0.25em}\left(x,\tfrac{x'}{\epsh}\right)\right)dx &=& - \lim_{h \to 0} \int_{\Omega} \partial_{x_2}u^h_3(x)\,\partial_{x_1}F_3\hspace{-0.25em}\left(x,\tfrac{x'}{\epsh}\right)dx \\ \nonumber &=&\lim_{h \to 0}\int_{\Omega} \partial_{x_3}u^h_2\,\partial_{x_1}F_3\hspace{-0.25em}\left(x,\tfrac{x'}{\epsh}\right)dx \\ \nonumber &=& - \lim_{h \to 0}\int_{\Omega}  u^h_2\,\partial_{x_1x_3}F_3\hspace{-0.25em}\left(x,\tfrac{x'}{\epsh}\right)dx\\ \nonumber 
&=& -\int_{\Omega \times \calY} (\bar{u}_2-x_3 \partial_{x_2}u_3)\,\partial_{x_1 x_3}F_3 (x,y) dx dy \\ \label{eq5}  &=& \int_{\Omega} \partial_{x_1 x_2}u_3\,\int_{\calY } F_3 (x,y) dy dx.
\end{eqnarray}
From \eqref{eq3}, \eqref{eq4}, \eqref{eq5} we infer
\begin{eqnarray} \label{eq6} 
 \nonumber \lim_{h \to 0}\frac{1}{h} \int_{\Omega} u^h_3(x)\,\partial_{x_1 y_2}F_3\hspace{-0.25em}\left(x,\tfrac{x'}{\epsh}\right) \,dx
 &=& \lim_{h \to 0}\frac{1}{\epsh\gamma} \int_{\Omega} u^h_3(x)\,\partial_{x_1 y_2}F_3\hspace{-0.25em}\left(x,\tfrac{x'}{\epsh}\right) \,dx\\
 &=& 0. 
\end{eqnarray} 
In a similar way for $u^h_3$ (recalling \eqref{recall C h^2}), we deduce 
\begin{eqnarray} \label{eq7}
    \nonumber \lim_{h \to 0}\frac{1}{h} \int_{\Omega} u^h_3(x)\,\partial_{x_1 x_3}F_2\hspace{-0.25em}\left(x,\tfrac{x'}{\epsh}\right) \,dx
    &=& - \lim_{h \to 0} \frac{1}{h} \int_{\Omega} \partial_{x_3}u^h_3(x)\,\partial_{x_1}F_2\hspace{-0.25em}\left(x,\tfrac{x'}{\epsh}\right) \,dx\\ 
    &=&  0. 
\end{eqnarray}
From \eqref{eq1}, \eqref{eq6}, \eqref{eq7} we conclude that
\begin{equation} \label{eq8} 
    \lim_{h \to 0}\frac{1}{h} \int_{\Omega} u^h_3(x)\,\partial_{x_1}\chi_{31}\hspace{-0.25em}\left(x,\tfrac{x'}{\epsh}\right) \,dx = 0. 
\end{equation} 
Analogously, we obtain 
\begin{equation} \label{eq9} 
    \lim_{h \to 0}\frac{1}{h} \int_{\Omega} u^h_3(x)\,\partial_{x_2}\chi_{32}\hspace{-0.25em}\left(x,\tfrac{x'}{\epsh}\right) \,dx = 0.
\end{equation}
Lastly, using similar arguments as above, we compute
\begin{align} \label{eq10}
    & \nonumber \lim_{h \to 0} \left(\frac{h}{\epsh\gamma}-1\right) \left( \sum_{\alpha=1,2} \frac{1}{h} \int_{\Omega} u^h_{\alpha}(x)\,\partial_{x_3}\chi_{\alpha 3}\hspace{-0.25em}\left(x,\tfrac{x'}{\epsh}\right) \,dx + \frac{1}{h^2} \int_{\Omega} u^h_3(x)\,\partial_{x_3}\chi_{33}\hspace{-0.25em}\left(x,\tfrac{x'}{\epsh}\right) \,dx \right)\\
    &= \nonumber \lim_{h \to 0} \left(\frac{h}{\epsh\gamma}-1\right) \left( - \sum_{\alpha=1,2} \frac{1}{h} \int_{\Omega} \partial_{x_3}u^h_{\alpha}(x)\,\chi_{\alpha 3}\hspace{-0.25em}\left(x,\tfrac{x'}{\epsh}\right) \,dx + \frac{1}{h^2} \int_{\Omega} u^h_3(x)\,\partial_{x_3}\chi_{33}\hspace{-0.25em}\left(x,\tfrac{x'}{\epsh}\right) \,dx \right)\\
    &= \nonumber \lim_{h \to 0} \left(\frac{h}{\epsh\gamma}-1\right) \left( \sum_{\alpha=1,2} \frac{1}{h} \int_{\Omega} \partial_{x_\alpha}u^h_3(x)\,\chi_{\alpha 3}\hspace{-0.25em}\left(x,\tfrac{x'}{\epsh}\right) \,dx + \frac{1}{h^2} \int_{\Omega} u^h_3(x)\,\partial_{x_3}\chi_{33}\hspace{-0.25em}\left(x,\tfrac{x'}{\epsh}\right) \,dx \right)\\
    &= \nonumber \lim_{h \to 0} \left(\frac{h}{\epsh\gamma}-1\right) \left( - \frac{1}{h} \int_{\Omega} u^h_3(x)\,(\partial_{x_1}\chi_{31}+\partial_{x_2}\chi_{32})\hspace{-0.25em}\left(x,\tfrac{x'}{\epsh}\right) \,dx +  \left(\frac{h}{\epsh\gamma}+1\right) \frac{1}{h^2} \int_{\Omega} u^h_3(x)\,\partial_{x_3}\chi_{33}\hspace{-0.25em}\left(x,\tfrac{x'}{\epsh}\right) \,dx \right)\\
    &= 0.
\end{align}

From \eqref{eq0}, \eqref{eqelisa}, \eqref{eq8}, \eqref{eq9}, \eqref{eq10} we have that
\begin{equation*}
    \int_{\Omega \times \calY} \chi(x,y) : d\left(\lambda(x,y) - \begin{pmatrix} \uKL & 0\\ 0 & 0 \end{pmatrix} \otimes \calL^{2}_{y}\right) = 0.
\end{equation*}
From this and Proposition \ref{duality lemma - regime gamma} we find that there exists $\mu \in \calXgamma{\omega}$ such that 
\begin{equation*}
    \lambda - \begin{pmatrix} E\bar{u} - x_3 D^2u_3 & 0\\ 0 & 0 \end{pmatrix} \otimes \calL^{2}_{y} = \widetilde{E}_{\gamma}\mu.
\end{equation*}
This, in turn, yields the claim.
\end{proof}

\subsection{Unfolding adapted to dimension reduction}
We proceed along the lines of \cite[Section 4.3]{Francfort.Giacomini.2014}. 

For every $\eps > 0$ and $i \in \Z^2$, let
\begin{equation*}
    Q_\eps^i := \left\{ x \in \R^2 : \frac{x-\eps i}{\eps} \in Y \right\}.
\end{equation*}
Given an open set $\omega \subseteq \R^2$, we will set
\begin{equation*}
    I_\eps(\omega) := \left\{ i \in \Z^2 : Q_\eps^i \subset \omega \right\}.
\end{equation*}
Given $\mu_\eps \in \Mb(\omega \times I)$ and $Q_\eps^i \subset \omega$, we define $\mu_\eps^i \in \Mb(\MMM I \times \calY \BBB)$ such that
\begin{equation*}
    \int_{\MMM I \times \calY \BBB} \psi(\MMM x_3,y\BBB) \,d\mu_\eps^i(\MMM x_3,y\BBB) = \frac{1}{\eps^2} \int_{Q_\eps^i \times I} \psi\left(\BLUE x_3,\frac{x'}{\eps}\BLACK\right) \,d\mu_\eps(x), \quad \psi \in C(\MMM I \times \calY \BBB).
\end{equation*}
\begin{definition}
For every $\eps > 0$, \emph{the unfolding measure} associated with $\mu_\eps$ is the measure $\tilde{\lambda}_\eps \in \Mb(\omega \times \MMM I \times \calY \BBB)$ defined by
\begin{equation*}
    \tilde{\lambda}_\eps := \sum_{i \in I_\eps(\omega)} \left( \calL_{x'}^2\mres{Q_\eps^i} \right) \otimes \mu_\eps^i.
\end{equation*}
\end{definition}

The following proposition provides the relationship between the two-scale weak* convergence and unfolding measures.
The proof is analogous to \cite[Proposition 4.11.]{Francfort.Giacomini.2014}.

\begin{proposition} \label{unfolding measure weak* convergence}
Let $\omega \subseteq \R^2$ be an open set and let $\{\mu_\eps\} \subset \Mb(\omega \times I)$ be a bounded family such that
\begin{equation*}
    \mu_\eps \weakstartwoscale \mu_0 \quad \text{two-scale weakly* in $\Mb(\omega \times \MMM I \times \calY \BBB)$}.
\end{equation*}
Let $\{\tilde{\lambda}_\eps\} \subset \Mb(\omega \times \MMM I \times \calY \BBB)$ be the family of unfolding measures associated with $\{\mu_\eps\}$. 
Then 
\begin{equation*}
    \tilde{\lambda}_\eps \weakstar \mu_0 \quad \text{weakly* in $\Mb(\omega \times \MMM I \times \calY \BBB)$}.
\end{equation*}
\end{proposition}

To analyze the sequences of symmetrized scaled gradients of $BD$ function in the context of unfolding, we will need to consider the following auxiliary spaces
\begin{gather*}
    BD_{\frac{h}{\eps}}(\MMM I \times \calY \BBB) := \Big\{ u \in L^1(\BLUE I \times \calY\BLACK;\R^3) : \widetilde{E}_{\frac{h}{\eps}}u \in \Mb(\MMM I \times \calY \BBB;\M^{3 \times 3}_{\sym}) \Big\},\\
    BD_{\frac{h}{\eps}}\left((0, 1)^2 \times I\right) := \Big\{ u \in L^1\left((0, 1)^2 \times I;\R^3\right) : E_{\frac{h}{\eps}}u \in \Mb\left((0, 1)^2 \times I;\M^{3 \times 3}_{\sym}\right) \Big\},
\end{gather*}
where $\widetilde{E}_{\frac{h}{\eps}}$ and $E_{\frac{h}{\eps}}$ denote the distributional symmetrized scaled gradients, \MMM cf. \eqref{defsymmscgrad}. \BBB
Similarly as in \Cref{scaling from BD to BDgamma}, scaling in the the first two components shows that these auxiliary spaces are equivalent to the usual $BD$ space on the appropriate domain.

\begin{proposition} \label{associated unfolding measure}
Let $\omega \subseteq \R^2$ be an open set and let $\calB \subseteq \calY$ be an open set with Lipschitz boundary. 
Let $\gamma_0 \in (0,1]$ and let $h,\, \eps > 0$ be such that
\[
    \gamma_0 \leq \frac{h}{\eps} \leq \frac{1}{\gamma_0}.
\]
If $u_\eps \in BD(\omega \times I)$, the unfolding measure associated with $\Lambda_h Eu_\eps\mres{(\calB_\eps \setminus \calC_\eps) \times I}$ is given by
\begin{equation} \label{unfolding symmetrized gradient}
    \sum_{i \in I_\eps(\omega)} \left( \calL_{x'}^2\mres{Q_\eps^i} \right) \otimes \widetilde{E}_{\frac{h}{\eps}}\hat{u}_{h,\eps}^i\mres{\MMM I \times (\calB \setminus \calC) \BBB },
\end{equation}
where $\hat{u}_{h,\eps}^i \in BD_{\frac{h}{\eps}}(\MMM I \times \calY \BBB)$ is such that
\MMM
\begin{equation} \label{unfolding inequality}
   \int_{I \times \calB} |\hat{u}_{h,\eps}^i| \,dx_3\,dy+ \int_{\MMM I \times \partial\calB \BBB} |\hat{u}_{h,\eps}^i| \,d\calH^{2} + |\widetilde{E}_{\frac{h}{\eps}}\hat{u}_{h,\eps}^i|\left(\MMM I \times (\calB \cap \calC) \BBB \right) \leq \frac{C}{\eps^2} |\Lambda_h Eu_\eps|\left(int(Q_\eps^i) \times I\right),
\end{equation}
\BBB
for some constant $C$ independent of $i$, $h$ and $\eps$.
\end{proposition}

\begin{proof}
Since $\calB_\eps$ has Lipschitz boundary, $u_\eps \charfun{\calB_\eps \times I} \in BD_{loc}(\omega \times I)$ with
\begin{equation*}
    Eu_\eps\mres{\calB_\eps \times I} = E\left(u_\eps \charfun{\calB_\eps \times I}\right) + \left[u_\eps\mres{\partial\calB_\eps \times I} \odot \nu\right] \calH^{2}\mres{\partial\calB_\eps \times I},
\end{equation*}
where $u_\eps\mres{\partial\calB_\eps \times I}$ denotes the trace of $u_\eps \charfun{\calB_\eps \times I}$ on $\partial\calB_\eps \times I$, while $\nu$ is the exterior normal to $\partial\calB_\eps \times I$. We note that the third component of $\nu$ is equal to zero.

Remark that $\calC_\eps = \left( \cup _i \partial Q_\eps^i\right) \cap \omega$. 
Accordingly, for $i \in I_\eps(\omega)$ and $\psi \in C^1(I \times \calY;\M^{3 \times 3}_{\sym})$,
\begin{align*}
    &\int_{Q_\eps^i \times I} \psi\left(\BLUE x_3,\frac{x'}{\eps}\BLACK\right) : d\left(\Lambda_h Eu_\eps\mres{(\calB_\eps \setminus \calC_\eps) \times I}\right)(x)
    = \int_{int(Q_\eps^i) \times I} \psi\left(\BLUE x_3,\frac{x'}{\eps}\BLACK\right) : d\left(\Lambda_h Eu_\eps\mres{\calB_\eps \times I}\right)(x)\\
    &= \int_{int(Q_\eps^i) \times I} \psi\left(\BLUE x_3,\frac{x'}{\eps}\BLACK\right) : d\Lambda_h E\left(u_\eps \charfun{\calB_\eps \times I}\right)(x)\\
    &\,\quad + \int_{int(Q_\eps^i) \times I} \psi\left(\BLUE x_3,\frac{x'}{\eps}\BLACK\right) : \Lambda_h \left[u_\eps\mres{\partial\calB_\eps \times I} \odot \nu\right] \,d\calH^{2}\mres{\partial\calB_\eps \times I}(x).
\end{align*}
We set $v_{h,\eps}^i(\zz) := \diag\left(1,1,\frac{1}{h}\right)\,u_\eps(\eps i+\eps \zz', \zz_3)$ for $\zz \in (0, 1)^2 \times I$. 
Then $v_{h,\eps}^i \in BD_{\frac{h}{\eps}}\left((0, 1)^2 \times I\right)$, and $E_{\frac{h}{\eps}}v_{h,\eps}^i(\zz) = \eps \Lambda_h Eu_\eps(\eps i+\eps \zz', \zz_3)$. 
Performing a change of variables, we find
\begin{align*}
    &\int_{Q_\eps^i \times I} \psi\left(\BLUE x_3,\frac{x'}{\eps}\BLACK\right) : d\left(\Lambda_h Eu_\eps\mres{(\calB_\eps \setminus \calC_\eps) \times I}\right)(x)\\
    &= \eps \int_{(0, 1)^2 \times I} \psi\left(\BLUE x_3,x'\BLACK\right) : dE_{\frac{h}{\eps}}\left(v_{h,\eps}^i \charfun{\calI(\calB) \times I}\right)(\zz)\\
    &\,\quad + \eps \int_{(0, 1)^2 \times I} \psi\left(\BLUE x_3,x'\BLACK\right) : \Lambda_h \left[\diag(1,1,h)\,v_{h,\eps}^i\mres{\calI(\partial\calB) \times I} \odot \nu\right] \,d\calH^{2}(\zz)\\
    &= \eps \int_{(0, 1)^2 \times I} \psi\left(\BLUE x_3,x'\BLACK\right) : dE_{\frac{h}{\eps}}\left(v_{h,\eps}^i \charfun{\calI(\calB) \times I}\right)(\zz) + \eps \int_{(0, 1)^2 \times I} \psi\left(\BLUE x_3,x'\BLACK\right) : \left[v_{h,\eps}^i\mres{\calI(\partial\calB) \times I} \odot \nu\right] \,d\calH^{2}(\zz).
\end{align*}
Notice that we can assume that
\MMM
\begin{equation*}
   \int_{(0, 1)^2 \times I}|v_{h,\eps}^i|dx+ \int_{\partial(0, 1)^2 \times I} |v_{h,\eps}^i\mres{\partial(0, 1)^2 \times I}| \,d\calH^{2} \leq C |E_{\frac{h}{\eps}}v_{h,\eps}^i|\left((0, 1)^2 \times I\right) = \frac{C}{\eps} |\Lambda_h Eu_\eps|\left(int(Q_\eps^i) \times I\right),
\end{equation*}
\BBB
for some constant $C$ independent of $i$, $h$ and $\eps$. 
\MMM This can be achieved by using \Cref{poincarekornrem} since subtracting a rigid deformation \BBB to $u_\eps$ on $Q_\eps^i \times I$ corresponds to subtracting an element of the kernel of $E_{\frac{h}{\eps}}$  to $v_{h,\eps}^i$, which does not modify the calculations done thus far. 
Hence, by the trace theorem and Poincar\'{e}-Korn's inequality in $BD\left((0, 1)^2 \times I\right)$, we get the desired inequality.

Defining $\hat{u}_{h,\eps}^i(\MMM x_3,y \BBB) := \frac{1}{\eps} v_{h,\eps}^i\left(\BLUE \calI(y),x_3 \BLACK\right)$, we obtain
\begin{align*}
     |\widetilde{E}_{\frac{h}{\eps}}\hat{u}_{h,\eps}^i|\left(\MMM I \times \calY \BBB \right) &\leq \int_{\MMM I \times \calC\BBB} |\hat{u}_{h,\eps}^i\mres{\MMM I \times \calC\BBB }| \,d\calH^{2} + |\widetilde{E}_{\frac{h}{\eps}}\hat{u}_{h,\eps}^i|\left(\MMM I \times (\calY \setminus \calC) \BBB\right)\\
     &= \frac{1}{\eps} \int_{\partial(0, 1)^2 \times I} |v_{h,\eps}^i\mres{\partial(0, 1)^2 \times I}| \,d\calH^{2} + \frac{1}{\eps} |E_{\frac{h}{\eps}}v_{h,\eps}^i|\left((0, 1)^2 \times I\right)\\
     &\leq \frac{C+1}{\eps} |E_{\frac{h}{\eps}}v_{h,\eps}^i|\left((0, 1)^2 \times I\right) = \frac{C+1}{\eps^2} |\Lambda_h Eu_\eps|\left(int(Q_\eps^i) \times I\right).
\end{align*}
Furthermore,
\begin{equation*}
    \eps \int_{(0, 1)^2 \times I} \psi : dE_{\frac{h}{\eps}}\left(v_{h,\eps}^i \charfun{\calI(\calB) \times I}\right) = \eps^2 \int_{\MMM I \times (\calY \setminus \calC)\BBB} \psi : d\widetilde{E}_{\frac{h}{\eps}}\left(\hat{u}_{h,\eps}^i \charfun{\calB \times I}\right)
\end{equation*}
and
\begin{equation*}
    \eps \int_{(0, 1)^2 \times I} \psi : \left[v_{h,\eps}^i\mres{\MMM \calI(\partial\calB) \times I \BBB} \odot \nu\right] \,d\calH^{2} = \eps^2 \int_{\MMM I \times (\calY \setminus \calC)} \psi : \left[\hat{u}_{h,\eps}^i\mres{\MMM I \times (\partial\calB \setminus \calC) \BBB} \odot \nu\right] \,d\calH^{2}.
\end{equation*}
So we have
\begin{align*}
    &\frac{1}{\eps^2} \int_{Q_\eps^i \times I} \psi\left(\BLUE x_3,\frac{x'}{\eps}\BLACK\right) : d\left(\Lambda_h Eu_\eps\mres{(\calB_\eps \setminus \calC_\eps) \times I}\right)(x)\\
    &= \int_{\MMM I \times (\calY \setminus \calC) \BBB} \psi(\MMM x_3,y\BBB) : d\widetilde{E}_{\frac{h}{\eps}}\left(\hat{u}_{h,\eps}^i \charfun{\calB \times I}\right)(y, x_3)\\
    &\,\quad + \int_{\MMM I \times (\calY \setminus \calC)\BBB} \psi(\MMM x_3,y\BBB) : \left[\hat{u}_{h,\eps}^i\mres{\MMM I \times (\partial\calB \setminus \calC) \BBB} \odot \nu\right] \,d\calH^{2}(\MMM x_3,y \BBB)\\
    &= \int_{\MMM I \times \calY\BBB} \psi(\MMM x_3,y \BBB) : d\widetilde{E}_{\frac{h}{\eps}}\hat{u}_{h,\eps}^i\mres{\MMM I \times (\calB \setminus \calC)\BBB }(\MMM x_3,y \BBB),
\end{align*}
from which \eqref{unfolding symmetrized gradient} follows. 
It remains to prove \eqref{unfolding inequality}. 
Again, up to adding \MMM an affine transformation  to $\hat{u}_{h,\eps}^i$ (cf. \Cref{poincarekornrem}) \BBB on $\MMM I \times \calB \BBB$, we can assume
\MMM 
\begin{align*}
    &\int_{\MMM I \times \calB \BBB} |\hat{u}_{h,\eps}^i|\,dx_3dy+\int_{\MMM I \times \partial\calB \BBB} |\hat{u}_{h,\eps}^i| \,d\calH^{2} + |\widetilde{E}_{\frac{h}{\eps}}\hat{u}_{h,\eps}^i|\left(\MMM I \times (\calB \cap \calC) \BBB\right)\\
    &\leq C |\widetilde{E}_{\frac{h}{\eps}}\hat{u}_{h,\eps}^i|\left(\MMM I \times \calB \BBB\right) + |\widetilde{E}_{\frac{h}{\eps}}\hat{u}_{h,\eps}^i|\left(\MMM I \times (\calB \cap \calC)\BBB \right) \leq C |\widetilde{E}_{\frac{h}{\eps}}\hat{u}_{h,\eps}^i|\left(\MMM I \times \calY \BBB \right)\\
    &\leq \frac{C}{\eps^2} |\Lambda_h Eu_\eps|\left(int(Q_\eps^i) \times I\right).
\end{align*}
\BBB
This concludes the proof of the theorem.
\end{proof}

As a consequence of Proposition \ref{associated unfolding measure}, we deduce the following lemma, which in turn will be used in the proof of the lower semicontinuity of $\calH^{hom}$ in \Cref{Lower semicontinuity of energy functionals}.

\begin{lemma} \label{rank-1 lemma}
Let $\calB \subseteq \calY$ be an open set with Lipschitz boundary, such that $\partial\calB \setminus \calT$ is a $C^1$-hypersurface, for some compact set $\calT$ with $\calH^{1}(\calT) = 0$. 
Additionally, assume that $\partial\calB \cap \calC \subseteq \calT$. 
Let $v^h \in BD({{\Omega}})$ be such that 
\begin{equation*}
    v^h \weakstar v \quad \text{weakly* in $BD({{\Omega}})$}
\end{equation*}
 and
\begin{equation*}
    \Lambda_h Ev^h\mres{{{\Omega}} \cap (\calB_\epsh \times I)} \weakstartwoscale \pi \quad \text{two-scale weakly* in $\Mb({{\Omega}} \times \calY;\M^{3 \times 3}_{\sym})$}.
\end{equation*}
Then $\pi$ is supported in ${{\Omega}} \times \bar{\calB}$ and
\begin{equation} \label{rank-1 structure}
    \pi\mres{{{\Omega}} \times (\partial\calB \setminus \calT)} = a(x,y) \odot \nu(y) \,\zeta,
\end{equation}
where $\zeta \in \Mb^+ ({{\Omega}} \times (\partial\calB \setminus \calT))$, $a : {{\Omega}} \times (\partial\calB \setminus \calT) \to \R^3$ is a Borel map, and $\nu$ is the exterior normal to $\partial\calB$.
\end{lemma}

\begin{proof}
Denote by $\tilde{\pi} \in \Mb({{\Omega}} \times \calY;\M^{3 \times 3}_{\sym})$ the two-scale weak* limit (up to a subsequence) of
\begin{equation*}
    \Lambda_h Ev^h\mres{{{\Omega}} \cap ((\calB_\epsh \setminus \calC_\epsh) \times I)} \in \Mb({{\Omega}};\M^{3 \times 3}_{\sym}).
\end{equation*}
Then it is enough to prove the analogue of \eqref{rank-1 structure} for $\tilde{\pi}$. 
Indeed, the two-scale weak* limit (up to a subsequence) of
\begin{equation*}
    \Lambda_h Ev^h\mres{{{\Omega}} \cap ((\calB_\epsh \cap \calC_\epsh) \times I)} \in \Mb({{\Omega}};\M^{3 \times 3}_{\sym})
\end{equation*}
is supported on ${{\Omega}} \times \closure{\calB \cap \calC}$. 
Since by assumption $\partial\calB \cap \calC \subseteq \calT$, we have that $\partial\calB \setminus \calT$ and $\closure{\calB \cap \calC}$ are disjoint sets, which implies
\begin{equation*}
    \pi\mres{{{\Omega}} \times (\partial\calB \setminus \calT)} = \tilde{\pi}\mres{{{\Omega}} \times (\partial\calB \setminus \calT)}.
\end{equation*}
By Theorem \ref{associated unfolding measure}, the unfolding measure associated with $\Lambda_h Ev^h\mres{(\calB_\epsh \setminus \calC_\epsh) \times I}$ is given by
\begin{equation} \label{unfolding symmetrized gradient 2}
    \sum_{i \in I_\epsh({{\omega}})} \left( \calL_{x'}^2\mres{Q_\epsh^i} \right) \otimes \widetilde{E}_{\frac{h}{\epsh}}\hat{v}_{\epsh}^i\mres{\MMM I \times (\calB \setminus \calC)\BBB},
\end{equation}
where $\hat{v}_{\epsh}^i \in BD(\MMM I \times \calY\BBB)$ is such that
\MMM
\begin{equation} \label{unfolding inequality 2}
    \int_{I \times \calB } |\hat{v}_{\epsh}^i|\,dx_3dy+\int_{I \times \partial\calB } |\hat{v}_{\epsh}^i| \,d\calH^{2} + |\widetilde{E}_{\frac{h}{\epsh}}\hat{v}_{\epsh}^i|\left(\MMM I \times (\calB \cap \calC)\BBB\right) \leq \frac{C}{\epsh^2} |\Lambda_h Ev^h|\left(int(Q_\epsh^i) \times I\right).
\end{equation}
\BBB
Further, by Theorem \ref{unfolding measure weak* convergence}, the family of associated measures in \eqref{unfolding symmetrized gradient 2} converge weakly* to $\tilde{\pi}$ in $\Mb({{\Omega}} \times \calY;\M^{3 \times 3}_{\sym})$. 
Then, for every $\chi \in C_c^{\infty}({{\Omega}} \times \calY;\M^{3 \times 3}_{\sym})$ with $\widetilde{\div}_{\gamma}\chi(x,y) = 0$, we get
\begin{align} 
    \nonumber&\int_{{{\Omega}} \times \calY} \chi(x, y) : d\tilde{\pi}(x, y)\\
    \nonumber&= \lim\limits_{h} \int_{{{\Omega}} \times \calY} \chi(x, y) : d\left(\sum_{i \in I_\epsh({{\omega}})} \left( \calL_{x'}^2\mres{Q_\epsh^i} \right) \otimes \widetilde{E}_{\frac{h}{\epsh}}\hat{v}_{\epsh}^i\mres{\MMM I \times (\calB \setminus \calC)\BBB}\right)\\
    \nonumber&= \lim\limits_{h} \sum_{i \in I_\epsh({{\omega}})} \int_{Q_\epsh^i} \left( \int_{\MMM I \times (\calB \setminus \calC) \BBB} \chi(x, y) : d\widetilde{E}_{\frac{h}{\epsh}}\hat{v}_{\epsh}^i \right) \,dx'\\
    \nonumber&= \lim\limits_{h} \sum_{i \in I_\epsh({{\omega}})} \int_{Q_\epsh^i} \left( \int_{\MMM I \times \calB\BBB} \chi(x, y) : d\widetilde{E}_{\frac{h}{\epsh}}\hat{v}_{\epsh}^i - \int_{\MMM I \times (\calB \cap \calC)\BBB} \chi(x, y) : d\widetilde{E}_{\frac{h}{\epsh}}\hat{v}_{\epsh}^i \right) \,dx'.
\end{align}
\MMM   By the integration by parts formula for $BD$ functions over $I \times \calB$ we have \BBB
\begin{align*} 
    \nonumber&\int_{{{\Omega}} \times \calY} \chi(x, y) : d\tilde{\pi}(x, y)\\
    \nonumber&= \lim\limits_{h} \sum_{i \in I_\epsh({{\omega}})} \int_{Q_\epsh^i} \Bigg( - \int_{\MMM I \times \calB \BBB} \widetilde{\div}_{\frac{h}{\epsh}} \chi(x,y) \cdot \hat{v}_{\epsh}^i(\MMM x_3,y \BBB) \MMM \,dx_3dy \BBB+ \int_{\MMM I \times \partial\calB \BBB } \chi(x, y) : \left[\hat{v}_{\epsh}^i(\MMM x_3,y \BBB) \odot \nu\right] \,d\calH^{2}(\MMM x_3,y \BBB)\\\nonumber&\hspace{8.1em} - \int_{\MMM I \times (\calB \cap \calC)\BBB} \chi(x, y) : d\widetilde{E}_{\frac{h}{\epsh}}\hat{v}_{\epsh}^i \Bigg) \,dx'\\
    \nonumber&= \lim\limits_{h} \sum_{i \in I_\epsh({{\omega}})} \int_{Q_\epsh^i} \Bigg( -\left(\frac{\epsh}{h}-\frac{1}{\gamma}\right) \int_{\MMM I \times \calB \BBB } \partial_{x_3}\chi(x,y) \cdot \hat{v}_{\epsh}^i(y,x_3) \MMM\,dx_3dy\BBB \\&\hspace{8.1em} + \int_{\MMM I \times \partial\calB \BBB} \chi(x, y) : \left[\hat{v}_{\epsh}^i(y,x_3) \odot \nu\right] \,d\calH^{2}(y,x_3) - \int_{(\calB \cap \calC) \times I} \chi(x, y) : d\widetilde{E}_{\frac{h}{\epsh}}\hat{v}_{\epsh}^i \Bigg) \,dx'.
\end{align*}

Owing to  \eqref{unfolding inequality 2}, we conclude that 
the \MMM the sum $$\sum_{i \in I_\epsh({{\omega}})} \int_{Q_\epsh^i} \int_{I \times \calB} \partial_{x_3}\chi(x,y) \cdot \hat{v}_{\epsh}^i(y,x_3) \,dx_3dy $$  is \BBB finite.
Further, in view of \eqref{unfolding inequality 2} we can rewrite 
the above limit as
\begin{align} \label{unfolding limit 2}
    \int_{{{\Omega}} \times \calY} \chi(x, y) : d\tilde{\pi}(x, y)
    = \lim\limits_{h} \left( \int_{{{\Omega}} \times \calY} \chi(x, y) : d\lambda^h_1(x, y) + \int_{{{\Omega}} \times \calY} \chi(x, y) : d\lambda^h_2(x, y) \right),
\end{align}
with $\lambda^h_1,\, \lambda^h_2 \in \Mb({{\Omega}} \times \calY;\M^{3 \times 3}_{\sym})$, such that (up to a subsequence)
\begin{equation*}
    \lambda^h_1 \weakstar \lambda_1 \;\text{ and }\; \lambda^h_2 \weakstar \lambda_2 \quad \text{weakly* in $\Mb({{\Omega}} \times \calY;\M^{3 \times 3}_{\sym})$}
\end{equation*}
for suitable $\lambda_1, \lambda_2 \in \Mb({{\Omega}} \times \calY;\M^{3 \times 3}_{\sym})$. 
Then, we have $\supp(\lambda_1) \subseteq {{\Omega}} \times \partial\calB$ and $\supp(\lambda_2) \subseteq {{\Omega}} \times (\closure{\calB \cap \calC})$.

By the density argument described in \Cref{density argument for duality lemma - regime gamma}, we conclude that \eqref{unfolding limit 2} holds for every $\chi \in C_0({{\Omega}} \times \calY;\M^{3 \times 3}_{\sym})$ with $\widetilde{\div}_{\gamma}\chi = 0$. 
The definition of $\lambda_1$ and $\lambda_2$ then yields
\begin{equation*}
    \int_{{{\Omega}} \times \calY } \chi(x,y) : d\left(\tilde{\pi} - \lambda_1 - \lambda_2\right)(x,y) = 0.
\end{equation*}
Thus, from \Cref{duality lemma - regime gamma} we conclude that there exists $\mu \in \calXgamma{{{\omega}}}$ such that 
\begin{equation*}
    \tilde{\pi}-\lambda_1-\lambda_2 = \widetilde{E}_{\gamma}\mu.
\end{equation*}
Recalling the assumption that $\partial\calB \cap \calC \subseteq \calT$ and using the same argument as above, we obtain
\begin{equation*}
    \tilde{\pi}\mres{{{\Omega}} \times (\partial\calB \setminus \calT)} = \lambda_1\mres{{{\Omega}} \times (\partial\calB \setminus \calT)} + \widetilde{E}_{\gamma}\mu\mres{{{\Omega}} \times (\partial\calB \setminus \calT)}
\end{equation*}
In view of \Cref{corrector on C^1-hypersurface - regime gamma} and recalling the assumption that $\partial\calB \setminus \calT$ is a $C^1$-hypersurface, we are left to prove the analogue of \eqref{rank-1 structure} for $\lambda_1$.

We consider
\begin{equation*}
    \hat{v}^h(x,y) = \sum_{i \in I_\epsh({{\omega}})} \charfun{Q_\epsh^i}(x')\,\hat{v}_{\epsh}^i(\MMM x_3,y \BBB),
\end{equation*}
so that $\lambda^h_1(x,y) = \left[\hat{v}^h(\MMM x_3,y \BBB) \odot \nu\right]\calL_{x'}^2 \otimes (\calH^{2}_{x_3,y}\mres{I \times \partial\calB})$. 
Then $\{\hat{v}^h\}$ is bounded in $L^1({{\Omega}} \times \partial\calB;\R^3)$ by \eqref{unfolding inequality 2}. 
Up to a subsequence,
\begin{equation*}
    \hat{v}^h \,\calL_{x'}^2 \otimes (\calH^{2}_{x_3,y}\mres{I \times \partial\calB}) \weakstar \eta \quad \text{weakly* in $\Mb({{\Omega}} \times \partial\calB;\R^3)$}
\end{equation*}
for a suitable $\eta \in \Mb({{\Omega}} \times \partial\calB;\R^3)$. Since $\nu$ is continuous on $\partial\calB \setminus \calT$, we infer
\begin{equation*}
    \lambda_1\mres{{{\Omega}} \times (\partial\calB \setminus \calT)} = \frac{\eta}{|\eta|}(x,y) \odot \nu(y) \,|\eta|\mres{{{\Omega}} \times (\partial\calB \setminus \calT)},
\end{equation*}
which concludes the proof, \MMM since $\frac{\eta}{|\eta|}$ is a Borel function. \BBB
\end{proof}

\section{Two-scale statics and duality}
\label{statics}
In this section we define a notion of stress-strain duality and analyze the two-scale behavior of our functionals.
\subsection{Stress-plastic strain duality on the cell}
\begin{definition} \label{definition K_gamma}
Let $\gamma \in (0,+\infty)$. 
The set $\calK_{\gamma}$ of admissible stresses is defined as the set of all elements $\Sigma \in L^2(I \times \calY;\M^{3 \times 3}_{\sym})$ satisfying:
\begin{enumerate}[label=(\roman*)]
    \item \label{definition K_gamma (i)} $\widetilde{\div}_{\gamma}\Sigma = 0 \text{ in } I \times \calY$,
    \item \label{definition K_gamma (ii)} $\Sigma\,\vec{e}_3 = 0 \text{ on } \partial{I} \times \calY$,
    \item \label{definition K_gamma (iii)} $\Sigma_{\dev}(x_3,y) \in K(y) \,\text{ for } \calL^{1}_{x_3} \otimes \calL^{2}_{y}\text{-a.e. } (x_3,y) \in I \times \calY$.
\end{enumerate}
\end{definition}

Since condition \ref{definition K_gamma (iii)} implies that $\Sigma_{\dev} \in L^{\infty}(I \times \calY;\M^{3 \times 3}_{\sym})$, for every $\Sigma \in \calK_{\gamma}$ we deduce from \Cref{Kohn-Temam embedding lemma} that $\Sigma \in L^p(I \times \calY;\M^{3 \times 3}_{\sym})$ for every $1 \leq p < \infty$.

\begin{definition} \label{definition A_gamma}
Let $\gamma \in (0,+\infty)$. 
The family $\calA_{\gamma}$ of admissible configurations is given by the set of triplets
\begin{equation*}
    u \in \BDgamma, \qquad E \in L^2(I \times \calY;\M^{3 \times 3}_{\sym}), \qquad P \in \Mb(I \times \calY;\M^{3 \times 3}_{\dev}),
\end{equation*}
such that
\begin{equation*}
    \widetilde{E}_{\gamma}u = E \,\calL^{1}_{x_3} \otimes \calL^{2}_{y} + P \quad \textit{ in } I \times \calY.
\end{equation*}
\end{definition}

\begin{definition}
\label{def:dist}
Let $\Sigma \in \calK_{\gamma}$ and let 
$(u, E, P) \in \calA_{\gamma}$.
We define the distribution $[ \Sigma_{\dev} : P ]$ on $\R \times \calY$ by
\begin{equation} \label{cell stress-strain duality - regime gamma}
    [ \Sigma_{\dev} : P ](\varphi) := 
    - \int_{I \times \calY} \varphi\,\Sigma : E \,dx_3 dy 
    - \int_{I \times \calY} \Sigma : \big( u \odot \widetilde{\nabla}_{\gamma}\varphi \big) \,dx_3 dy,
\end{equation}
for every $\varphi \in C_c^{\infty}(\R \times \calY)$.
\end{definition}

\begin{remark}
Note that the second integral in \eqref{cell stress-strain duality - regime gamma} is well defined since $BD(I \times \calY)$ is embedded into $L^{3/2}(I \times \calY;\R^3)$. 
Moreover, the definition of $[ \Sigma_{\dev} : P ]$ is independent of the choice of $(u, E)$, so \eqref{cell stress-strain duality - regime gamma} defines a meaningful distribution on $\R \times \calY$.
\end{remark}

The following results can be established from the proofs of \cite[Theorem 6.2]{Francfort.Giacomini.2012} and \cite[Proposition 3.9]{Francfort.Giacomini.2012} respectively, by treating the relative boundary of the "Dirichlet" part as empty, the "Neumann" part as $\partial{I} \times \calY$, and considering approximating sequences which must be periodic in $\calY$.

\begin{proposition} \label{duality distribution is actually a measure - regime gamma}
Let $\Sigma \in \calK_{\gamma}$ and $(u, E, P) \in \calA_{\gamma}$.
Then $[ \Sigma_{\dev} : P ]$ can be extended to a bounded Radon measure on $\R \times \calY$, whose variation satisfies
\begin{equation*}
    | [ \Sigma_{\dev} : P ] | \leq \| \Sigma_{\dev} \|_{L^{\infty}(I \times \calY;\M^{3 \times 3}_{\sym})} |P| 
    \quad \text{ in } \Mb(\R \times \calY).
\end{equation*}
\end{proposition}

\begin{proposition} \label{cell Hill's principle - regime gamma}
Let $\Sigma \in \calK_{\gamma}$ and $(u, E, P) \in \calA_{\gamma}$. 
If $\calY$ is a geometrically admissible multi-phase torus, then
\begin{equation*}
    H\left(y, \frac{dP}{d|P|}\right)\,|P| \geq [ \Sigma_{\dev} : P ] 
    \quad \text{ in } \Mb(I \times \calY).
\end{equation*}
\end{proposition}

\subsection{Disintegration of admissible configurations}
\label{subs:dis}
Let $\ext{\omega} \subseteq \R^2$ be an open and bounded set such that $\omega \subset \ext{\omega}$ and $\ext{\omega} \cap \partial{\omega} = \gamma_\Dir$. 
We also denote by $\ext{\Omega} = \ext{\omega} \times I$ the associated reference domain.

In order to make sense of the duality between the two-scale limits of stresses and plastic strains, we will need to disintegrate the two-scale limits of the
kinematically admissible fields in such a way to obtain elements 
of $\calA_{\gamma}$, for $\gamma \in (0,+\infty)$. 
\begin{definition} \label{definition A^hom_gamma}
Let $w \in H^1(\ext{\Omega};\R^3) \cap KL(\ext{\Omega})$. 
We define the class $\calA^{hom}_{\gamma}(w)$ of admissible two-scale configurations relative to the boundary datum $w$ as the set of triplets $(u,E,P)$ with
\begin{equation*}
    u \in KL(\ext{\Omega}), \qquad E \in L^2(\ext{\Omega} \times \calY;\M^{3 \times 3}_{\sym}), \qquad P \in \Mb(\ext{\Omega} \times \calY;\M^{3 \times 3}_{\dev}),
\end{equation*}
such that
\begin{equation*}
    u = w, \qquad E = Ew, \qquad P = 0 \qquad \text{ on } (\ext{\Omega} \setminus \closure{\Omega}) \times \calY,
\end{equation*}
and also such that there exists $\mu \in \calXgamma{\ext{\omega}}$ with
\begin{equation} \label{admissible two-scale configurations - regime gamma}
    Eu \otimes \calL^{2}_{y} + \widetilde{E}_{\gamma}\mu = E \,\calL^{3}_{x} \otimes \calL^{2}_{y} + P \qquad \text{ in } \ext{\Omega} \times \calY.
\end{equation}
\end{definition}

\begin{lemma} \label{disintegration result - regime gamma}
Let $(u,E,P) \in \calA^{hom}_{\gamma}(w)$ with the associated $\mu \in \calXgamma{\ext{\omega}}$, and let $\bar{u} \in BD(\ext{\omega})$ and $u_3 \in BH(\ext{\omega})$ be the Kirchhoff-Love components of $u$. 
Set
\begin{equation*}
    \eta := \calL^{2}_{x'} + (\proj_{\#}|P|)^s \in \Mb^+(\ext{\omega}).
\end{equation*}
Then the following disintegrations hold true:
\begin{align}
    \label{disintegration result 1 - regime gamma} Eu \otimes \calL^{2}_{y} &= \begin{pmatrix} A_1(x') + x_3 A_2(x') & 0 \\ 0 & 0 \end{pmatrix} \eta \otimes \calL^{1}_{x_3} \otimes \calL^{2}_{y},\\
    \label{disintegration result 2 - regime gamma} E \,\calL^{3}_{x} \otimes \calL^{2}_{y} &= C(x') E(x,y) \,\eta \otimes \calL^{1}_{x_3} \otimes \calL^{2}_{y}\\
    \label{disintegration result 3 - regime gamma} P &= \eta \genprod P_{x'}.
\end{align}
Above, $A_1, A_2 : \ext{\omega} \to \M^{2 \times 2}_{\sym}$ and $C : \ext{\omega} \to [0, +\infty]$ are  Radon-Nikodym derivatives of $E\bar{u}$, $-D^2u_3$ and $\calL^{2}_{x'}$ with respect to $\eta$, $E(x,y)$ is a Borel representative of $E$, and $P_{x'} \in \Mb(I \times \calY;\M^{3 \times 3}_{\dev})$ for $\eta$-a.e. $x' \in \ext{\omega}$.

Furthermore, we can choose a Borel map $(x',x_3,y) \in \ext{\Omega} \times \calY \mapsto \mu_{x'}(x_3,y) \in \R^3$ such that, for $\eta$-a.e. $x' \in \ext{\omega}$,
\begin{equation} \label{disintegration result 4 - regime gamma} 
    \mu = \mu_{x'}(x_3,y) \,\eta \otimes \calL^{1}_{x_3} \otimes \calL^{2}_{y}, \quad \widetilde{E}_{\gamma}\mu = \eta \genprod \widetilde{E}_{\gamma}\mu_{x'},
\end{equation}
where $\mu_{x'} \in \BDgamma$, $\int_{I \times \calY} \mu_{x'}(x_3,y) \,dx_3 dy = 0$.
\end{lemma}

\begin{proof}
The proof is analogous to \cite[Lemma 5.4]{Francfort.Giacomini.2014}. 
The only difference is the statement and argument for the disintregration of $Eu \otimes \calL^{2}_{y}$, that we detail below.

First we note that $\proj_{\#}\left(\widetilde{E}_{\gamma}\mu\right)_{\alpha \beta} = \proj_{\#}\left(E_{y}\mu\right)_{\alpha \beta} = 0$ for $\alpha,\beta=1,2$. 
Then, from \eqref{admissible two-scale configurations - regime gamma} we get
\begin{align*}
    \left(E\bar{u}\right)_{\alpha \beta} = \proj_{\#}\left(Eu \otimes \calL^{2}_{y}\right)_{\alpha \beta} &= \left( \int_{I \times \calY} E_{\alpha \beta}(x,y) \,dx_3 dy \right) \calL^{2}_{x'} + \proj_{\#}(P)_{\alpha \beta}\\
    &\leq e^{(1)}_{\alpha \beta}(x') \,\calL^{2}_{x'} + (\proj_{\#}|P|)^s_{\alpha \beta},
\end{align*}
where we set $e^{(1)}(x') :=\int_{I \times \calY} \MMM|E(x,y)| \BBB\,dx_3 dy  + (\proj_{\#}|P|)^a \in L^2(\ext{\omega};\M^{3 \times 3}_{\sym})$. 
Similarly, after multipliying equation \eqref{admissible two-scale configurations - regime gamma} by $x_3$, we have that
\begin{align*}
    \left(-D^2u_3\right)_{\alpha \beta} = \frac{1}{12}\,\proj_{\#}\left(x_3 Eu \otimes \calL^{2}_{y}\right)_{\alpha \beta} &= \frac{1}{12}\,\left( \int_{I \times \calY} x_3 E_{\alpha \beta}(x,y) \,dx_3 dy \right) \calL^{2}_{x'} + \frac{1}{12}\,\proj_{\#}(x_3 P)\\
    &\leq e^{(2)}_{\alpha \beta}(x') \,\calL^{2}_{x'} + \frac{1}{12}\,(\proj_{\#}|x_3 P|)^s_{\alpha \beta},
\end{align*}
where we set $e^{(2)}(x') := \frac{1}{12}\,\int_{I \times \calY} \MMM |x_3 E(x,y) |\BBB\,dx_3 dy  + \frac{1}{12}\,(\proj_{\#}|x_3 P|)^a \in L^2(\ext{\omega};\M^{3 \times 3}_{\sym})$. 
Consequently, the measures $E\bar{u}$ and $-D^2u_3$ are absolutely continuous with respect to $\eta$, so we find
\begin{align*}
     E\bar{u} \otimes \calL^{2}_{y} &= A_1(x') \,\eta \otimes \calL^{1}_{x_3} \otimes \calL^{2}_{y},\\
     -D^2u_3 \otimes \calL^{2}_{y} &= A_2(x') \,\eta \otimes \calL^{1}_{x_3} \otimes \calL^{2}_{y},
\end{align*}
for suitable $A_1, A_2 : \ext{\omega} \to \M^{2 \times 2}_{\sym}$ such that \eqref{disintegration result 1 - regime gamma} hold true.
\end{proof}

\begin{remark} \label{admissible configurations and disintegration - regime gamma}
From the above disintegration, we have that, for $\eta$-a.e. $x' \in \ext{\omega}$, 
\begin{equation*}
    \widetilde{E}_{\gamma}\mu_{x'} = \left[ C(x') E(x,y) - \begin{pmatrix} A_1(x') + x_3 A_2(x') & 0 \\ 0 & 0 \end{pmatrix} \right] \calL^{1}_{x_3} \otimes \calL^{2}_{y} + P_{x'} \quad \textit{ in } I \times \calY.
\end{equation*}
Thus, the triple
\begin{equation*}
    \left( \mu_{x'}, \left[ C(x') E(x,y) - \begin{pmatrix} A_1(x') + x_3 A_2(x') & 0 \\ 0 & 0 \end{pmatrix} \right], P_{x'} \right)
\end{equation*}
is an element of $\calA_{\gamma}$.
\end{remark}

\subsection{Admissible stress configurations and approximations}
For every $e^h \in L^2(\Omega;\M^{3 \times 3}_{\sym})$ we define $\sigma^h(x) := \C\left(\frac{x'}{\epsh}\right) \Lambda_h e^h(x)$. Then, in view of \cite[Theorem 3.6]{Francfort.Giacomini.2012}, we introduce the set
\begin{align*}
    \calK_h = \bigg\{\sigma^h &\in L^2(\Omega;\M^{3 \times 3}_{\sym}) : \div_{h}\sigma^h = 0 \text{ in } \Omega,\ \sigma^h\,\nu = 0 \text{ in } \partial\Omega \setminus {\closure{\Gamma}_\Dir},\\
    &\sigma^h_{\dev}(x',x_3) \in K\left(\frac{x'}{\epsh}\right) \,\text{ for a.e. } x' \in \omega,\, x_3 \in I\bigg\},
\end{align*}
\MMM which is the set of stresses for the rescaled $h$ problems. \BBB
\MMM Next we introduce the set of two-scale limiting stresses. 
\begin{definition} \label{definition K^hom_gamma}
The set $\calK^{hom}_{\gamma}$ is the set of all elements $\Sigma \in L^2(\Omega \times \calY;\M^{3 \times 3}_{\sym})$ satisfying:
\begin{enumerate}[label=(\roman*)]
    \item $\widetilde{\div}_{\gamma}\Sigma(x',\cdot) = 0 \text{ in } I \times \calY \,\text{ for  a.e. } x' \in \omega$,
    \item $\Sigma(x',\cdot)\,\vec{e}_3 = 0 \text{ on } \partial{I} \times \calY \,\text{ for  a.e. } x' \in \omega$,
    \item $\Sigma_{\dev}(x,y) \in K(y) \,\text{ for } \calL^{3}_{x} \otimes \calL^{2}_{y}\text{-a.e. } (x,y) \in \Omega \times \calY$,
    \item $\sigma_{i3}(x) = 0 \,\text{ for } i=1,2,3$,
    \item $\div_{x'}\bar{\sigma} = 0 \text{ in } \omega$,
    \item $\div_{x'}\div_{x'}\hat{\sigma} = 0 \text{ in } \omega$,
\end{enumerate}
where $\sigma := \int_{\calY} \Sigma(\cdot,y) \,dy$, and $\bar{\sigma},\, \hat{\sigma} \in L^2(\omega;\M^{2 \times 2}_{\sym})$ are the zero-th and first order moments of the $2 \times 2$ minor of $\sigma$.
\end{definition}
\begin{remark} 
Notice that as a consequence of the properties (iii) and (iv) in the \Cref{definition K^hom_gamma} we can actually conclude that $\bar{\sigma},\, \hat{\sigma} \in L^\infty(\omega;\M^{2 \times 2}_{\sym})$. Namely, the uniform boundedness of sets $K(y)$ implies that the deviatoric part of the weak limit, i.e. $\sigma_{\dev} = \sigma - \frac{1}{3} \tr{\sigma} I_{3 \times 3}$, is bounded in $L^\infty(\Omega;\M^{3 \times 3}_{\sym})$. Thus we have that
\begin{equation*}
    \begin{pmatrix} 
    \sigma_{11} & \sigma_{12} & 0\\ 
    \sigma_{12} & \sigma_{22} & 0\\ 
    0 & 0 & 0 
    \end{pmatrix} 
    - \frac{1}{3} 
    \begin{pmatrix} 
    \sigma_{11}+\sigma_{22} & 0 & 0\\ 
    0 & \sigma_{11}+\sigma_{22} & 0\\ 
    0 & 0 & \sigma_{11}+\sigma_{22} 
    \end{pmatrix} 
    \text{ is bounded in  } L^\infty(\Omega;\M^{3 \times 3}_{\sym}).
\end{equation*}
Hence, the components $\sigma_{\alpha \beta}$ are all bounded in $L^\infty(\Omega)$.
\end{remark} 

\BBB
In the following proposition we show that the set $\calK_\gamma^{\rm hom}$ characterizes weak two-scale limits of sequences of elastic stresses $\{\sigma^h\}$.

\begin{proposition} \label{two-scale weak limit of admissible stress - regime gamma}
Let $\{\sigma^h\}$ be a bounded family in $L^2(\Omega;\M^{3 \times 3}_{\sym})$ such that $\sigma^h \in \calK_h$ for every $h$, and
\begin{equation*}
    \sigma^h \weaktwoscale \Sigma \quad \text{two-scale weakly in $L^2(\Omega \times \calY;\M^{3 \times 3}_{\sym})$}.
\end{equation*}
Then $\Sigma \in \calK^{hom}_{\gamma}$.
\end{proposition}
\MMM
\begin{proof}
Consider a sequence $\{\sigma^h\} \subset L^2(\Omega;\M^{3\times 3}_{\rm sym})$ such that $\sigma_h \in \calK_h$ for every $h$, and assume that $\sigma^h \rightharpoonup \sigma$ weakly in $L^2(\Omega;\M^{3\times 3}_{\rm sym})$. We first establish the macroscopic properties (iv), (v), (vi). 
To obtain (iv), let $v \in C_c^{\infty}(\Omega;\R^3)$ and $V \in C^{\infty}(\closure{\Omega};\R^3)$ be defined by
\begin{equation*}
    V(x',x_3) := \int_{-\frac{1}{2}}^{x_3} v(x',\zeta) \,d\zeta.
\end{equation*}
From the condition $\div_{h}\sigma^h = 0 \text{ in } \Omega$, for every $\varphi \in H^1(\Omega;\R^3)$ with $\varphi = 0$ on $\Gamma_\Dir$ we have
\begin{equation} \label{div sigma^h = 0}
    \int_{\Omega} \sigma^h(x) : E_{h}\varphi(x) \,dx = 0.
\end{equation}
Setting
\begin{equation*}
    \varphi(x) = \left(\begin{array}{c} 2h\,V_1(x) \\ 2h\,V_2(x) \\ h\,V_3(x) \end{array}\right),
\end{equation*}
and passing to the limit as $h\to 0$, we find
\begin{equation*}
   \int_{\Omega} \sigma(x) : \left(\begin{array}{ccc} 0 & 0 & v_1(x) \\ 0 & 0 & v_2(x) \\ v_1(x) & v_2(x) & v_3(x) \end{array} \right) \,dx = \int_{\Omega} \sigma(x) : \left(\begin{array}{ccc} 0 & 0 & \partial_{x_3}V_1(x) \\ 0 & 0 & \partial_{x_3}V_2(x) \\ \partial_{x_3}V_1(x) & \partial_{x_3}V_2(x) & \partial_{x_3}V_3(x) \end{array} \right) \,dx = 0.
\end{equation*}
Consequently, from the arbitrariness of $v$, we infer that $\sigma_{i3} = 0$.

To obtain (iv) and (v) let $\bar\varphi \in C_c^{\infty}(\omega;\R^3)$ and  choose the test function
\begin{equation*}
    \varphi(x) = \left(\begin{array}{c} \bar\varphi_1(x') - x_3\,\partial_{x_1}\bar\varphi_3(x') \\ \bar\varphi_2(x') - x_3\,\partial_{x_2}\bar\varphi_3(x') \\ \frac{1}{h}\,\bar\varphi_3(x') \end{array}\right).
\end{equation*}
We deduce from \eqref{div sigma^h = 0} that
\begin{equation*}
    \int_{\Omega} \sigma^h(x) : \begin{pmatrix} E\bar\varphi(x') - x_3 D^2\bar\varphi_3(x') & 0 \\ 0 & 0 \end{pmatrix} \,dx = 0.
\end{equation*}
Passing to the limit, we conclude that
\begin{equation*}
    \div_{x'}\bar{\sigma} = 0 \text{ in } \omega,  \;\text{ and }\;  \div_{x'}\div_{x'}\hat{\sigma} = 0 \text{ in } \omega.
\end{equation*}
Next we prove the microscopic properties (i), (ii) and (iii). 
Consider test functions $\epsh\,\phi\left(x,\frac{x'}{\epsh}\right)$, for $\phi \in C_c^{\infty}(\omega;C^{\infty}(\closure{I} \times \calY;\R^3))$ in \eqref{div sigma^h = 0}.
We first observe that the sequence
\begin{equation*}
    \nabla_{h}\left(\epsh\,\phi\left(x,\frac{x'}{\epsh}\right)\right) = \left[\; \epsh\,\nabla_{x'}\phi\left(x,\frac{x'}{\epsh}\right) + \nabla_{y}\phi\left(x,\frac{x'}{\epsh}\right) \;\right|\left.\; \frac{\epsh}{h}\,\partial_{x_3}\phi\left(x,\frac{x'}{\epsh}\right) \;\right]
\end{equation*}
converges strongly two-scale in $L^2(\Omega \times \calY;\M^{3 \times 3})$.
Hence, passing to the limit as $h\to 0$, we find
\begin{equation*}
    \int_{\Omega \times \calY} \Sigma(x,y) : \widetilde{E}_{\gamma}\phi\left(x,y\right) \,dx dy = 0.
\end{equation*}
Suppose now that $\phi\left(x,y\right) = \psi^{(1)}(x')\,\psi^{(2)}(x_3,y)$ for $\psi^{(1)} \in C_c^{\infty}(\omega)$ and $\psi^{(2)} \in C^{\infty}(\closure{I} \times \calY;\R^3)$.
Then
\begin{equation*}
    \int_{\omega} \psi^{(1)}(x') \left(\int_{I \times \calY} \Sigma(x,y) : \widetilde{E}_{\gamma}\psi^{(2)}(x_3,y) \,dx_3 dy\right) \,dx' = 0.
\end{equation*}
Thus, for a.e. $x' \in \omega$,
\begin{align*}
    0 &= \int_{I \times \calY} \Sigma(x,y) : \widetilde{E}_{\gamma}\psi^{(2)}(x_3,y) \,dx_3 dy\\
    &= - \int_{I \times \calY} \widetilde{\div}_{\gamma}\Sigma(x,y) \cdot \psi^{(2)}(x_3,y) \,dx_3 dy + \int_{\partial(I \times \calY)} \Sigma(x,y)\,\nu \cdot \psi^{(2)}(x_3,y) \,d\calH^{2}(x_3,y)\\
    &= - \int_{I \times \calY} \widetilde{\div}_{\gamma}\Sigma(x,y) \cdot \psi^{(2)}(x_3,y) \,dx_3 dy + \int_{\partial{I} \times \calY} \Sigma(x,y)\,\vec{e}_3 \cdot \psi^{(2)}(x_3,y) \,d\calH^{2}(x_3,y),
\end{align*}
from which we infer $\widetilde{\div}_{\gamma}\Sigma(x',\cdot) = 0 \text{ in } I \times \calY$ and $\Sigma(x',\cdot)\,\vec{e}_3 = 0 \text{ on } \partial{I} \times \calY$.

Finally, we define
\begin{equation} \label{approximating sequence for K^hom}
    \Sigma^h(x,y) = \sum_{i \in I_\epsh(\ext{\omega})} \charfun{Q_\epsh^i}(x')\,\sigma^h(\epsh i + \epsh\calI(y),x_3),
\end{equation}
and consider the set
\begin{equation*}
    S = \{ \Xi \in L^2(\Omega \times \calY;\M^{3 \times 3}_{\sym}) : \Xi_{\dev}(x,y) \in K(y) \,\text{ for } \calL^{3}_{x} \otimes \calL^{2}_{y}\text{-a.e. } (x,y) \in \Omega \times \calY \}.
\end{equation*}
The construction of $\Sigma^h$ from $\sigma^h \in \calK_h$ ensures that $\Sigma^h \in S$ and that $\Sigma^h \weak \Sigma \;\text{ weakly in } L^2(\Omega \times \calY;\M^{3 \times 3}_{\sym})$. 
Since the compactness of $K(y)$ implies that $S$ is is convex and weakly closed in $L^2(\Omega \times \calY;\M^{3 \times 3}_{\sym})$, we have that $\Sigma \in S$, which concludes the proof.
\end{proof}
\BBB
Conversely, under additional star-shapedness assumptions on $\omega$, we now provide an approximation result for elements of \MMM $\calK_\gamma^{\rm hom}$. \BBB
\begin{lemma} \label{approximation of stresses - regime gamma}
Let $\omega \subset \R^2$ be an open bounded set that is star-shaped with respect to one of its points and let $\Sigma \in \calK^{hom}_{\gamma}$. 
Then, there exists a sequence $\Sigma_n \in L^2(\R^2 \times I \times \calY;\M^{3 \times 3}_{\sym})$ such that the following holds:
\begin{enumerate}[label=(\alph*)]
    \item \label{approximation of stresses (a) - regime gamma} $\Sigma_n \in C^\infty(\R^2;L^2(I \times \calY;\M^{3 \times 3}_{\sym}))$ and $\Sigma_n \strong \Sigma$ strongly in $L^2(\omega \times I \times \calY;\M^{3 \times 3}_{\sym})$,
    \item \label{approximation of stresses (b) - regime gamma} $\widetilde{\div}_{\gamma}\Sigma_n(x',\cdot) = 0 \text{ on } I \times \calY$ for every $x' \in \R^2$,
    \item \label{approximation of stresses (c) - regime gamma} $\Sigma_n(x',\cdot)\,\vec{e}_3 = 0 \text{ on } \partial{I} \times \calY$ for every $x' \in \R^2$,
    \item \label{approximation of stresses (d) - regime gamma} $(\Sigma_n(x,y))_{\dev} \in K(y) \,\text{ for every } x' \in \R^2 \text{ and } \calL^{1}_{x_3} \otimes \calL^{2}_{y}\text{-a.e. } (x_3,y) \in I \times \calY$.
\end{enumerate}
    Further, if we set $\sigma_n(x) := \int_{\calY} \Sigma_n(x,y) \,dy$, and $\bar{\sigma}_n,\, \hat{\sigma}_n \in L^2(\omega;\M^{2 \times 2}_{\sym})$ are the zero-th and first order moments of the $2 \times 2$ minor of $\sigma_n$, then:
\begin{enumerate}[label=(\alph*), resume]
    \item \label{approximation of stresses (e) - regime gamma} $\sigma_n \in C^\infty(\R^2 \times I;\M^{3 \times 3}_{\sym})$ and $\sigma_n \strong \sigma$ strongly in $L^2(\omega \times I;\M^{3 \times 3}_{\sym})$,
    \item \label{approximation of stresses (f) - regime gamma} $\div_{x'}\bar{\sigma}_n = 0 \text{ in } \omega$,
    \item \label{approximation of stresses (g) - regime gamma} $\div_{x'}\div_{x'}\hat{\sigma}_n = 0 \text{ in } \omega$.
\end{enumerate}
\end{lemma}

\begin{proof}
After a translation we may assume that $\omega$ is star-shaped with respect to the origin.

Thus, in particular,
\begin{equation} \label{star-shapedness 2 w.r.t. origin}
    \omega \subseteq \alpha \omega, \ \text{ for all } \alpha \geq 1.
\end{equation}

We extend $\Sigma$ to $\R^2 \times I \times \calY$ by setting $\Sigma = 0$ outside $\Omega \times \calY$. 
Let $\rho$ be the standard mollifier on $\R^2$ and define the planar dilation $d_n(x') = \left( \tfrac{n}{n+1}x'\right)$, for every $n \in \N$. 
Owing to \eqref{star-shapedness 2 w.r.t. origin}, there exists a vanishing sequence $\epsilon_n > 0$ such that for every map $\varphi \in C_c^{\infty}(\omega;\R^2)$ 
\begin{equation} \label{star-shapedness supp}
    \supp( \rho_{\epsilon_n} \ast \varphi ) \subset\subset \tfrac{n+1}{n}\omega = d_n^{-1}(\omega)\text{ implies } \supp\left(( \rho_{\epsilon_n} \ast \varphi )\circ d_n^{-1}\right) \subset\subset \omega.
\end{equation}
We then set 
\begin{equation}
    \Sigma_n(x',x_3,y) := \big(\left( \Sigma \circ d_n \right)(\cdot,x_3,y) \ast \rho_{\epsilon_n}\big)(x').
\end{equation}
With a slight abuse of notation, we have
\begin{align*}
    \sigma_n(x',x_3) &= \big(\left( \sigma \circ d_n \right)(\cdot,x_3) \ast \rho_{\epsilon_n}\big)(x'),\\
    \bar{\sigma}_n(x') &= \big(\left( \bar{\sigma} \circ d_n \right) \ast \rho_{\epsilon_n}\big)(x'),\\
    \hat{\sigma}_n(x') &= \big(\left( \hat{\sigma} \circ d_n \right) \ast \rho_{\epsilon_n}\big)(x').
\end{align*}

Items \ref{approximation of stresses (a) - regime gamma} and \ref{approximation of stresses (e) - regime gamma} are immediate consequences of the above construction, while item \ref{approximation of stresses (d) - regime gamma} follows from Jensen's inequality since $K(y)$ is convex. 
Next, for $x' \in \R^2$
\begin{equation*}
    \widetilde{\div}_{\gamma}\Sigma_n(x',\cdot) = \widetilde{\div}_{\gamma}\left( \Sigma \circ d_n \right) \ast \rho_{\epsilon_n} = 0 \text{ in } I \times \calY,
\end{equation*}
which proves item \ref{approximation of stresses (b) - regime gamma}.

To prove item \ref{approximation of stresses (f) - regime gamma}, we observe that, for every map $\varphi \in C_c^{\infty}(\omega;\R^2)$ there holds
\begin{align*}
     \langle \div_{x'}\bar{\sigma}_n, \varphi \rangle &= -\int_{\R^2} \bar{\sigma}_n : \nabla_{x'}\varphi \,dx'= -\int_{\R^2} ( \bar{\sigma} \circ d_n ) : (\rho_{\epsilon_n} \ast \nabla_{x'}\varphi) \,dx'\\
     &= -\int_{\R^2} ( \bar{\sigma} \circ d_n ) : \nabla_{x'}(\rho_{\epsilon_n} \ast \varphi) \,dx'= -(\tfrac{n+1}{n})^2\int_{\R^2} \bar{\sigma} : [ \nabla_{x'}(\rho_{\epsilon_n} \ast \varphi) \circ d_n^{-1}] \,dx'\\
     &= -(\tfrac{n+1}{n})\int_{\R^2} \bar{\sigma} : \nabla_{x'}[(\rho_{\epsilon_n} \ast \varphi) \circ d_n^{-1}] \,dx'= (\tfrac{n+1}{n}) \langle \div_{x'}\bar{\sigma}, (\rho_{\epsilon_n} \ast \varphi) \circ d_n^{-1} \rangle = 0,
\end{align*}
where in last equation we used that $\div_{x'}\bar{\sigma} = 0 \text{ in } \omega$ and \eqref{star-shapedness supp}.

Similarly for item \ref{approximation of stresses (g) - regime gamma}, for every map $\varphi \in C_c^{\infty}(\omega)$ we have
\begin{align*}
     \langle \div_{x'}\div_{x'}\hat{\sigma}_n, \varphi \rangle &= \int_{\R^2} \bar{\sigma}_n : \nabla_{x'}^2\varphi \,dx'= \int_{\R^2} ( \hat{\sigma} \circ d_n ) : (\rho_{\epsilon_n} \ast \nabla_{x'}^2\varphi) \,dx'\\
     &= \int_{\R^2} ( \hat{\sigma} \circ d_n ) : \nabla_{x'}^2(\rho_{\epsilon_n} \ast \varphi) \,dx'= (\tfrac{n+1}{n})^2\int_{\R^2} \hat{\sigma} : [\nabla_{x'}^2(\rho_{\epsilon_n} \ast \varphi) \circ d_n^{-1}] \,dx'\\
     &= \int_{\R^2} \hat{\sigma} : \nabla_{x'}^2[(\rho_{\epsilon_n} \ast \varphi) \circ d_n^{-1}] \,dx'= \langle \div_{x'}\div_{x'}\hat{\sigma}, (\rho_{\epsilon_n} \ast \varphi) \circ d_n^{-1} \rangle = 0,
\end{align*}
where in last equation we used that $\div_{x'}\div_{x'}\hat{\sigma} = 0 \text{ in } \omega$ and \eqref{star-shapedness supp}.
\end{proof}

\subsection{The principle of maximum plastic work}
The aim of this subsection is to prove an inequality between two-scale dissipation and plastic work, which in turn will be essential to prove the global stability condition of two-scale quasistatic evolutions. \MMM The claim 
is given in \Cref{two-scale dissipation and plastic work inequality}
below. 
\BBB

\MMM The proof of the following proposition and consequently \Cref{two-scale Hill's principle - regime gamma} relies on the approximation argument given in \Cref{approximation of stresses - regime gamma} and on two-scale duality, which can be established only for smooth stresses by disintegration and \Cref{def:dist}, see also \cite[Proposition 5.11]{Francfort.Giacomini.2014}. The problem is that the measure $\eta$ defined in \Cref{disintegration result - regime gamma} can concentrate on the points where the stress (which is only in $L^2$) is not well-defined. The difference with respect to  \cite[Proposition 5.11]{Francfort.Giacomini.2014} is that one can rely only on the approximation given by \Cref{approximation of stresses - regime gamma} which is given for star-shaped domains. To prove it for general domains we use the localization argument (see the proof of Step 2 of \Cref{two-scale stress-strain duality - regime gamma} and the proof of \Cref{two-scale Hill's principle - regime gamma}).\BBB 
\begin{proposition} \label{two-scale stress-strain duality - regime gamma}
Let $\Sigma \in \calK^{hom}_{\gamma}$ and $(u,E,P) \in \calA^{hom}_{\gamma}(w)$ with the associated $\mu \in \calXgamma{\ext{\omega}}$. 
There exists an element $\lambda \in \Mb(\ext{\Omega} \times \calY)$ such that
for every $\varphi \in C_c^2(\ext{\omega})$
\begin{align*}
    \langle \lambda, \varphi \rangle =& - \int_{\Omega \times \calY} \varphi(x')\,\Sigma : E \,dx dy + \int_{\omega} \varphi\,\bar{\sigma} : E\bar{w} \,dx' - \frac{1}{12}\int_{\omega} \varphi\,\hat{\sigma} : D^2w_3 \,dx' \\
    & - \int_{\omega} \bar{\sigma} : \left( (\bar{u}-\bar{w}) \odot \nabla\varphi \right) \,dx' - \frac{1}{6} \int_{\omega} \hat{\sigma} : \big( \nabla (u_3 - w_3) \odot \nabla \varphi \big) \,dx' \\
    & - \frac{1}{12} \int_{\omega} (u_3 - w_3)\,\hat{\sigma} : \nabla^2 \varphi \,dx'.
\end{align*}
Furthermore, the mass of $\lambda$ is given by
\begin{equation} \label{mass of lambda - regime gamma}
    \lambda(\ext{\Omega} \times \calY) = -\int_{\Omega \times \calY} \Sigma : E \,dx dy + \int_{\omega} \bar{\sigma} : E\bar{w} \,dx' - \frac{1}{12} \int_{\omega} \hat{\sigma} : D^2w_3 \,dx'.
\end{equation}
\end{proposition}

\begin{proof}
The proof is subdivided into two steps.

\noindent{\bf Step 1.}
Suppose that $\omega$ is star-shaped with respect to one of its points.

Let $\{ \Sigma_n \} \subset C^\infty(\R^2;L^2(I \times \calY;\M^{3 \times 3}_{\sym}))$ be the sequence given by Lemma \ref{approximation of stresses - regime gamma}. 
We set
\begin{equation*}
    \lambda_n := \eta \genprod [ (\Sigma_n)_{\dev}(x',\cdot) : P_{x'} ] \in \Mb(\ext{\Omega} \times \calY),
\end{equation*}
where the duality $[ (\Sigma_n)_{\dev}(x',\cdot) : P_{x'} ]$ is a well defined bounded measure on $I \times \calY$ for $\eta$-a.e. $x' \in \ext{\omega}$ and \MMM $\eta$ is defined in \Cref{disintegration result - regime gamma}. \BBB 
Further, in view of \Cref{admissible configurations and disintegration - regime gamma}, Definition \ref{def:dist} gives
\begin{align*}
    &\int_{\R \times \calY} \psi\,d[ (\Sigma_n)_{\dev}(x',\cdot) : P_{x'} ] \\
    &= -\int_{I \times \calY} \psi(x_3,y)\,\Sigma_n(x,y) : \left[ C(x') E(x,y) - \begin{pmatrix} A_1(x') + x_3 A_2(x') & 0 \\ 0 & 0 \end{pmatrix} \right] \,dx_3 dy\\
    &\,\quad - \int_{I \times \calY} \Sigma_n(x,y) : \big( \mu_{x'}(x_3,y) \odot \widetilde{\nabla}_{\gamma}\psi(x_3,y) \big) \,dx_3 dy,
\end{align*}
for every $\psi \in C^1(\R \times \calY)$, and
\begin{equation*}
    | [ (\Sigma_n)_{\dev}(x',\cdot) : P_{x'} ] | \leq \| (\Sigma_n)_{\dev}(x',\cdot) \|_{L^{\infty}(I \times \calY;\M^{3 \times 3}_{\sym})} |P_{x'}| \leq C\,|P_{x'}|,
\end{equation*}
where the last inequality stems from item \ref{approximation of stresses (d) - regime gamma} in Lemma \ref{approximation of stresses - regime gamma}. 
This in turn implies that
\begin{equation*}
    |\lambda_n | = \eta \genprod | [ (\Sigma_n)_{\dev}(x',\cdot) : P_{x'} ] | \leq C \,\eta \genprod |P_{x'}| = C\,|P|,
\end{equation*}
from which we conclude that $\{ \lambda_n \}$ is a bounded sequence.

Let now $\ext{I} \supset I$ be an open set which compactly contains $I$. 
Let $\xi$ be a smooth cut-off function with $\xi \equiv 1$ on $I$, and with support contained in $\ext{I}$. 
Finally, consider a test function $\phi(x,y) := \varphi(x')\xi(x_3)$, for $\varphi \in C_c^{\infty}(\ext{\omega})$. 
Since $\widetilde{\nabla}_{\gamma}\phi(x,y) = 0$, we have
\begin{align} 
    \langle \lambda_n, \phi \rangle 
    & \nonumber= \int_{\ext{\omega}} \left( \int_{I \times \calY} \phi(x,y) \,d[ (\Sigma_n)_{\dev}(x',\cdot) : P_{x'} ] \right) \,d\eta(x')\\
    & \nonumber= - \int_{\ext{\Omega} \times \calY} \varphi(x')\,\Sigma_n(x,y) : \left[ C(x') E(x,y) - \begin{pmatrix} A_1(x') + x_3 A_2(x') & 0\\ 0 & 0 \end{pmatrix} \right] \,d\left(\eta \otimes \calL^{1}_{x_3} \otimes \calL^{2}_{y}\right)\\
    &\nonumber = - \int_{\ext{\Omega} \times \calY} \varphi(x')\,\Sigma_n(x,y) : E(x,y) \,dx dy + \int_{\ext{\Omega}} \varphi(x')\,\sigma_n(x) : \begin{pmatrix} A_1(x') + x_3 A_2(x') & 0 \\ 0 & 0 \end{pmatrix} \,d\left(\eta \otimes \calL^{1}_{x_3}\right)\\
    & \label{krajnje1}= - \int_{\ext{\Omega} \times \calY} \varphi(x')\,\Sigma_n(x,y) : E(x,y) \,dx dy + \int_{\ext{\Omega}} \varphi(x')\,\sigma_n(x) : \,dEu(x)
\end{align}
Since $u \in KL(\ext{\Omega})$, we infer 
\begin{align}\label{krajnje2} 
    \int_{\ext{\Omega}} \varphi(x')\,\sigma_n(x) : \,dEu(x) 
    &= \int_{\ext{\omega}} \varphi(x')\,\bar{\sigma}_n(x') : \,dE\bar{u}(x') - \frac{1}{12}\int_{\ext{\omega}} \varphi(x')\,\hat{\sigma}_n(x') : \,dD^2u_3(x'),
\end{align}
where $\bar{u} \in BD(\ext{\omega})$ and $u_3 \in BH(\ext{\omega})$ are the Kirchhoff-Love components of $u$. 
From the characterization given in \Cref{A_KL characherization}, we can thus conclude that
\begin{align}
    \nonumber\int_{\ext{\Omega}} \varphi(x')\,\sigma_n(x) : \,dEu(x) 
    =& \int_{\ext{\omega}} \varphi(x')\,\bar{\sigma}_n(x') : \bar{e}(x') \,dx' + \int_{\ext{\omega}} \varphi(x')\,\bar{\sigma}_n(x') : \,d\bar{p}(x')\\
    \nonumber& + \frac{1}{12}\int_{\ext{\omega}} \varphi(x')\,\hat{\sigma}_n(x') : \hat{e}(x') \,dx' + \frac{1}{12}\int_{\ext{\omega}} \varphi(x')\,\hat{\sigma}_n(x') : \,d\hat{p}(x')\\
    \nonumber=& \int_{\ext{\omega}} \varphi(x')\,\bar{\sigma}_n(x') : \bar{e}(x') \,dx' + \int_{\ext{\omega}} \varphi(x') \,d[\bar{\sigma}_n : \bar{p}](x')\\
    &  \label{krajnje3} + \frac{1}{12}\int_{\ext{\omega}} \varphi(x')\,\hat{\sigma}_n(x') : \hat{e}(x') \,dx' + \frac{1}{12}\int_{\ext{\omega}} \varphi(x') \,d[\hat{\sigma}_n : \hat{p}](x'),
\end{align}
where in the last equality we used that $\bar{\sigma}_n$ and $\hat{\sigma}_n$ are smooth functions. 
Notice that, since $\bar{p} \equiv 0$ and $\hat{p} \equiv 0$ outside of $\omega \cup \gamma_\Dir$, there holds
\begin{equation*}
    \int_{\ext{\omega}} \varphi \,d[\bar{\sigma}_n : \bar{p}] = \int_{\omega \cup \gamma_\Dir} \varphi \,d[\bar{\sigma}_n : \bar{p}], \quad 
    \int_{\ext{\omega}} \varphi \,d[\hat{\sigma}_n : \hat{p}] = \int_{\omega \cup \gamma_\Dir} \varphi \,d[\hat{\sigma}_n : \hat{p}].
\end{equation*}
Since $e = E = E\bar{w} - x_3 D^2w_3$ on $\ext{\Omega} \setminus \Omega$, we deduce, \MMM using \eqref{krajnje1}-\eqref{krajnje3},\BBB that
\begin{align}
    \nonumber\langle \lambda_n, \phi \rangle 
    &= - \int_{\ext{\Omega} \times \calY} \varphi(x')\,\Sigma_n : E \,dx dy + \int_{\ext{\omega}} \varphi\,\bar{\sigma}_n : \bar{e} \,dx' + \frac{1}{12}\int_{\ext{\omega}} \varphi\,\hat{\sigma}_n : \hat{e} \,dx'\\
     \nonumber&\,\quad + \int_{\omega \cup \gamma_\Dir} \varphi \,d[\bar{\sigma}_n : \bar{p}] + \frac{1}{12}\int_{\omega \cup \gamma_\Dir} \varphi \,d[\hat{\sigma}_n : \hat{p}]\\
     \nonumber&= - \int_{\Omega \times \calY} \varphi(x')\,\Sigma_n : E \,dx dy + \int_{\omega} \varphi\,\bar{\sigma}_n : \bar{e} \,dx' + \frac{1}{12}\int_{\omega} \varphi\,\hat{\sigma}_n : \hat{e} \,dx'\\
     &\,\quad + \int_{\omega \cup \gamma_\Dir} \varphi \,d[\bar{\sigma}_n : \bar{p}] + \frac{1}{12}\int_{\omega \cup \gamma_\Dir} \varphi \,d[\hat{\sigma}_n : \hat{p}]. \label{krajnje4} 
\end{align}
\MMM Using that $\div_{x'}\bar{\sigma}_n = 0 \text{ in } \omega$, by applying an integration by parts (see also \cite[Proposition 7.2]{Davoli.Mora.2013}) \MMM we obtain for every $\varphi \in C^1(\closure{\omega})$
\begin{align}\label{krajnje5} 
    \int_{\omega \cup \gamma_\Dir} \varphi \,d[\bar{\sigma}_n : \bar{p}] + \int_{\omega} \varphi\,\bar{\sigma}_n : (\bar{e} - E\bar{w}) \,dx' +  \int_{\omega} \bar{\sigma}_n : \left( (\bar{u} - \bar{w}) \odot \nabla\varphi \right) \,dx' = 0.
\end{align}
Likewise \MMM in view of the fact that $\div_{x'}\div_{x'}\hat{\sigma}_n = 0 \text{ in } \omega$ 
and $u_3 = w_3 \text{ on } \gamma_\Dir$, 
by integration by parts (see also \cite[Proposition 7.6]{Davoli.Mora.2013}) \BBB we find that for every $\varphi \in C^2(\closure{\omega})$
\begin{align}
    \nonumber &\int_{\omega \cup \gamma_\Dir} \varphi \,d[\hat{\sigma}_n : \hat{p}] + \int_{\omega} \varphi\,\hat{\sigma}_n : (\hat{e} + D^2w_3) \,dx' \\
    &+ 2 \int_{\omega} \hat{\sigma}_n : \big( \nabla (u_3 - w_3) \odot \nabla \varphi \big) \,dx' + \int_{\omega} (u_3 - w_3)\,\hat{\sigma}_n : \nabla^2 \varphi \,dx'
    = 0. \label{krajnje6} 
\end{align}

Let now $\lambda \in \Mb(\ext{\Omega} \times \calY)$ be such that (up to a subsequence)
\begin{equation*}
    \lambda_n \weakstar \lambda \quad \text{weakly* in $\Mb(\ext{\Omega} \times \calY)$}.
\end{equation*}
By items \ref{approximation of stresses (a) - regime gamma} and \ref{approximation of stresses (e) - regime gamma} in Lemma \ref{approximation of stresses - regime gamma}, \MMM owing to \eqref{krajnje4}-\eqref{krajnje6} \BBB we obtain
\begin{align*}
    \langle \lambda, \phi \rangle &= \lim_n \,\langle \lambda_n, \phi \rangle\\
    &= \lim_n \,\Big[ - \int_{\Omega \times \calY} \varphi(x')\,\Sigma_n : E \,dx dy + \int_{\omega} \varphi\,\bar{\sigma}_n : E\bar{w} \,dx' - \frac{1}{12}\int_{\omega} \varphi\,\hat{\sigma}_n : D^2w_3 \,dx' \\
    &\,\quad\,\quad\,\quad - \int_{\omega} \bar{\sigma}_n : \left( (\bar{u}-\bar{w}) \odot \nabla\varphi \right) \,dx' - \frac{1}{6} \int_{\omega} \hat{\sigma}_n : \big( \nabla (u_3 - w_3) \odot \nabla \varphi \big) \,dx' \\
    &\,\quad\,\quad\,\quad - \frac{1}{12} \int_{\omega} (u_3 - w_3)\,\hat{\sigma}_n : \nabla^2 \varphi \,dx' \Big]\\
    &= - \int_{\Omega \times \calY} \varphi(x')\,\Sigma : E \,dx dy + \int_{\omega} \varphi\,\bar{\sigma} : E\bar{w} \,dx' - \frac{1}{12}\int_{\omega} \varphi\,\hat{\sigma} : D^2w_3 \,dx' \\
    &\,\quad - \int_{\omega} \bar{\sigma} : \left( (\bar{u}-\bar{w}) \odot \nabla\varphi \right) \,dx' - \frac{1}{6} \int_{\omega} \hat{\sigma} : \big( \nabla (u_3 - w_3) \odot \nabla \varphi \big) \,dx' \\
    &\,\quad - \frac{1}{12} \int_{\omega} (u_3 - w_3)\,\hat{\sigma} : \nabla^2 \varphi \,dx'.
\end{align*}
Taking $\varphi \nearrow \charfun{\ext{\omega}}$, we deduce \eqref{mass of lambda - regime gamma}.

\noindent{\bf Step 2.}
If $\omega$ is not star-shaped, then since $\omega$ is a bounded $C^2$ domain (in particular, with Lipschitz boundary) by \cite[Proposition 2.5.4]{Carbone.DeArcangelis.2002} there exists a finite open covering $\{U_i\}$ of $\closure{\omega}$ such that $\omega \cap U_i$ is (strongly) star-shaped with Lipschitz boundary. 

Let $\{\psi_i\}$ be a smooth partition of unity subordinate to the covering $\{U_i\}$, i.e. $\psi_i \in C^{\infty}(\closure{\omega})$, with $0 \leq \psi_i \leq 1$, such that $\supp(\psi_i) \subset U_i$ and $\sum_{i} \psi_i = 1$ on $\closure{\omega}$. 

For each $i$, let
\begin{equation*}
    \Sigma^i(x,y)
    :=
    \begin{cases}
    \Sigma(x,y) & \ \text{ if } x' \in \omega \cap U_i,\\
    0 & \ \text{ otherwise}.
    \end{cases}
\end{equation*}
Since $\Sigma^i \in \calK^{hom}_{\gamma}$, the construction in Step 1 yields that there exist sequences $\{ \Sigma^i_n \} \subset C^\infty(\R^2;L^2(I \times \calY;\M^{3 \times 3}_{\sym}))$ and 
\begin{equation*}
    \lambda^i_n := \eta \genprod [ (\Sigma^i_n)_{\dev}(x',\cdot) : P_{x'} ] \in \Mb((\omega \cap U_i) \times I \times \calY),
\end{equation*}
\MMM where again $\eta$ is  defined in \Cref{disintegration result - regime gamma} \BBB such that
\begin{equation*}
    \lambda^i_n \weakstar \lambda^i \quad \text{weakly* in $\Mb((\omega \cap U_i) \times I \times \calY)$},
\end{equation*}
with
\begin{align*}
    \langle \lambda^i, \varphi \rangle 
    &= - \int_{(\omega \cap U_i) \times I \times \calY} \varphi(x')\,\Sigma : E \,dx dy + \int_{\omega \cap U_i} \varphi\,\bar{\sigma} : E\bar{w} \,dx' - \frac{1}{12}\int_{\omega \cap U_i} \varphi\,\hat{\sigma} : D^2w_3 \,dx' \\
    &\,\quad - \int_{\omega \cap U_i} \bar{\sigma} : \left( (\bar{u}-\bar{w}) \odot \nabla\varphi \right) \,dx' - \frac{1}{6} \int_{\omega \cap U_i} \hat{\sigma} : \big( \nabla (u_3 - w_3) \odot \nabla \varphi \big) \,dx' \\
    &\,\quad - \frac{1}{12} \int_{\omega \cap U_i} (u_3 - w_3)\,\hat{\sigma} : \nabla^2 \varphi \,dx'.
\end{align*}
for every $\varphi \in C_c^2(\closure{\omega} \cap U_i)$.
This allows us to define measures on $\ext{\Omega} \times \calY$ by letting, for every $\phi \in C_0(\ext{\Omega} \times \calY)$,
\begin{equation*}
    \langle \lambda_n, \phi \rangle := \sum_{i} \langle \lambda^i_n, \psi_i(x')\,\phi \rangle,
\end{equation*}
and
\begin{equation*}
    \langle \lambda, \phi \rangle := \sum_{i} \langle \lambda^i, \psi_i(x')\,\phi \rangle.
\end{equation*}
From the above computations, $\lambda_n \weakstar \lambda$ weakly* in $\Mb(\ext{\Omega} \times \calY)$, and $\lambda$ satisfies all the required properties.
\end{proof}
\EEE The next theorem allows us to compare the density of the dissipation due to the limiting two-scale plastic strain and that of the measure $\lambda$. \BBB
\begin{theorem} \label{two-scale Hill's principle - regime gamma} 
Let $\Sigma \in \calK^{hom}_{\gamma}$ and $(u,E,P) \in \calA^{hom}_{\gamma}(w)$ with the associated $\mu \in \calXgamma{\ext{\omega}}$. 
Then
\begin{equation*}
    H\left(y, \frac{dP}{d|P|}\right)\,|P| \geq \lambda,
\end{equation*}
where $\lambda \in \Mb(\ext{\Omega} \times \calY)$ is given by \Cref{two-scale stress-strain duality - regime gamma}.
\end{theorem}

\begin{proof}
Let $\{\Sigma^i_n\}$, $\{\lambda^i_n\}$ and $\lambda^i$ be defined as in Step 2 of the proof of \Cref{two-scale stress-strain duality - regime gamma}. 
Item \ref{approximation of stresses (d) - regime gamma} in \Cref{approximation of stresses - regime gamma} implies that
\begin{equation*}
    (\Sigma^i_n)_{\dev}(x,y) \in K(y) \,\text{ for every } x' \in \omega \text{ and } \calL^{1}_{x_3} \otimes \calL^{2}_{y}\text{-a.e. } (x_3,y) \in I \times \calY.   
\end{equation*}
By \Cref{cell Hill's principle - regime gamma}, we have for $\eta$-a.e. $x' \in \ext{\omega}$
\begin{equation*}
    H\left(y, \frac{dP_{x'}}{d|P_{x'}|}\right)\,|P_{x'}| \geq [ (\Sigma^i_n)_{\dev}(x',\cdot) : P_{x'} ] \ \text{ as measures on } I \times \calY.
\end{equation*}
Since $\frac{dP}{d|P|}(x,y) = \frac{dP_{x'}}{d|P_{x'}|}(x_3,y)$ for $|P_{x'}|$-a.e. $(x_3,y) \in I \times \calY$ by \Cref{important disintegration remark}, we can conclude that
\begin{align*}
    H\left(y, \frac{dP}{d|P|}\right)\,|P| &= \eta \genprod H\left(y, \frac{dP}{d|P|}\right)\,|P_{x'}| = \eta \genprod H\left(y, \frac{dP_{x'}}{d|P_{x'}|}\right)\,|P_{x'}|\\
    &= \sum_{i} \psi_i \eta \genprod H\left(y, \frac{dP_{x'}}{d|P_{x'}|}\right)\,|P_{x'}| \\
    &\geq \sum_{i} \psi_i \eta \genprod [ (\Sigma^i_n)_{\dev}(x',\cdot) : P_{x'} ] \\
    &= \sum_{i} \psi_i \lambda^i_n = \lambda_n.
\end{align*}
By passing to the limit, we have the desired inequality.
\end{proof}
\EEE As a \BLUE direct\BLACK consequence of the previous theorem \BLUE and \eqref{mass of lambda - regime gamma}\BLACK\EEE, we are now in a position to \BLUE state\EEE a principle of maximum plastic work in our setting. \BBB
\MMM
\begin{corollary} \label{two-scale dissipation and plastic work inequality}
Let $\gamma \in (0, +\infty)$.
Then
\begin{equation*}
    \calH^{hom}(P) \geq -\int_{\Omega \times \calY} \Sigma : E \,dx dy + \int_{\omega} \bar{\sigma} : E\bar{w} \,dx' - \frac{1}{12} \int_{\omega} \hat{\sigma} : D^2w_3 \,dx',
\end{equation*}
for every $\Sigma \in \calK^{hom}_{\gamma}$ and $(u,E,P) \in \calA^{hom}_{\gamma}(w)$.
\end{corollary}
\BBB

\subsection{Liminf inequalities under weak two-scale convergence} 
\label{Lower semicontinuity of energy functionals}
For $(u,e,p) \in \calA_h(w)$, we recall the definition of energy functionals $\calQ_h$ and $\calH_h$ given in \eqref{definition Q_h} and \eqref{definition H_h}. 
For $(u,E,P) \in \calA^{hom}_{\gamma}(w)$ we now define
\begin{equation} \label{definition Q^hom} 
    \calQ^{hom}(E) := \int_{\Omega \times \calY} Q\left(y, E\right) \,dx dy
\end{equation}
and
\begin{equation} \label{definition H^hom}
    \calH^{hom}(P) := \int_{\closure{\Omega} \times \calY} H\left(y, \frac{dP}{d|P|} \right) \,d|P|.
\end{equation}
\EEE The next result shows that $\calQ^{hom}$ and $\calH^{hom}$ provide lower bounds for the asymptotic behavior of our elastic energies and dissipation potential with respect to weak two-scale convergence of elastic and plastic stresses. \BBB
\begin{theorem} \label{lsc-gamma finite}
Let $\gamma \in (0,+\infty)$. 
Let $(u^h,e^h,p^h) \in \calA_h(w)$ be such that
\begin{align}
    &u^h \weakstar u \quad \text{weakly* in $BD(\ext{\Omega})$},\\
    &\Lambda_h e^h \weaktwoscale E \quad \text{two-scale weakly in $L^2(\ext{\Omega} \times \calY;\M^{3 \times 3}_{\sym})$}, \label{Lambda_h e^h two-scale weakly}\\
    &\Lambda_h p^h \weakstartwoscale P \quad \text{two-scale weakly* in $\Mb(\ext{\Omega} \times \calY;\M^{3 \times 3}_{\dev})$},
\end{align}
with $(u,E,P) \in \calA^{hom}_{\gamma}(w)$. Then,
\begin{equation} \label{lower semicontinuity Q}
    \calQ^{hom}(E) \leq \liminf\limits_{h} \calQ_h(\Lambda_h e^h)
\end{equation}
and
\begin{equation} \label{lower semicontinuity H}
    \calH^{hom}(P) \leq \liminf\limits_{h} \calH_h(\Lambda_h p^h).
\end{equation}
\end{theorem}

\begin{proof}
Let $\varphi \in C_c^{\infty}(\Omega \times \calY;\M^{3 \times 3}_{\sym})$. 
From the coercivity condition on the quadratic form $\calQ_h$ we obtain 
\begin{equation*}
    0 \le\, \frac{1}{2} \int_{\Omega} \C\left(\frac{x'}{\epsh}\right) \left(\Lambda_h e^h(x) - \varphi\left(x,\frac{x'}{\epsh}\right)\right) : \left(\Lambda_h e^h(x) - \varphi\left(x,\frac{x'}{\epsh}\right)\right) \,dx.
\end{equation*}
Since $\C\left(\frac{x'}{\epsh}\right) \Lambda_h e^h(x) \weaktwoscale \C(y) E(x,y)$ weakly two-scale in $L^2(\Omega \times \calY;\M^{3 \times 3}_{\sym})$, we can apply the $\liminf$ to the above inequality and we find
\begin{equation*}
    \int_{\Omega \times \calY} \C(y) E(x,y) : \varphi\left(x,y\right) \,dx dy - \frac{1}{2} \int_{\Omega \times \calY} \C(y) \varphi\left(x,y\right)\ : \varphi\left(x,y\right) \,dx \le\, \liminf\limits_{h} \calQ_h(\Lambda_h e).
\end{equation*}
Choosing $\varphi$ such that $\varphi \strong E$ strongly in $L^2(\Omega \times \calY;\M^{3 \times 3}_{\sym})$ yields \eqref{lower semicontinuity Q}.

To prove \eqref{lower semicontinuity H}, we can assume without loss of generality that 
\begin{equation} \label{WLOG finite liminf}
    \liminf\limits_{h} \calH_h(\Lambda_h p^h) < \infty.
\end{equation}
We write
\begin{equation} \label{p^h decomposed}
    p^h = \sum_{i} p^h_{i} + \sum_{i \neq j} p^h_{ij}
\end{equation}
where $p^h_{i} := p^h\mres{\Omega \cap ((\calY_i)_\epsh \times I)}$ and $p^h_{ij} := p^h\mres{\ext{\Omega} \cap ((\Gamma_{ij} \setminus S)_\epsh \times I)}$. Up to a subsequence,
\begin{align*}
    &\Lambda_h p^h_{i} \weakstartwoscale P_{i} \quad \text{two-scale weakly* in $\Mb(\ext{\Omega} \times \calY;\M^{3 \times 3}_{\dev})$},\\
    &\Lambda_h p^h_{ij} \weakstartwoscale P_{ij} \quad \text{two-scale weakly* in $\Mb(\ext{\Omega} \times \calY;\M^{3 \times 3}_{\dev})$}.
\end{align*}
Clearly,
\begin{equation*}
    P = \sum_{i} P_{i} + \sum_{i \neq j} P_{ij},
\end{equation*}
with $\supp(P_{i}) \subseteq \ext{\Omega} \times \closure{\calY}_i$ and $\supp(P_{ij}) \subseteq \ext{\Omega} \times \Gamma_{ij}$. 
In view of \eqref{Lambda_h e^h two-scale weakly}, we infer
\begin{equation*}
    \Lambda_h Eu^h\mres{\ext{\Omega} \cap ((\calY_i)_\epsh \times I)} \weakstartwoscale E \,\charfun{\ext{\Omega} \times \calY_i} \,\calL^{3}_{x} \otimes \calL^{2}_{y} + P_{i} \quad \text{two-scale weakly* in $\Mb(\ext{\Omega} \times \calY;\M^{3 \times 3}_{\sym})$}
\end{equation*}
Recalling \eqref{transversality condition}, we can additionally assume that $\Gamma_{ij} \cap \calC \subseteq S$. 
Then, with a normal $\nu$ on $\Gamma_{ij}$ that points from $\calY_j$ to $\calY_i$ for every $j \neq i$, Lemma \ref{rank-1 lemma} implies that
\begin{equation} \label{rank-1 lemma implication}
    P_{i}\mres{\ext{\Omega} \times (\Gamma_{ij} \setminus S)} = \MMM a_{ij}(x,y) \BBB \odot \nu(y) \,\eta_{ij}
\end{equation}
for suitable $\eta_{ij} \in \Mb^+ (\ext{\Omega} \times (\Gamma_{ij} \setminus S))$ and a Borel map $a_{ij} : \ext{\Omega} \times (\Gamma_{ij} \setminus S) \to \R^3$ such that $a_{ij} \perp \nu$ for $\eta_{ij}$-a.e. $(x,y) \in \ext{\Omega} \times (\Gamma_{ij} \setminus S)$.

Using a version of Reshetnyak's lower semicontinuity theorem adapted for two-scale convergence (see \cite[Lemma 4.6]{Francfort.Giacomini.2014}), we deduce
\begin{align} \label{H liminf 1}
    \nonumber &\liminf\limits_{h} \int_{\Omega \cup \Gamma_\Dir} H\left(\frac{x'}{\epsh}, \frac{d\Lambda_h p^h_{i}}{d|\Lambda_h p^h_{i}|}\right) \,d|\Lambda_h p^h_{i}|\\
    \nonumber &= \liminf\limits_{h} \int_{\ext{\Omega}} H_i\left(\frac{d\Lambda_h p^h_{i}}{d|\Lambda_h p^h_{i}|}\right) \,d|\Lambda_h p^h_{i}|
    \geq \int_{\ext{\Omega} \times \calY} H_i\left(\frac{dP_{i}}{d|P_{i}|}\right) \,d|P_{i}|\\
    \nonumber &= \int_{\ext{\Omega} \times \calY_i} H_i\left(\frac{dP_{i}}{d|P_{i}|}\right) \,d|P_{i}| + \int_{\ext{\Omega} \times \Gamma} H_i\left(\frac{dP_{i}}{d|P_{i}|}\right) \,d|P_{i}|\\
    \nonumber &\geq \int_{\ext{\Omega} \times \calY_i} H\left(y, \frac{dP_{i}}{d|P_{i}|}\right) \,d|P_{i}| + \sum_{j \neq i} \int_{\ext{\Omega} \times (\Gamma_{ij} \setminus S)} H_i\left(\frac{dP_{i}}{d|P_{i}|}\right) \,d|P_{i}|\\
    &\geq \int_{\ext{\Omega} \times \calY_i} H\left(y, \frac{dP_{i}}{d|P_{i}|}\right) \,d|P_{i}| + \sum_{j \neq i} \int_{\ext{\Omega} \times (\Gamma_{ij} \setminus S)} H_i\left(-a_{ij}(x,y) \odot \nu(y)\right) \,d\eta_{ij}.
\end{align}
Next, we have
\begin{align*}
    \Lambda_h p^h_{ij} &= \Lambda_h \left[(u^h_i - u^h_j) \odot \nu\left(\frac{x'}{\epsh}\right)\right] \,\calH^{2}\mres{\ext{\Omega} \cap ((\Gamma_{ij} \setminus S)_\epsh \times I)}\\
    &= \left[\diag\left(1,1,\frac{1}{h}\right)\,(u^h_i - u^h_j) \odot \nu\left(\frac{x'}{\epsh}\right)\right] \,\calH^{2}\mres{\ext{\Omega} \cap ((\Gamma_{ij} \setminus S)_\epsh \times I)},
\end{align*}
where $u^h_i$ and $u^h_j$ are the traces on $\ext{\Omega} \cap ((\Gamma_{ij} \setminus S)_\epsh \times I)$ of the restrictions of $u^h$ to $(\calY_i)_\epsh \times I$ and $(\calY_j)_\epsh \times I$ respectively, such that $u^h_i - u^h_j$ is perpendicular to $\nu$. 
Then, since the infimum in the inf-convolution definition of $H$ on $\Gamma \setminus S$ is actually a minimum, we obtain
\begin{align} \label{b expression}
    \nonumber &\int_{\Omega \cup \Gamma_\Dir} H\left(\frac{x'}{\epsh}, \frac{d\Lambda_h p^h_{ij}}{d|\Lambda_h p^h_{ij}|}\right) \,d|\Lambda_h p^h_{ij}|\\
    \nonumber &= \int_{\ext{\Omega} \cap ((\Gamma_{ij} \setminus S)_\epsh \times I)} H\left(\frac{x'}{\epsh}, \frac{d\Lambda_h p^h_{ij}}{d|\Lambda_h p^h_{ij}|}\right) \,d|\Lambda_h p^h_{ij}|\\
    \nonumber &= \int_{\ext{\Omega} \cap ((\Gamma_{ij} \setminus S)_\epsh \times I)} H\left(\frac{x'}{\epsh}, \left[\diag\left(1,1,\frac{1}{h}\right)\,(u^h_i - u^h_j) \odot \nu\left(\frac{x'}{\epsh}\right)\right]\right) \,d\calH^{2}(x)\\
    \nonumber &= \int_{\ext{\Omega} \cap ((\Gamma_{ij} \setminus S)_\epsh \times I)} H_{ij}\left(\diag\left(1,1,\frac{1}{h}\right)\,(u^h_i - u^h_j), \nu\left(\frac{x'}{\epsh}\right)\right) \,d\calH^{2}(x)\\
    &= \int_{\ext{\Omega} \cap ((\Gamma_{ij} \setminus S)_\epsh \times I)} \left[H_{i}\left(b^{h,ij}_i(x) \odot \nu\left(\frac{x'}{\epsh}\right)\right) + H_{j}\left(-b^{h,ij}_j(x) \odot \nu\left(\frac{x'}{\epsh}\right)\right)\right] \,d\calH^{2}(x)
\end{align}
for suitable Borel functions $b^{h,ij}_i, b^{h,ij}_j : \ext{\Omega} \cap ((\Gamma_{ij} \setminus S)_\epsh \times I) \to \R^3$ which are orthogonal to $\nu$ for $\calH^{2}$-a.e. $x \in (\Gamma_{ij} \setminus S)_\epsh \times I$ and such that
\begin{align*}
    b^{h,ij}_i - b^{h,ij}_j = \diag\left(1,1,\frac{1}{h}\right)\,(u^h_i - u^h_j) \qquad \text{ for }\calH^{2}\text{-a.e. } x \in (\Gamma_{ij} \setminus S)_\epsh \times I.
\end{align*}
From the coercivity condition of the dissipation potential $H$ and \eqref{WLOG finite liminf}, we conclude that
\begin{equation*}
    \int_{\ext{\Omega} \cap ((\Gamma_{ij} \setminus S)_\epsh \times I)} \left[\left|b^{h,ij}_i(x) \odot \nu\left(\frac{x'}{\epsh}\right)\right| + \left|b^{h,ij}_j(x) \odot \nu\left(\frac{x'}{\epsh}\right)\right|\right] \,d\calH^{2}(x) \leq C, 
\end{equation*}
for some constant $C > 0$, \MMM which implies the boundedness of $b^{h,ij}_i$ and $b^{h,ij}_j$ in $L^1$.
 We can now argue as in Step 2 of the proof of \cite[Theorem 5.7]{Francfort.Giacomini.2014} or \cite[Proposition 2.3]{Francfort.Giacomini.2012}, \MMM using also \eqref{rank-1 lemma implication}, \BBB and infer that the existence of suitable measures $\zeta_{ij} \in \Mb^+(\ext{\Omega} \times (\Gamma_{ij} \setminus S))$, and Borel functions $c^i, c^j : \ext{\Omega} \times (\Gamma_{ij} \setminus S) \to \R^3$ which are orthogonal to $\nu$ for $\zeta_{ij}$-a.e. $(x,y) \in \ext{\Omega} \times (\Gamma_{ij} \setminus S)$, and such that
\begin{align*}
    P\mres{\ext{\Omega} \times (\Gamma_{ij} \setminus S)} = \left(c^i(x,y)-c^j(x,y)\right) \odot \nu(y) \,\zeta_{ij}.
\end{align*}
Thus, by \eqref{H liminf 1}, we have
\begin{align*}
    &\liminf\limits_{h} \calH_h(\Lambda_h p^h)\\
    & \MMM \geq \BBB \int_{\ext{\Omega} \times ( \cup _i \calY_i)} H\left(y, \frac{dP}{d|P|}\right) \,d|P|\\
    &\hspace{1em} + \sum_{i \neq j} \int_{\ext{\Omega} \times (\Gamma_{ij} \setminus S)} \left[ H_i\left(c^i(x,y) \odot \nu(y)\right) + H_j\left(-c^j(x,y) \odot \nu(y)\right) \right] \,d\zeta_{ij}\\
    &\geq \int_{\ext{\Omega} \times ( \cup _i \calY_i)} H\left(y, \frac{dP}{d|P|}\right) \,d|P| + \sum_{i \neq j} \int_{\ext{\Omega} \times (\Gamma_{ij} \setminus S)} H\left(y, \left(c^i(x,y)-c^j(x,y)\right) \odot \nu(y)\right) \,d\zeta_{ij}\\
    &= \int_{\ext{\Omega} \times ( \cup _i \calY_i)} H\left(y, \frac{dP}{d|P|}\right) \,d|P| + \sum_{i \neq j} \int_{\ext{\Omega} \times (\Gamma_{ij} \setminus S)} H\left(y, \frac{dP}{d|P|}\right) \,d|P|\\
    &= \calH^{hom}(P),
\end{align*}
which in turn concludes the proof. \BLACK
\end{proof}

\section{Two-scale quasistatic evolutions}
\label{dynamics}

We recall the definition of energy functionals $\calQ^{hom}$ and $\calH^{hom}$ given in \eqref{definition Q^hom} and \eqref{definition H^hom}. 
The associated $\calH^{hom}$-variation of a function $P : [0,T] \to \Mb(\ext{\Omega} \times \calY;\M^{3 \times 3}_{\dev})$ on $[a,b]$ is then defined as
\begin{equation*}
    \calD_{\calH^{hom}}(P; a, b) := \sup\left\{ \sum_{i = 1}^{n-1} \calH^{hom}\left(P(t_{i+1}) - P(t_i)\right) : a = t_1 < t_2 < \ldots < t_n = b,\ n \in \N \right\}.
\end{equation*}
\BLUE In this section we prescribe for every $t \in [0, T]$ a boundary datum $w(t) \in H^1(\ext{\Omega};\R^3) \cap KL(\ext{\Omega})$ and we assume the map $t\mapsto w(t)$ to be absolutely continuous from $[0, T]$ into $H^1(\ext{\Omega};\R^3)$.\BLACK

We now give the notion of the limiting quasistatic elasto-plastic evolution.

\begin{definition}
A \emph{two-scale quasistatic evolution} for the boundary datum $w$ is a function $t \mapsto (u(t), E(t), P(t))$ from $[0,T]$ into $KL(\ext{\Omega}) \times L^2(\ext{\Omega} \times \calY;\M^{3 \times 3}_{\sym}) \times \Mb(\ext{\Omega} \times \calY;\M^{3 \times 3}_{\dev})$ which satisfies the following conditions:
\begin{enumerate}[label=(qs\arabic*)$^{hom}_{\gamma}$]
    \item \label{hom-qs S} for every $t \in [0,T]$ we have $(u(t), E(t), P(t)) \in \calA^{hom}_{\gamma}(w(t))$ and
    \begin{equation*}
        \calQ^{hom}(E(t)) \leq \calQ^{hom}(H) + \calH^{hom}(\Pi-P(t)),
    \end{equation*}
    for every $(\upsilon,H,\Pi) \in \calA^{hom}_{\gamma}(w(t))$.
    \item \label{hom-qs E} the function $t \mapsto P(t)$ from $[0, T]$ into $\Mb(\ext{\Omega} \times \calY;\M^{3 \times 3}_{\dev})$ has bounded variation and for every $t \in [0, T]$
    \begin{equation*}
        \calQ^{hom}(E(t)) + \calD_{\calH^{hom}}(P; 0, t) = \calQ^{hom}(E(0)) 
        + \int_0^t \int_{\Omega \times \calY} \C(y) E(s) : E\dot{w}(s) \,dx dy ds.
    \end{equation*}
\end{enumerate}
\end{definition}

Recalling the definition of $h$-quasistatic evolution for the boundary datum $w(t)$ given in Definition \ref{h-quasistatic evolution}, we are in a position to formulate the main result of the paper.

\begin{theorem} \label{main result}
Let $t \mapsto w(t)$ be absolutely continuous from $[0,T]$ into $H^1(\ext{\Omega};\R^3) \cap KL(\ext{\Omega})$.  \MMM Assume \eqref{tensorassumption} and \eqref{coercivity of Q} and \BBB
 that there exists a sequence of triples $(u^h_0, e^h_0, p^h_0) \in \calA_h(w(0))$ such that
\begin{align}
    &u^h_0 \weakstar u_0 \quad \text{weakly* in $BD(\ext{\Omega})$}, \label{main result u^h_0 condition}\\
    &\Lambda_h e^h_0 \strongtwoscale E_0 \quad \text{two-scale strongly in $L^2(\ext{\Omega} \times \calY;\M^{3 \times 3}_{\sym})$}, \label{main result e^h_0 condition}\\
    &\Lambda_h p^h_0 \weakstartwoscale P_0 \quad \text{two-scale weakly* in $\Mb(\ext{\Omega} \times \calY;\M^{3 \times 3}_{\dev})$}, \label{main result p^h_0 condition}
\end{align}
for $(u_0,E_0,P_0) \in \calA^{hom}_{\gamma}(w(0))$.  
For every $h > 0$, let 
\begin{equation*}
    t \mapsto (u^h(t), e^h(t), p^h(t))
\end{equation*}
be a $h$-quasistatic evolution \MMM in the sense of Definition \ref{h-quasistatic evolution} \BBB for the boundary datum $w$ such that $u^h(0) = u^h_0$, $e^h(0) = e^h_0$, and $p^h(0) = p^h_0$. 
Then, there exists a two-scale quasistatic evolution
\begin{equation*}
    t \mapsto (u(t), E(t), P(t)) 
\end{equation*}
for the boundary datum $w$ such that $u(0) = u_0$,\, $E(0) = E_0$, and $P(0) = P_0$, and such that (up to subsequences)
for every $t \in [0,T]$
\begin{align}
    &u^h(t) \weakstar u(t) \quad \text{weakly* in $BD(\ext{\Omega})$}, \label{main result u^h(t)}\\
    &\Lambda_h e^h(t) \weaktwoscale E(t) \quad \text{two-scale weakly in $L^2(\ext{\Omega} \times \calY;\M^{3 \times 3}_{\sym})$}, \label{main result e^h(t)}\\
    &\Lambda_h p^h(t) \weakstartwoscale P(t) \quad \text{two-scale weakly* in $\Mb(\ext{\Omega} \times \calY;\M^{3 \times 3}_{\dev})$}. \label{main result p^h(t)}
\end{align}
\end{theorem}

\begin{proof}
The proof is subdivided into three steps, in the spirit of evolutionary $\Gamma$-convergence. 

\noindent{\bf Step 1: \em Compactness.}

We first prove that that there exists a constant $C$, depending only on the initial and boundary data, such that
\begin{equation} \label{boundness in time 1}
    \sup_{t \in [0,T]} \left\|\Lambda_h e^h(t)\right\|_{L^2(\BLUE\ext{\Omega}\BLACK;\M^{3 \times 3}_{\sym})} \leq C \;\text{ and }\; \calD_{\calH_h}(\Lambda_h p^h; 0, T) \leq C,
\end{equation}
for every $h>0$. 
Indeed, the energy balance of the $h$-quasistatic evolution \ref{h-qs E} and \eqref{coercivity of Q} imply
\begin{align*}
    &r_c \left\|\Lambda_h e^h(t)\right\|_{L^2(\BLUE\ext{\Omega}\BLACK;\M^{3 \times 3}_{\sym})} + \calD_{\calH_h}(\Lambda_h p^h; 0, t) \\
    &\leq R_c \left\|\Lambda_h e^h(0)\right\|_{L^2(\BLUE\ext{\Omega}\BLACK;\M^{3 \times 3}_{\sym})} + 2 R_c \sup_{t \in [0,T]} \left\|\Lambda_h e^h(t)\right\|_{L^2(\BLUE\ext{\Omega}\BLACK;\M^{3 \times 3}_{\sym})} \int_0^T \left\|E\dot{w}(s)\right\|_{L^2(\ext{\Omega};\M^{3 \times 3}_{\sym})} \,ds,
\end{align*}
where the last integral is well defined as $t \mapsto E\dot{w}(t)$ belongs to $L^1([0,T];L^2(\ext{\Omega};\M^{3 \times 3}_{\sym}))$. 
In view of the boundedness of $\Lambda_h e^h_0$ that is implied by \eqref{main result e^h_0 condition}, property \eqref{boundness in time 1} now follows by the Cauchy-Schwarz inequality.

From \eqref{boundness in time 1} and \eqref{coercivity of H_i}, we infer that
\begin{equation*}
    r_k \left\|\Lambda_h p^h(t) - \Lambda_h p^h_0\right\|_{\Mb(\BLUE\ext{\Omega}\BLACK;\M^{3 \times 3}_{\dev})} \leq \calH_h\left(\Lambda_h p^h(t) - \Lambda_h p^h_0\right) \leq \calD_{\calH_h}(\Lambda_h p^h; 0, t) \leq C,
\end{equation*}
for every $t \in [0,T]$, which together with \eqref{main result p^h_0 condition} implies 
\begin{equation} \label{boundness in time 2}
    \sup_{t \in [0,T]} \left\|\Lambda_h p^h(t)\right\|_{\Mb(\BLUE\ext{\Omega}\BLACK;\M^{3 \times 3}_{\dev})} \leq C.
\end{equation}

Next, we note that $\left\|\cdot\right\|_{L^1(\ext{\Omega} \setminus \closure{\Omega};\M^{3 \times 3}_{\sym})}$ is a continuous seminorm on $BD(\ext{\Omega})$ which is also a norm on the set of rigid motions. 
Then, using a variant of Poincar\'{e}-Korn's inequality (see \cite[Chapter II, Proposition 2.4]{Temam.1985}) and the fact that $(u^h(t), e^h(t), p^h(t)) \in \calA_h(w(t))$, we conclude that, for every $h > 0$ and $t \in [0,T]$,
\begin{align*}
    \left\|u^h(t)\right\|_{BD(\ext{\Omega})} &\leq C \left(\left\|u^h(t)\right\|_{L^1(\ext{\Omega} \setminus \closure{\Omega};\R^3)} + \left\|Eu^h(t)\right\|_{\Mb(\ext{\Omega};\M^{3 \times 3}_{\sym})}\right)\\
    &\leq C \left(\left\|w(t)\right\|_{L^1(\ext{\Omega} \setminus \closure{\Omega};\R^3)} + \left\|e^h(t)\right\|_{L^2(\ext{\Omega};\M^{3 \times 3}_{\sym})} + \left\|p^h(t)\right\|_{\Mb(\ext{\Omega};\M^{3 \times 3}_{\dev})}\right)\\
    &\leq C \left(\left\|w(t)\right\|_{L^2(\ext{\Omega};\R^3)} + \left\|\Lambda_h e^h(t)\right\|_{L^2(\ext{\Omega};\M^{3 \times 3}_{\sym})} + \left\|\Lambda_h p^h(t)\right\|_{\Mb(\ext{\Omega};\M^{3 \times 3}_{\dev})}\right).
\end{align*}
In view of the assumption on \MMM $w$ \BBB, from \eqref{boundness in time 2} and the former inequality in \eqref{boundness in time 1} it follows that the sequences $\{u^h(t)\}$ are bounded in $BD(\ext{\Omega})$ uniformly with respect to $t$.

Owing to \eqref{equivalence of variations}, we obtain that $\calD_{\calH_h}$ and $\calV$ are equivalent norms, which immediately implies
\begin{equation} \label{boundness in time 3}
    \calV(\Lambda_h p^h; 0, T) \leq C,
\end{equation}
for every $h>0$. 
Hence, by a generalized version of Helly's selection theorem (see \cite[Lemma 7.2]{DalMaso.DeSimone.Mora.2006}) \MMM and Remark \ref{transfertwoscale}\BBB, there exists a (not relabeled) subsequence, independent of $t$, and $P \in BV(0,T;\Mb(\ext{\Omega} \times \calY;\M^{3 \times 3}_{\dev}))$ such that
\begin{equation*}
    \Lambda_h p^h(t) \weakstartwoscale P(t) \quad \text{two-scale weakly* in $\Mb(\ext{\Omega} \times \calY;\M^{3 \times 3}_{\dev})$},
\end{equation*}
for every $t \in [0,T]$, and 
$\calV(P; 0, T) \leq C$.
By extracting a further subsequence (possibly depending on $t$),
\begin{align*}
    &u^{h_t}(t) \weakstar u(t) \quad \text{weakly* in $BD(\ext{\Omega})$},\\
    &\Lambda_{h_t} e^{h_t}(t) \weaktwoscale E(t) \quad \text{two-scale weakly in $L^2(\ext{\Omega} \times \calY;\M^{3 \times 3}_{\sym})$},
\end{align*}
for every $t \in [0,T]$. 
From \MMM \Cref{two-scale weak limit of scaled strains - 2x2 submatrix}\BBB, we conclude that $u(t) \in KL(\ext{\Omega})$ for every $t \in [0,T]$. 
According to \Cref{two-scale weak limit of scaled strains}, the above subsequence can be chosen so that there exists $\mu(t) \in \calXgamma{\ext{\omega}}$ for which
\begin{equation*}
    \Lambda_h Eu^{h_t}(t) \weakstartwoscale Eu(t) \otimes \calL^{2}_{y} + \widetilde{E}_{\gamma}\mu(t).
\end{equation*}
Since, $\Lambda_{h_t} Eu^{h_t}(t) = \Lambda_{h_t} e^{h_t}(t) + \Lambda_{h_t} p^{h_t}(t)$ in $\ext{\Omega}$ for every $h > 0$ and $t \in [0,T]$, we deduce that $(u(t),E(t),P(t)) \in \calA^{hom}_{\gamma}(w(t))$.

Consider now for every $t \in [0,T]$ the maps
\begin{equation*}
    \sigma^{h_t}(t) := \C\left(\tfrac{x'}{\epsh_t}\right) \Lambda_{h_t} e^{h_t}(t).
\end{equation*}
For a (not relabeled) subsequence, we have
\begin{equation} \label{main result sigma^h(t)}
   \sigma^{h_t}(t) \weaktwoscale \Sigma(t) \quad \text{two-scale weakly in $L^2(\ext{\Omega} \times \calY;\M^{3 \times 3}_{\sym})$},
\end{equation}
where $\Sigma(t) := \C(y) E(t)$. 
Since $\sigma^{h_t}(t) \in \calK_{h_t}$ for every $t \in [0,T]$, by \Cref{two-scale weak limit of admissible stress - regime gamma} we obtain that $\Sigma(t) \in \calK^{hom}_{\gamma}$ for every $t\in [0,T]$.

\noindent{\bf Step 2: \em Global stability.}

Since from Step 1 we have $(u(t),E(t),P(t)) \in \calA^{hom}_{\gamma}(w(t))$ with the associated $\mu(t) \in \calXgamma{\ext{\omega}}$, then for every $(\upsilon,H,\Pi) \in \calA^{hom}_{\gamma}(w(t))$ with the associated $\nu \in \calXgamma{\ext{\omega}}$ we have
\begin{equation*}
    (\upsilon-u(t),H-E(t),\Pi-P(t)) \in \calA^{hom}_{\gamma}(0).
\end{equation*}
From the inclusion $\C(y) E(t) \in \calK^{hom}_{\gamma}$, by \Cref{two-scale dissipation and plastic work inequality} we infer 
\begin{align*}
    \calH^{hom}(\Pi-P(t)) &\geq -\int_{\omega \times I \times \calY} \C(y) E(t) : (H-E(t)) \,dx dy\\
    &= \calQ^{hom}(E(t)) + \calQ^{hom}(H-E(t)) - \calQ^{hom}(H).
\end{align*}
Thus,
\begin{equation*}
    \calH^{hom}(\Pi-P(t)) + \calQ^{hom}(H) \geq \calQ^{hom}(E(t)) + \calQ^{hom}(H-E(t)) \geq \calQ^{hom}(E(t)),
\end{equation*}
hence we deduce \ref{hom-qs S}.

Now we can prove that limit functions $u(t)$ and $E(t)$ do not depend on the subsequence. 
Assume that $(\upsilon(t),H(t),P(t)) \in \calA^{hom}_{\gamma}(w(t))$ with the associated $\nu(t) \in \calXgamma{\ext{\omega}}$ also satisfy the global stability condition in the definition of the two-scale quasistatic evolution. 
By the strict convexity of $\calQ^{hom}$, we find
\begin{equation*}
    H(t) = E(t).
\end{equation*}
Then, by \eqref{admissible two-scale configurations - regime gamma}, 
\begin{align*}
    E\upsilon(t) \otimes \calL^{2}_{y} + \widetilde{E}_{\gamma}\nu(t) &= H(t) \,\calL^{3}_{x} \otimes \calL^{2}_{y} + P(t)\\
    &= E(t) \,\calL^{3}_{x} \otimes \calL^{2}_{y} + P(t)\\
    &= Eu(t) \otimes \calL^{2}_{y} + \widetilde{E}_{\gamma}\mu(t).
\end{align*}
Identifing $Eu(t)$ and $E\upsilon(t)$ with elements of $\Mb(\ext{\Omega};\M^{2 \times 2}_{\sym})$ and 
integrating over $\calY$, 
we obtain
\begin{equation*}
    E\upsilon(t) = Eu(t).
\end{equation*}
Using the variant of Poincar\'{e}-Korn inequality in Step 1, we infer that $\upsilon(t) = u(t)$ on $\ext{\Omega}$.

This implies that the whole sequences converges without need to extract further $t$-dependent subsequences, i.e.
\begin{align*}
    &u^{h}(t) \weakstar u(t) \quad \text{weakly* in $BD(\ext{\Omega})$},\\
    &\Lambda_h e^h(t) \weaktwoscale E(t) \quad \text{two-scale weakly in $L^2(\ext{\Omega} \times \calY;\M^{3 \times 3}_{\sym})$}.
\end{align*}

\noindent{\bf Step 3: \em Energy balance.}

In order to prove \ref{hom-qs E}, it is enough (by arguing as in, e.g. \cite[Theorem 4.7]{DalMaso.DeSimone.Mora.2006} and \cite[Theorem 2.7]{Francfort.Giacomini.2012}) to prove the energy inequality
\begin{align} \label{step 3 inequality}
\begin{split}
    &\calQ^{hom}(E(t)) + \calD_{\calH^{hom}}(P; 0, t) \\
    &\leq \calQ^{hom}(E(0)) 
    + \int_0^t \int_{\Omega \times \calY} \C(y) E(s) : E\dot{w}(s) \,dx dy ds.
\end{split}
\end{align}

For a fixed $t \in [0,T]$, consider a subdivision $0 = t_1 < t_2 < \ldots < t_n = t$ of $[0,t]$. 
In view of the lower semicontinuity of $\calQ^{hom}$ and $\calH^{hom}$ (see \eqref{lower semicontinuity Q} and \eqref{lower semicontinuity H}), from \ref{h-qs E} we have 
\begin{align*}
    &\calQ^{hom}(E(t)) + \sum_{i = 1}^{n} \calH^{hom}\left(P(t_{i+1}) - P(t_i)\right)\\
    &\leq \liminf\limits_{h}\left( \calQ_h(\Lambda_h e^h(t)) + \sum_{i = 1}^{n} \calH_h\left(\Lambda_h p^h(t_{i+1}) - \Lambda_h p^h(t_i)\right) \right)\\
    &\leq \liminf\limits_{h}\left( \calQ_h(\Lambda_h e^h(t)) + \calD_{\calH_h}(\Lambda_h p^h; 0, t) \right)\\
    &= \liminf\limits_{h}\left( \calQ_h(\Lambda_h e^h(0)) 
    + \int_0^t \int_{\Omega} \C\left(\tfrac{x'}{\epsh}\right) \Lambda_h e^h(s) : E\dot{w}(s) \,dx ds \right).
\end{align*}
By the strong convergence assumed in \eqref{main result e^h_0 condition} and \eqref{main result sigma^h(t)}, owing to the Lebesgue's dominated convergence theorem we obtain
\begin{align*}
    &\lim\limits_{h}\left( \calQ_h(\Lambda_h e^h(0)) 
    + \int_0^t \int_{\Omega} \C\left(\tfrac{x'}{\epsh}\right) \Lambda_h e^h(s) : E\dot{w}(s) \,dx ds \right)\\
    &= \calQ^{hom}(E(0)) 
    + \int_0^t \int_{\Omega \times \calY} \C(y) \Lambda_h E(s) : E\dot{w}(s) \,dx dy ds.
\end{align*}
Hence, we have
\begin{align*}
    &\calQ^{hom}(E(t)) + \sum_{i = 1}^{n} \calH^{hom}\left(P(t_{i+1}) - P(t_i)\right) \\
    &\leq \calQ^{hom}(E(0)) 
    + \int_0^t \int_{\Omega \times \calY} \C(y) \Lambda_h E(s) : E\dot{w}(s) \,dx dy ds
\end{align*}
Taking the supremum over all partitions of $[0,t]$ yields \eqref{step 3 inequality}, which concludes the proof.
\end{proof}
\begin{remark}
We point out that as a Corollary of Theorem \ref{main result} and of the fact that the limiting model satisfies an energy equality, we find that strong two-scale convergence in the $L^2$-topology of the scaled initial elastic strains and weak two-scale convergence in measure of the scaled initial plastic strains are enough to guarantee the strong two-scale convergence of the rescaled elastic strains in the $L^2$-topology to the effective one, as well as the convergence of rescaled dissipations to the limiting one.
\end{remark}

\section*{Acknowledgements}
M.~Bu\v{z}an\v{c}i\'{c} and I.~Vel\v{c}i\'{c} were supported by the Croatian Science Foundation under Grant Agreement no. IP-2018-01-8904 (Homdirestroptcm). 
The research of E. Davoli was supported by the Austrian Science Fund (FWF) projects F65, V 662, Y1292, and I 4052. All authors are thankful for the support from the OeAD-WTZ project HR 08/2020.
\printbibliography
\end{document}